\documentclass{amsart}
\usepackage{macros}
\standardsettings
\colorcommentstrue

\newcommand{\bS}{\mathbb{S}}
\newcommand{\emb}{\Phi}
\newcommand{\Qrank}{{p_\Q}}
\newcommand{\Rrank}{{p_\R}}
\newcommand{\Krank}{{p_\K}}
\renewcommand{\matrix}{\phi}
\newcommand{\cee}{c}
\newcommand{\mindist}{\delta}
\newcommand{\numberq}{T}
\renewcommand{\height}{H}
\renewcommand{\ss}{\mathbf s}
\renewcommand{\prim}{\mathrm{pr}}

\newcommand{\aff}{{\mathrm{aff}}}
\DeclareMathOperator{\Span}{Span}
\newcommand{\subgp}{\subset}

\newcommand{\form}{Q}
\newcommand{\nform}{R}
\renewcommand{\O}{\stopcompiling}
\newcommand{\orth}{\mathrm O}
\renewcommand{\dist}{\operatorname{dist}}
\renewcommand{\A}{\mathrm A}

\newcommand\eq[2]{
\begin{equation}
\label{eq:#1}
{#2}
\end{equation}
}
\newcommand{\equ}[1]{\eqref{eq:#1}}

\renewcommand{\HD}{\dim}

\theoremstyle{theorem}
\maketheorem{affinecorollary}{Affine Corollary}
\theoremstyle{definition}
\maketheorem{affinecorollaryD}{Affine Corollary}

\draftfalse

\begin{document}
\title{Intrinsic Diophantine approximation on quadric hypersurfaces 
}

\authorlior\authordmitry\authorkeith\authordavid


\begin{abstract}
We consider the question of how well points in a quadric hypersurface $M\subset\mathbb R^d$ can be approximated by rational points of $\mathbb Q^d\cap M$. This contrasts with the more common setup of approximating points in a manifold by all rational points in $\mathbb Q^d$. We provide complete answers to major questions of Diophantine approximation in this context. Of particular interest are the impact of the real and rational ranks of the defining quadratic form, quantities whose roles in Diophantine approximation have never been previously elucidated. Our methods include a correspondence between the intrinsic Diophantine approximation theory on a rational quadric hypersurface and the dynamics of the group of projective transformations which preserve that hypersurface, similar to earlier results in the non-intrinsic setting due to Dani ('86) and Kleinbock--Margulis ('99).
\end{abstract}

\date{December 2020}
\maketitle

\tableofcontents

\section{Introduction and motivation}\label{introduction}

Classical theorems in Diophantine approximation theory address questions regarding the way points $\xx\in\R^d$ are approximated by rational points, considering the trade-off between the height of the rational point -- the size of its denominator -- and its distance to $\xx$; see \cite{Cassels, Schmidt2} for a general introduction. Often $\xx$ is assumed to lie on a certain subset of $\R^d$, for example a smooth manifold $M$, leading to \emph{Diophantine approximation on manifolds\/}. This area of research has experienced rapid progress during the last two decades, owing much of it to methods coming from flows on homogeneous spaces.

It was observed in \cite{DickinsonDodson, Drutu, BDL} that all sufficiently good rational approximants to points on certain rational varieties must in fact be \emph{intrinsic} -- that is, they are rational points lying on the variety itself. These results, in part, have motivated a new field of \emph{intrinsic approximation}, which examines the degree to which points on a manifold or variety can be approximated by rationals lying on that same subset. Questions about the quality of these approximations were raised already by 
Lang \cite{Lang_diophantine} and 
Mahler \cite{Mahler}. 
Following some recent results on  quadric hypersurfaces \cite{Schmutz, GGN1, GGN2} and a comprehensive treatment of Diophantine approximation on spheres \cite{KleinbockMerrill},  this paper seeks to fully explore the topic of intrinsic approximation on quadrics.
One of the most novel and important aspects of our work is an elucidation of the role of the $\Q$-rank and the $\R$-rank of the defining quadratic form (see Definition \ref{definitionrank}). 
It turns out there are qualitative differences between the intrinsic approximation theories of forms with different rank pairs, highlighting the importance of rank, rather than the dimension of the hypersurface. In particular, we will see below that our Dirichlet-type theorem, Theorem \ref{theoremdirichletquadratic}, is \emph{independent} of the dimension $d$, but changes depending on whether the $\Q$-rank and $\R$-rank are equal or different. We remark that \cite{KleinbockMerrill} considers only the case where both ranks equal 1, therefore the dependence on the ranks is not explored there, and significant new ideas have had to be developed in the present paper.


\begin{convention}
The symbols $\lesssim$, $\gtrsim$, and $\asymp$ will denote asymptotics; a subscript of $\plus$ indicates that the asymptotic is additive, and a subscript of $\times$ indicates that it is multiplicative. For example, $A\lesssim_{\times,K} B$ means that there exists a constant $C > 0$ (the \emph{implied constant}), depending only on $K$, such that $A\leq C B$. $A\lesssim_{\plus,\times}B$ means that there exist constants $C_1,C_2 > 0$ so that $A\leq C_1 B + C_2$. In general, dependence of the implied constant(s) on universal objects such as the manifold $M$ will be omitted from the notation.
\end{convention}

\begin{convention}
\label{convention2}
For any $\cee\geq 0$ we let
\[
\psi_\cee(q) := \frac{1}{q^\cee}\cdot
\]
\end{convention}

\begin{convention}
The symbol $\triangleleft$ will be used to indicate the end of a nested proof.
\end{convention}

\smallskip
\noindent\bf Glossary of Notation. \rm For the reader's convenience we summarize a list of notations and
terminology in the order that they appear in the sequel.

\begin{itemize}
\item $M$ \dotfill a complete metric space (Section \ref{introduction})
\item $\QQ$ \dotfill a countable subset of $M$ (Section \ref{introduction})
\item $H$ \dotfill a height function (Section \ref{introduction})
\item $\BA({\psi,M,\QQ, H})$ \dotfill the set of badly approximable points (Section \ref{introduction})
\item $\WA({\psi,M,\QQ, H})$ \dotfill the set of well approximable points (Section \ref{introduction})
\item $\A({\psi,M,\QQ, H})$ \dotfill the set of $\psi$-approximable points (Section \ref{introduction})
\item $H_\std$ \dotfill the standard height on projective space (Section \ref{subsectionquadratic})
\item $\form$ \dotfill a quadratic form on $\R^{d+1}$ (Section \ref{subsectionquadratic})
\item $L_\form$ \dotfill the light cone of $\form$ (Section \ref{subsectionquadratic})
\item $M_\form$ \dotfill a nonsingular rational quadric hypersurface (Section \ref{subsectionquadratic})
\item $\Rrank$ \dotfill the real rank of $\form$  (Section \ref{subsectionquadratic})
\item $\Qrank$ \dotfill the rational rank of $\form$  (Section \ref{subsectionquadratic})
\item $\form_\aff$ \dotfill a quadratic polynomial with integer coefficients on $\R^{d}$ (Section \ref{subsectionquadratic})
\item $M_{\form_\aff}$ \dotfill the nonsingular rational quadric hypersurface associated to $\form_\aff$ (Section \ref{subsectionquadratic})
\item $\A_{M_\form}(\psi)$ \dotfill the set of $\psi$-approximable points on $M_\form$  (Section \ref{subsectionquadratic})
\item $\BA_{M_\form}$ \dotfill the set of badly approximable points on $M_\form$  (Section \ref{subsectionquadratic})
\item $\form_0$\dotfill the exceptional quadratic form  (Section \ref{subsectionquadratic})
\item $B_\form$ \dotfill the symmetric, bilinear form associated to $\form$ (Section \ref{sectionpreliminaries})
\item $\LL_m$\dotfill  $\sum_{i = 0}^{m - 1}\R\ee_i$ (Section \ref{sectionpreliminaries})
\item $\tilde\form$ \dotfill the remainder of the form $\form$ after normalizing (Section \ref{sectionpreliminaries})
\item $\widehat\matrix$ \dotfill the reverse  of the matrix $\matrix$  (Section \ref{sectionpreliminaries})
\item $g_\matrix$\dotfill  $\left[\begin{array}{ccc}
\matrix &&\\
& I_{d + 1 - 2m} &\\
&& \widehat\matrix\end{array}\right]
$ (Section \ref{sectionpreliminaries})

{\item
$g_\tt$ \dotfill $\ g_{\diag(e^{-t_0},\ldots,e^{-t_{m - 1}})}$ (Section \ref{sectionpreliminaries})}
\item
$g_t$ \dotfill $\left[\begin{array}{ccc}
e^{-t} &&\\
& I_{d - 1} &\\
&& e^t
\end{array}\right]$ (Section \ref{sectionpreliminaries})
{\item $\mindist(\Lambda)$ \dotfill $\displaystyle\min_{\pp\in\Lambda\smallsetminus\{\0\}} \|\pp\|$
(Section \ref{sectionpreliminaries})
\item $\mindist_\form(\Lambda)$ \dotfill $\displaystyle\min_{\pp\in\Lambda\cap L_\form\smallsetminus\{\0\}} \|\pp\|$
(Section \ref{sectionpreliminaries})}
\item $\orth(\form)$ \dotfill  $\{g\in\SL_{d + 1}^\pm(\R): \form\circ g = \nform\}$ (Section\ref{sectionpreliminaries})
\item $\Omega_\form$ \dotfill the space of $\form$-arithmetic lattices (Section\ref{sectionpreliminaries})
\item $\Omega_d$ \dotfill the space of all lattices in $\R^{d+1}$ (Section\ref{sectionpreliminaries})
\item $\orth(\form;\Lambda)$ \dotfill the stabilizer of $\Lambda$ under the action of $\orth(\form)$ (Section\ref{sectionpreliminaries})
\item $\Omega_{\form,\Lambda}$ \dotfill the homogeneous space $\orth(\form)/\orth(\form;\Lambda)$ (Section\ref{sectionpreliminaries})
{\item $\pi_1,\pi_2$ \dotfill projections $\orth(\nform)\to M_\form$ and $\orth(\nform)\to \Omega_{\nform,\Lambda_*}$ (Section \ref{sectioncorrespondence})}
\item $\Lambda_\prim$ \dotfill the set of primitive vectors of $\Lambda$ (Section \ref{sectioncorrespondence})
{\item $\mu_\nform,\mu_{\nform,\Lambda_*}$ \dotfill Haar measures on $\orth(\nform)$ and $\Omega_{\nform,\Lambda_*}$ (Section \ref{sectioncorrespondence})
\item $\Codiam(\Gamma)$ \dotfill  the diameter of the quotient space $\Span(\Gamma)/\Gamma$ (Section \ref{sectionproofdirichlet})}
\item $N_M(\numberq)$\dotfill $\#\big\{[\rr]\in \P_\Q^d\cap M: H_\std([\rr])\leq \numberq\big\}$ (Section \ref{sectionkhinchinquadratic})
\item $S_{\Delta,z}$ \dotfill $\{x\in X: \Delta(x) \geq z\}$ (Section \ref{sectionkhinchinproof})
{\item $\Phi_\Delta(z)$ \dotfill $\mu(S_{\Delta,z})$, the tail distribution function of $\Delta$ (Section \ref{sectionkhinchinproof})}

\item $\varphi^{(C)}(x)$\dotfill $\max_{\dist_X{(x' , x)}\leq C}\varphi(x')$ (Section \ref{sectionreductiontheory})
\item $\varphi_{(C)}(x)$\dotfill $ \min_{\dist_X{(x' , x)}\leq C}\varphi(x')$ (Section \ref{sectionreductiontheory})
\item $P$ \dotfill a parabolic subgroup of $G$ (Section \ref{sectionreductiontheory})
\item $\rho_P$\dotfill the modular function of $P$ (Section \ref{sectionreductiontheory})
\item $A$ \dotfill a maximal $\Q$-split torus  (Section \ref{sectionreductiontheory})
\item $\rho$ \dotfill the sum of the positive roots of $A$, counted with multiplicity  (Section \ref{sectionreductiontheory})
\end{itemize}

\medskip
\textbf{Acknowledgements.} The first-named author was supported in part by the Simons Foundation grant \#245708. The second-named author was supported in part by the NSF grants DMS-1101320 and  DMS-1600814. The fourth-named author was supported in part by the EPSRC Programme Grant EP/J018260/1. The authors would   like to thank Victor Beresnevich, Cornelia Dru\c tu, and Sanju Velani for helpful discussions, {and an anonymous referee for useful comments}.

\subsection{General terminology and basic problems in metric Diophantine approximation}
In order to review some known facts and state our theorems, let us first introduce basic notations which we will follow throughout the paper (some of it has been introduced in a different context in \cite{FishmanSimmons5}).

\begin{definition}
\label{definitiontriple}
By a \emph{Diophantine triple} we will mean a triple $ (M,\QQ, H)$, where $M$ is a closed subset of a complete metric space $(X,\dist)$, $\QQ$ is a countable subset of $X$ whose closure contains $M$, and $H$ is a function from $\QQ$ to $(0, \infty)$. 
\end{definition}

\begin{definition}
\label{definitiontdirichletfunction}
Say that a non-increasing\Footnote{The approximating functions $\psi$ will be assumed to be non-increasing throughout the paper.} function $\psi:(0,\infty)\to(0,\infty)$ is a \emph{Dirichlet function} for $ (M,\QQ, H)$ if for every $\xx\in M$ there exists $C_{\xx} > 0$ and a sequence $(\rr_n)_1^\infty$ in $\QQ$ such that
\begin{equation}
\label{dirichlet}
\rr_n\tendsto n \xx \text{ and } \dist(\rr_n,\xx) \leq C_{\xx} \psi\big( \height(\rr_n)\big).
\end{equation}
If $C_{\xx}$ can be chosen independent of $\xx$, then we call $\psi$ \emph{uniformly Dirichlet}.
\end{definition}


When $\psi$ is a Dirichlet function, it is often important to understand whether a faster decaying function can also be Dirichlet. We formalize this thought in the next definition:

\begin{definition}
\label{definitiontoptimality} 
A Dirichlet function $\psi$ is \emph{optimal} for $ (M,\QQ, H)$ if there is no function $\varphi$ which is Dirichlet for $ (M,\QQ, H)$ and satisfies
$
\frac{\varphi(x)}{\psi(x)}\to 0\text{ as } x\to\infty
$.
\end{definition}

It turns out that the optimality of $\psi$ is under some fairly general assumptions equivalent to the existence of so-called \emph{badly approximable} points. This notion deserves a special definition:

\begin{definition}
\label{definitionbadlyapproximable}
If $ (M,\QQ, H)$ is a Diophantine triple and if $\psi:(0,\infty)\to (0,\infty)$, then a point $\xx\in M$ is said to be \emph{badly approximable} with respect to $\psi$ if there exists $\epsilon > 0$ such that for all $\rr\in\QQ$,
\[
\dist(\rr,\xx) \geq \epsilon\psi\big( \height(\rr)\big).
\]
The set of such points will be denoted $\BA({\psi,M,\QQ, H})$, and its complement will be denoted $\WA({\psi,M,\QQ, H})$ (the set of \emph{well approximable} points).
\end{definition}


If $\BA({\psi,M,\QQ, H}) \neq \emptyset$, then it is easy to see that $\psi$ is an optimal Dirichlet function for $(M,\QQ, H)$.\Footnote{Cf.\ \cite[Theorem 2.6]{FishmanSimmons5} where this is stated under the assumption that $M = X$; one can check that the latter assumption is not necessary for the argument. Furthermore, the converse is 
true assuming the $\sigma$-compactness of $M$, see \cite[Proposition 2.7]{FishmanSimmons5}. 
} Note also that $\QQ\cap M$ is always contained in $\WA({\psi,M,\QQ, H})$.

\begin{definition}
\label{definitionapproximable}
Also, we will let
\[
\begin{split}
\A(\psi,M,\QQ, H)
&:=\big\{\xx\in M : \exists\,\infty \text{ many }\rr\in\QQ\text{ with }\dist(\rr,\xx) \leq \psi\big( \height(\rr)\big)\big\}\\
&= \limsup_{\rr \in \QQ} \Big(B\left(\rr, \psi\big(H(\rr)\big)\right)\cap M\Big)
\end{split}
\]
be the set of \emph{$\psi$-approximable} points. Note that
\[
\WA(\psi,M,\QQ, H) = (\QQ\cap M)\cup\bigcap_{\epsilon > 0} \A(\epsilon\psi,M,\QQ,H).
\]
\end{definition}
We can now list a few basic general problems one can pose, given a Diophantine triple $ (M,\QQ, H)$: 

\begin{itemize}
\item[1.] Find a Dirichlet function for $ (M,\QQ, H)$. Even better -- find an optimal one; determine whether or not it is uniformly Dirichlet.
\item[2.] Find a function $\psi$ such that $\BA({\psi,M,\QQ, H}) \neq \emptyset$. Even better -- do it for a Dirichlet function, thus proving it to be optimal. In the latter case determine how big is the set $\BA({\psi,M,\QQ, H})$, e.g. in terms of its Hausdorff dimension.
\item[3.] Given a function $\psi$ and a measure on $M$, what is the measure of the set $\A({\psi,M,\QQ, H})$? This measure could be a Riemannian volume on $M$ if the latter is a manifold, or, more generally, the Hausdorff measure relative to some dimension function. A special case of the last question is a determination of the Hausdorff dimension of $\A({\psi,M,\QQ, H})$.
\end{itemize}

Note that, since $\A({\psi,M,\QQ, H})$ is a limsup set, the easy direction of the Borel--Cantelli lemma shows that for any measure $\mu$ on $M$, if the series
\eq{khinchin sum}
{\sum_{\rr\in\QQ\,\cap\, U}\mu\Big(B\left(\rr, \psi\big(H(\rr)\big)\right)\cap M\Big)
}
converges whenever $U$ is a bounded subset of $X$, then one has $\mu\big(\A({\psi,M,\QQ, H})\big) = 0$. The hope is that for ``nice" measures the (much harder) complementary divergence case can be established. Also in general it is not clear how to explicitly decide for which functions $\psi$ the sum \equ{khinchin sum} converges or diverges; for that one often needs extra information concerning the number of points of $\QQ$ satisfying a given height bound.

\subsection{Diophantine approximation in $\R^d$}\label{rd}
In the classical Diophantine approximation setup one has $X = M = \R^d$, $\QQ = \Q^d$, and 
\eq{standard height}{H(\rr) = H_\std(\rr) := q\text{ where $\rr = \pp/q$ is written in reduced form}}
\label{standardHight}(this will be referred to as the \emph{standard} height).

Dirichlet's theorem asserts that for all $\xx\in\R^d$ and $\numberq\geq 1$ there exists $\pp/q\in\Q^d$ with $q\leq\numberq$ satisfying
\eq{standard dt}{
\dist\left(\frac\pp q,\xx\right) \leq \frac{C}{q \numberq^{1/d}}\,,
}
where $C > 0$ is a constant depending on the choice of the norm on $\R^d$. A corollary is that
\begin{equation}
\label{unifdiraffine}
\text{$\psi_{1 + 1/d}$ is uniformly Dirichlet for $(\R^d, \Q^d,H_\std)$}
\end{equation}
(see Convention \ref{convention2}). Note that when the distance is given by the supremum norm on $\R^d$ one can take $C = 1$ in \equ{standard dt}, and thus $C_\xx\equiv 1$ in \eqref{dirichlet}. (It is clear that the property of $\psi$ being Dirichlet or uniformly Dirichlet does not depend on the choice of the norm.)

On the other hand, it is well-known that for all $d$, the set
\[
\BA_d := \BA({\psi_{1 + 1/d},\R^d, \Q^d,H_\std})
\]
of badly approximable vectors in $\R^d$ is nonempty (see e.g.~\cite{Perron, Schmidt2}), implying the optimality of $\psi_{1 + 1/d}$ as a Dirichlet function for $(\R^d, \Q^d,H)$. Indeed, 
Schmidt \cite{Schmidt2} showed that
\begin{equation}
\label{jarnikschmidt}
\text{$\BA_d$ has full Hausdorff dimension in $\R^d$},
\end{equation}
generalizing a result of 
Jarn\'ik \cite{Jarnik1}, who proved the case $d = 1$ of \eqref{jarnikschmidt}. 
We shall refer to \eqref{jarnikschmidt} as the \emph{Jarn\'ik--Schmidt theorem}. Note that together, Dirichlet's theorem and the Jarn\'ik--Schmidt theorem solve problems 1 and 2 above for the case of the Diophantine triple $(\R^d,\Q^d,H_\std)$.

Resolving problem 3 gives rise to theorems of Khintchine and of Jarn\'ik--Besicovitch. For convenience let us denote $\A({\psi,\R^d, \Q^d,H_\std})$ by $\A_d(\psi)$. If $\lambda$ is Lebesgue measure on $\R^d$, it was proven by 
Khintchine \cite{Khinchin2} that, if $\psi$ is non-increasing\Footnote{The monotonicity assumption is not needed if $d>1$, see \cite{Gallagher}.},
 $\A_d(\psi)$ is either null or conull depending on whether the series 
${\sum_{q = 1}^\infty q^{d - 1}\psi(q)^d}$
converges or diverges.
More generally, for $0< s < d$  one can replace $\lambda$ with $\HH^s$, the $s$-dimensional Hausdorff measure, and get the Jarn\'ik--Besicovitch theorem \cite{Jarnik2, Besicovitch}: $\HH^s\big(\A_d(\psi)\big)$ is either $0$ or $\infty$ depending on whether the series
${\sum_{q = 1}^\infty q^{d - 1} \psi(q)^s}$
converges or diverges. 
\section{Main results}\label{subsectionquadratic}

\smallskip
\begin{convention}
Throughout the paper, propositions which are proven later in the paper will be numbered according to the section they are proven in.
\end{convention}

\smallskip
We now consider the main setup of the paper, namely that of intrinsic approximation. One way to do it is to take $X = \R^d$, choose a 
submanifold $M$ of $\R^d$, and let $\QQ = \Q^d\cap M$ and $H = H_\std$ as in \equ{standard height}. However we have chosen a different approach: state and prove the main results of the paper for submanifolds of projective spaces. This way in most cases statements of results and their proofs become more natural and transparent, see Remark \ref{remarkprojective} below. 

Let $\P_\R^d$ denote the $d$-dimensional real projective space, and let $\pi:\R^{d + 1}\smallsetminus\{\0\}\to \P_\R^d$ be the quotient map $\pi(\xx) := [\xx]$, so that $[t\xx] = [\xx]$. The distance on $\P_\R^d$ will be given by the formula $\dist([\xx],[\yy]) =\min(\|\yy - \xx\|,\|\yy + \xx\|)$ ($\|\xx\| = \|\yy\| = 1$). For a subset $S$ of $\R^{d + 1}$, we let $[S] = \pi(S\smallsetminus\{\0\})$. With some abuse of notation, let us define the \emph{standard height function} $H_\std:\P_\Q^d\to\N$ by the formula 
$${\begin{aligned}H_\std([\pp]) = \|\pp\|, \text{ where $\pp$ is the unique (up to a sign)}\\ \text{ primitive integer representative of }[\pp].\qquad \end{aligned}}$$
Here and elsewhere $\|\cdot\|$ represents the max norm.
\begin{remark}
\label{remarkprojective}
To understand the difference between results for affine and projective spaces, note that if $\iota_d:\R^d\to\P_\R^d$ is given by the formula $\iota_d(\xx) = [(1,\xx)]$ and if $B\subset\R^d$ is a bounded set, then 
$\iota_d|_B$ is bi-Lipschitz and
\begin{equation}
\label{Hstdcomparison}
H_\std\big(\iota_d(\rr)\big) \asymp_{\times, B} H(\rr) \all\, \rr\in\Q^d\cap B.
\end{equation}
In particular, 
the Diophantine triples $T_{\aff} := (M,\Q^d\cap M,H_\std)$ and $T_{\mathrm{proj}} := \big(\iota_d(M),\P_\Q^d\cap\iota_d(M),H_\std\big)$ are ``locally isomorphic''. However, both the bi-Lipschitz constant and the implied constant of \eqref{Hstdcomparison} depend on the chosen bounded set $B$. Thus concepts which are robust under point-dependent multiplicative constants will not be affected by the transformation. For example, whether or not a function is Dirichlet will be the same for the triples $T_{\aff}$ and $T_{\mathrm{proj}}$, but it is conceivable that a function could be uniformly Dirichlet for the triple $T_{\mathrm{proj}}$ but not for the triple $T_{\aff}$.

Because of this difference, it is perhaps worthwhile to give a justification for why we are stating our results in projective space rather than affinely. The simplest answer to this question is that the projective statements are closest to how the results are actually proven. Moreover, in those cases where projective statements cannot be reformulated as affine statements, we feel it is important to keep the full strength of the projective theorem. To give a simple example, consider the classical Dirichlet's theorem. By examining its proof, we can deduce that
\begin{equation}
\label{unifdirprojective}
\psi_{1 + 1/d}\text{ is uniformly Dirichlet for }(\P_\R^d, \P_\Q^d,H_\std).
\end{equation}
This result is stronger than the classical \eqref{unifdiraffine}, in the sense that simply translating \eqref{unifdiraffine} to projective space along the lines indicated above does not yield \eqref{unifdirprojective}, while translating \eqref{unifdirprojective} to affine space yields \eqref{unifdiraffine} at least on the unit cube $[0,1]^d$, and applying translations recovers the full force of \eqref{unifdiraffine}.

To guide the reader, we have included Affine Corollaries after most of the main results. Each Affine Corollary can be deduced from its corresponding result together with Remark \ref{remarkprojective}. We omit those Affine Corollaries which would merely be restatements of the theorems with $\P_\R^d$ replaced by $\R^d$.
\end{remark}

\smallskip


In the following theorems we fix $d\geq 2$  
and let $\form$ be a nonsingular (see Definition \ref{definitionnonsingular}) quadratic form on $\R^{d + 1}$ with integer coefficients. (Cf.\ Remark \ref{remarkassumptionsquadratic} for a discussion of the singular case.)
Denote by 
\begin{equation}
\label{lightcone}
L_\form := \{\xx\in\R^{d + 1}: \form(\xx) = 0\}
\end{equation}
the \emph{light cone} of $\form$ and let $M_\form = [L_\form]$.  Manifolds $M_\form$ of this form are called \emph{nonsingular rational quadric hypersurfaces}.

We will denote by $\Rrank$ the \emph{$\R$-rank} of $\form$, defined as the dimension of any maximal totally isotropic (with respect to $\form$) subspace of $\R^{d + 1}$. Similarly, $\Qrank$ will stand for the \emph{$\Q$-rank} of $\form$, i.e.\ the dimension of any maximal totally isotropic rational subspace of $\R^{d + 1}$. Clearly $\Rrank \geq \Qrank$; see Section \ref{isotropic} for more details. To avoid trivialities, in our theorems we will make the standing assumption that $\Qrank\ge 1$, or, equivalently, that
\begin{equation}
\label{haverationalpoints}
\P_\Q^d\cap M_\form \neq \emptyset.
\end{equation}
Note that Meyer's theorem states that \eqref{haverationalpoints} is satisfied as soon as $d\geq 4$ and $M_\form\neq\emptyset$. Moreover, if $d = 2$ or $3$, the Hasse--Minkowski theorem (e.g.\  \cite[Theorem 1 on p.61]{BorevichShafarevich}) allows one to determine computationally whether \eqref{haverationalpoints} is satisfied for any given quadratic form $\form$; cf.\ \cite[Chapter 1, \67]{BorevichShafarevich}, in particular the remarks on the top of page 62.

For the affine corollaries to our theorems, we consider a quadratic polynomial $\form_\aff:\R^d\to\R$ with integer coefficients, and we let $\form:\R^{d + 1}\to\R$ be the projectivization of $\form_\aff$, that is, the unique homogeneous quadratic polynomial (i.e.\ quadratic form) $\form$ on $\R^{d + 1}$ such that $\form(1,\xx) = \form_\aff(\xx)$ for all $\xx\in\R^d$. Then $M_{\form_\aff}$, the zero set of $\form_\aff$, is equal to $\iota_d^{-1}(M_\form)$. We call $M_{\form_\aff}$ a nonsingular rational quadric hypersurface whenever $M_\form$ is. Note that it may be the case that $M_\form$ is singular due to ``singularities at infinity'' rather than singularities at finite points; in this case, we still consider the hypersurface $M_{\form_\aff}$ to be singular despite its having no ``singular points''.

The problem of intrinsic approximation on $M_\form$ was implicitly considered by 
Dru\c tu in \cite{Drutu} where the Hausdorff dimension of sets $\A_{M_\form}(\psi)$ was computed. (Dru\c tu actually studied ambient approximation on $M_\form$, and, generalizing an earlier result of 
Dickinson and 
Dodson \cite[Lemma 1]{DickinsonDodson}, showed that it reduces to intrinsic approximation if $\psi$ is assumed to decay fast enough.) The case $\form(\xx) = x_1^2 + \dots +x_d^2 - x_0^2$ was recently considered in \cite{KleinbockMerrill}.\Footnote{\cite{KleinbockMerrill} is written in the affine setup; specifically, the manifold $\bS^{d-1}\subset\R^d$ is discussed. Since this set is compact, Remark \ref{remarkprojective} gives an exact correspondence for Diophantine results in $\bS^{d-1}$ and those in $\iota_d(\bS^{d-1}) = M_\form$.} One of the theorems from the latter paper asserts\Footnote{Moshchevitin \cite{Moshchevitin2} has recently provided an elementary proof of this assertion for the case $M_{\form_\aff} =\bS^{2}$. His proof gives an explicit value for the constant $C$ appearing in \eqref{strongdirichletintro}.} that there exists $C > 0$ (possibly depending on $d$) such that for all $[\xx]\in M_\form$ and for all $\numberq\geq\numberq_0$ there exists $[\rr]\in \P_\Q^d\cap M_\form$ with
 \begin{equation}
 \label{strongdirichletintro}
H_\std([\rr]) \leq \numberq\text{ and }
\dist([\rr],[\xx]) \leq \frac{C}{\sqrt{H_\std([\rr]) \numberq}}\,\cdot
\end{equation}
In particular, it follows that $\psi_1$ is uniformly Dirichlet for intrinsic approximation on $M_\form$. It was also shown in \cite{KleinbockMerrill} that:

\begin{itemize}
\item[(i)] $\psi_1$ is optimal -- moreover, $\BA_{M_\form}(\psi_1)$ has full Hausdorff dimension;
\item[(ii)] for any 
$\psi: \N \to (0,\infty)$ such that 
$${\text{the function }q\mapsto q\psi(q)\text{ is nonincreasing,}}$$
the Lebesgue measure of $\A_{M_\form}(\psi)$ is full (resp.\ zero) iff the sum $\sum_{q = 1}^\infty q^{d - 2} \psi(q)^{d-1}$ diverges (resp.\ converges).
\end{itemize}

The last statement was also shown to imply, via the Mass Transference Principle of Beresnevich and Velani \cite[Theorem 2]{BeresnevichVelani}, a similar statement for Hausdorff measures.

%

In the present paper we generalize all the aforementioned results to the case of arbitrary quadric hypersurfaces. 

\medskip

\begin{reptheorem}{theoremdirichletquadratic}[Dirichlet-type theorem for quadric hypersurfaces]
Let $M_\form\subset\P_\R^d$ be a nonsingular rational quadric hypersurface 
with $\Qrank\ge 1$. Then
\begin{itemize}
\item[(i)] $\psi_1$ is Dirichlet for intrinsic approximation on $M_\form$.
\item[(ii)] $\psi_1$ is uniformly Dirichlet if and only if $\Qrank = \Rrank$.
\item[(iii)] The following are equivalent:
\begin{itemize}
\item[(A)] $\Qrank = \Rrank = 1$.
\item[(B)] (``Strong Dirichlet'') There exist $C,\numberq_0 > 0$ such that for all $[\xx]\in M_\form$ and for all $\numberq\geq\numberq_0$ there exists $[\rr]\in \P_\Q^d\cap M_\form$ such that \eqref{strongdirichletintro} holds.
\item[(C)] The set
\[
\{[\xx]\in M_\form : \exists \,C,\numberq_0 > 0 \all \,\numberq\geq\numberq_0 \;\; \exists\, [\rr]\in \P_\Q^d\cap M_\form \text{ satisfying \eqref{strongdirichletintro}}\}
\]
has positive $\lambda_{M_\form}$-measure.
\end{itemize}
\end{itemize}
\end{reptheorem}

\begin{affinecorollary*}
Let $M_{\form_\aff}\subset\R^d$ be a nonsingular rational quadric hypersurface 
with $\Qrank\ge 1$. Then
\begin{itemize}
\item[(i)] $\psi_1$ is Dirichlet for intrinsic approximation on $M_{\form_\aff}$.
\item[(ii)] If $\Qrank = \Rrank$, then $\psi_1$ is uniformly Dirichlet on compact subsets of $M_{\form_\aff}$.
\item[(iii)] The following are equivalent:
\begin{itemize}
\item[(A)] $\Qrank = \Rrank = 1$.
\item[(B)] (``Strong Dirichlet'') For every compact set $K\subset M_{\form_\aff}$, there exist $C,\numberq_0 > 0$ such that for all $\xx\in K$ and for all $\numberq\geq\numberq_0$ there exists $\rr\in \Q^d\cap M_{\form_\aff}$ such that
\begin{equation}
\label{strongdirichletaff}
H_\std(\rr) \leq \numberq\text{ and }
\dist(\rr,\xx) \leq \frac{C}{\sqrt{H_\std(\rr) \numberq}}\,
\end{equation}
\item[(C)] The set
\[
\{\xx\in M_{\form_\aff} : \exists \,C,\numberq_0 > 0 \all \,\numberq\geq\numberq_0 \;\; \exists\, \rr\in \Q^d\cap M_{\form_\aff} \text{ satisfying \eqref{strongdirichletaff}}\}
\]
has positive $\lambda_{M_{\form_\aff}}$-measure.
\end{itemize}
\end{itemize}

\end{affinecorollary*}




As for the optimality of Theorem \ref{theoremdirichletquadratic}, as stated above it suffices to show that the set 
\[
\BA_{M_\form} := \BA_{M_\form}(\psi_1)
\]
of \emph{intrinsically badly approximable} points of $M_\form$ is nonempty.
 It follows from the Correspondence Principle below (Lemma \ref{lemmacorrespondence2}) that 
 points in $\BA_{M_\form}$ correspond to bounded orbits of some dynamical system (cf.\ Corollary \ref{corollaryBAboundedorbits}). Then the results of \cite{KleinbockMargulis3} imply: 

\begin{reptheorem}{theoremBAquadratic}[Jarn\'ik--Schmidt for quadric hypersurfaces]
Let $M_\form\subset\P_\R^d$ be a nonsingular rational quadric hypersurface. Then $\HD(\BA_{M_\form}) = \HD(M_\form)$. In particular, the Dirichlet function $\psi_1$ is optimal.
\end{reptheorem}

(No changes needed for the Affine Corollary.)


Using the methods of \cite{KleinbockWeiss3} one can strengthen the conclusion of this theorem to say that $\BA_{M_\form}$ is winning (in the sense of Schmidt). This conclusion also follows from a much more general theorem in \cite{FKMS1} which applies to \emph{all} nondegenerate manifolds and asserts that the set of intrinsically badly approximable points is hyperplane absolute winning (see \cite{BFKRW} for the definition).





\smallskip
Before stating the analogue of Khintchine's theorem for intrinsic approximation on quadric hypersurfaces, let us introduce the following definitions, which will be used in Sections \ref{sectionkhinchinquadratic}--\ref{sectionspecialtype}:

\begin{definition}
\label{definitionregular}
Call a function $\psi$ \emph{regular} if for every (equivalently, for some) $C_1 > 1$  there exists $C_2 > 1$ such that for all $q_1,q_2$, if $1/C_1 \leq q_2/q_1 \leq C_1$, then $1/C_2 \leq \psi(q_2)/\psi(q_1) \leq C_2$. This may be stated succinctly as follows: $q_1 \asymp_\times q_2$ implies $\psi(q_1)\asymp_\times \psi(q_2)$. 
\end{definition}

\begin{definition}
\label{definitiontypes}
The \emph{exceptional quadric hypersurface} is the hypersurface $M_{\form_0}\subset\P_\R^3$ defined by the \emph{exceptional quadratic form}
\begin{equation}
\label{form0}
\form_0(x_0,x_1,x_2,x_3) = x_0 x_3 - x_1 x_2.
\end{equation}
If a quadratic form $\form:\R^4\to\R$ is conjugate over $\Q$ to $\form_0$, we will write $\form\sim\form_0$. We remark that $\form\sim\form_0$ holds if and only if $\form$ is a rational quadratic form in $4$ variables for which $
\Qrank = \Rrank = 2$ (see Lemma \ref{exceptional} for more detail).
\end{definition}

The 
hypersurface $M_{\form_0}$, which we study in detail in Section \ref{sectionspecialtype}, has very interesting properties for intrinsic Diophantine approximation. Note that if $\form\sim\form_0$, then the intrinsic Diophantine theory on $M_\form$ will be more or less the same as the intrinsic Diophantine theory on $M_{\form_0}$. Specifically, the rational equivalence between $\form$ and $\form_0$ defines a diffeomorphism between $M_\form$ and $M_{\form_0}$ which sends rational points to rational points and preserves heights up to a multiplicative constant.

\begin{reptheorem}{theoremkhinchinquadratic}[Khintchine-type theorem for quadric hypersurfaces]
Let $M_\form\subset\P_\R^d$ be a nonsingular rational quadric hypersurface 
with $\Qrank\ge 1$.
Fix $\psi:\N\to(0,\infty)$, and suppose that $\psi$ is regular
and that the function $q\mapsto q\psi(q)$ is nonincreasing.
Then $\A_{M_\form}(\psi)$ has full Lebesgue measure 
if 
the series\Footnote{Here and hereafter $2^\N$ stands for $\{2^n : n\in N\}$.}
\begin{repequation}{loglog}
\begin{cases}
\sum_{\numberq\in 2^\N} \numberq^{d-1} \psi^{d-1}(\numberq) & \form\not\sim\form_0 \\
\sum_{\numberq\in 2^\N} \numberq^2\log\log \numberq\, \psi^2(\numberq) & \form\sim\form_0
\end{cases}
\end{repequation}
diverges; otherwise, $\A_{M_\form}(\psi)$ is Lebesgue null.
\end{reptheorem}

(No changes needed for the Affine Corollary.) 

The appearance of two cases in Theorem \ref{theoremkhinchinquadratic} is due to nontrivial relations among the collection of sets defining $A_{M_{\form_0}}$ that are not present when $\form \not\sim \form_0$. A discussion of these relations, and their implications, is given in Section \ref{sectionspecialtype} (see particularly Remark \ref{remarknontrivialrelation}). 



\smallskip
Using the Mass Transference Principle of Beresnevich and Velani \cite[Theorem 2]{BeresnevichVelani}, one can 
deduce the divergence case of  the Jarn\'ik--Besicovitch theorem for quadric hypersurfaces (Theorem \ref{theoremjarnikquadratic}). Combined with the convergence case (Corollary \ref{corollarykhinchinquadratic}), this gives a complete analogue of the Jarn\'ik--Besicovitch theorem when $ \form\not\sim\form_0$, and a slight discrepancy between the convergence and divergence conditions in the exceptional case. This discrepancy, however,  does not affect  the computation of the Hausdorff dimension of the set of intrinsically $\psi_c$-approximable points  for all $\cee > 1$, 
$\A_{M_\form}(\psi_c)$; namely, Theorem \ref{theoremjarnikquadratic} immediately implies
\begin{equation}
\label{jarnikquadratic}
\HD\big(\A_{M_\form} (\psi_\cee)\big) = \frac{d-1}{\cee}\cdot
\end{equation}
See Section  \ref{sectionkhinchinquadratic}  for a detailed discussion.

\comdmitry{A question: can we upgrade it to
$\HD\big(\A_{M_\form} (\psi)\big) = \inf\left\{s : \sum_{\numberq\in 2^\N} \numberq^k \psi^s(\numberq) < \infty\right\}$?}

\begin{remark}
Let $\H^d$ denote the $d$-dimensional hyperbolic space. Given a quadric hypersurface $M_\form\subset\P_\R^d$ satisfying $\Qrank = \Rrank = 1$, there exists a lattice $\Gamma\subgp\Isom(\H^d)$ and a diffeomorphism $\Phi:\del\H^d\to M_\form$ such that if $P_\Gamma\subset\del\H^d$ is the set of parabolic fixed points of $\Gamma$, then $\Phi(P_\Gamma) = \P_\Q^d\cap M_\form$. This correspondence allows one to deduce the case $\Qrank = \Rrank = 1$ of all the results of this subsection as consequences of known theorems about Diophantine approximation of lattices in $\Isom(\H^d)$; see \6\ref{subsectionkleinian} for more detail.
\end{remark}

\begin{remark}
\label{remarkassumptionsquadratic}
In the above theorems, the form $\form$ is always assumed to be nonsingular with integer coefficients. The latter assumption may be made without loss of generality, since if $\form$ is a quadratic form which is not a scalar multiple of any quadratic form with integer coefficients, then $\P_\Q^d\cap M_\form$ is not dense in $M_\form$; cf.\ Remark \ref{remarknondense}. On the other hand, the nonsingularity assumption does involve a loss of generality. In Theorem \ref{theoremdirichletquadratic}, the singular case can be deduced from the nonsingular case; cf.\ Remark \ref{remarkquadraticsingular}. However, this is not the case for Theorem \ref{theoremkhinchinquadratic}. The use of the nonsingularity assumption appears unavoidable in Theorem \ref{theoremkhinchinquadratic} since if $\form$ is singular, then the associated algebraic group $\orth(\form)$ is not semisimple.
\end{remark}

\noindent{\bf The structure of the paper.}
In Section \ref{sectionpreliminaries} we recall the necessary preliminaries from the theory of quadratic forms. In Section \ref{sectioncorrespondence} we state and prove the Correspondence Principle, which relates intrinsic Diophantine approximation on a nonsingular rational quadric hypersurface $M_\form$ with dynamics on a certain space of arithmetic lattices. This correspondence is similar to the one developed for ambient approximation by Davenport--Schmidt and Dani, see \cite{Dani4, DavenportSchmidt1,DavenportSchmidt2, KleinbockMargulis2, KleinbockMargulis} and generalizes the one used in \cite{KleinbockMerrill}. In particular, we prove (Corollary \ref{corollaryBAboundedorbits}) that $[\xx]\in \BA_{M_\form}$ if and only if a certain trajectory on the corresponding homogeneous space is bounded. 

In Section \ref{sectionproofdirichlet} we prove Theorem \ref{theoremdirichletquadratic} (Dirichlet for quadric hypersurfaces). In Section \ref{sectionkhinchinquadratic} we use \cite[Theorem 1.7]{KleinbockMargulis} to reduce Theorem \ref{theoremkhinchinquadratic} (Khintchine for quadric hypersurfaces) to a statement about Haar measure on the space of $\form$-arithmetic lattices (Proposition \ref{propositionkDL}). In Section \ref{sectionreductiontheory} we use the generalized Iwasawa decomposition \cite[Proposition 8.44]{Knapp} and the reduction theory for algebraic groups \cite[Proposition 2.2]{Leuzinger} to prove Proposition \ref{propositionkDL}, thus completing the proof of Theorem \ref{theoremkhinchinquadratic}. Finally, in Section \ref{sectionspecialtype} we analyze in detail the exceptional quadric hypersurface $M_{\form_0}$
and 
explain intuitively why the converse to (the naive application of) Borel--Cantelli does not hold for intrinsic approximation on this hypersurface.

\section{Preliminaries on quadratic forms and lattices}
\label{sectionpreliminaries}
\subsection{Orthogonality and nonsingularity}
Let $V$ be a vector space over $\R$ and let $\form:V\to\R$ be a quadratic form. We denote by $B_\form$ the unique symmetric bilinear form on $V$ satisfying
\[
\form(\xx) = B_\form(\xx,\xx) \all\, \xx\in V.
\]
We remark that $B_\form$ may be written explicitly in terms of $\form$ via the formula $$B_\form(\xx,\yy) = \frac{\form(\xx + \yy) - \form(\xx) - \form(\yy)}2.$$
\begin{definition}\label{def_orthogonal}
Two elements $\xx,\yy\in V$ are \emph{$\form$-orthogonal} if $B_\form(\xx,\yy) = 0$.
The set of all vectors which are $\form$-orthogonal to a given vector $\xx$ will be denoted $\xx^\perp$, and for any $S\subset V$ we let $S^\perp := \bigcap_{\xx\in S}\xx^\perp$.
\end{definition}

\begin{definition}
\label{definitionnonsingular}
The quadratic form $\form$ is called \emph{nonsingular} if for every $\xx\in V\smallsetminus\{\0\}$, we have $\xx^\perp\propersubset V$, or equivalently, if the map $\xx\mapsto B_\form(\xx,\cdot)$ is an isomorphism between $V$ and $V^*$.
\end{definition}

Note that a form $\form$ is nonsingular if and only if its corresponding hypersurface $M_\form$ is nonsingular as a manifold.
Indeed, recall that $M_\form = [L_\form]$, where $L_\form$ is the light cone of $Q$ defined in \eqref{lightcone}. Then $M_\form$ is nonsingular if and only if $L_\form\smallsetminus\{\0\}$ is nonsingular, which in turn happens if and only if $\nabla \form(\xx)\neq 0$ for all $\xx\in L_\form\smallsetminus\{\0\}$. Since $\nabla \form(\xx) = 2B_\form(\xx,\cdot)$, we have $\nabla \form(\xx) = 0$ if and only if $\xx^\perp = \R^{d + 1}$. Thus $M_\form$ is nonsingular if and only if $\xx^\perp\propersubset \R^{d + 1}$ for all $\xx\in L_\form$. Since $\xx^\perp = \R^{d + 1}$ implies $\xx\in L_\form$, this proves the assertion.

\subsection{Totally isotropic subspaces; rank and renormalization}\label{isotropic}
Throughout this subsection, fix $\K\in\{\R,\Q\}$ and $d\geq 1$, and let $\form:\R^{d + 1}\to\R$ be a nonsingular quadratic form whose coefficients lie in $\K$. We will say that a subspace $E\subgp\R^{d + 1}$ is   a \emph{$\K$-subspace} if $E$ has a basis consisting of elements of $\K^{d + 1}$, or equivalently, if $E$ is defined by equations whose coefficients lie in $\K$. (In the literature, it is sometimes said that $E$ is \emph{defined over $\K$}.)

\begin{definition}
\label{definitionrank}
A subspace $E\subgp\R^{d + 1}$ is  \emph{totally isotropic} if 
$\form|_{E} = 0$.
It is known (see e.g. \cite[Corollary 8.12]{EKM2}) that any two maximal totally isotropic $\K$-subspaces of $\R^{d + 1}$ have the same dimension. This common dimension is called the \emph{$\K$-rank} of $\form$ and is denoted by $\Krank$.
\end{definition}



It turns out to be convenient to conjugate totally isotropic subspaces to canonical subspaces, namely to subspaces of the form
\begin{equation}
\label{Lmdef}
\LL_m := \sum_{i = 0}^{m - 1}\R\ee_i.
\end{equation}
By choosing the right conjugation map $\matrix$, we may also guarantee that the conjugated quadratic form $\nform = \form\circ\matrix$ has a particularly nice form. We make this rigorous as follows:

\begin{definition}
For $m\leq \frac{d + 1}2$, a quadratic form $\nform$ is \emph{$m$-normalized} if there exists a quadratic form $\w\nform$ on $\R^{d + 1 - 2m}$ such that
\begin{equation}
\label{mnormalized}
\nform(\xx) = x_0 x_d + x_1 x_{d - 1} + \ldots + x_{m - 1} x_{d - m + 1} + \w\nform(x_m,\ldots,x_{d - m}).
\end{equation}
The quadratic form $\w\nform$ will be called the \emph{remainder} of $\nform$.
\end{definition}

\begin{proposition}
\label{propositionrenormalization}
Let $E\subgp\R^{d + 1}$ be a totally isotropic $\K$-subspace of dimension $m$. Then $m\leq\frac{d + 1}2$, and there exists $\matrix\in\GL_{d + 1}(\K)$ such that
\begin{itemize}
\item[(i)] $\matrix^{-1}(E) = \LL_m$, and
\item[(ii)] $\nform := \form\circ \matrix$ is $m$-normalized.
\end{itemize}
\end{proposition}
\begin{proof}
Since $\form$ is nonsingular, we may identify $E^*$ with $\R^{d + 1}/E^\perp$ via the map
\begin{equation}
\label{implicitidentification}
\xx + E^\perp \mapsto 
B_\form(\xx,\cdot) |_{E}.
\end{equation}
Let $(\ff_i)_{i = 0}^{m - 1}$ be a $\K$-basis for $E$, and let $(\ff_{d - i}' + E^\perp)_{i = 0}^{m - 1}$ be its dual basis. Inductively define $\ff_{d - i}\in\ff_{d - i}' + E^\perp$ by letting
\[
\ff_{d - i} = \ff_{d - i}' - \sum_{j = 0}^{i - 1} B_\form(\ff_{d - i}',\ff_{d - j})\ff_j - \frac{1}{2}\form(\ff_{d - i}')\ff_i.
\]
Direct calculation shows that $B_\form(\ff_{d - i},\ff_{d - j}) = 0$ for $j\leq i$. Thus $E_2 := \sum_{i = 0}^{m - 1} \R\ff_{d - i}$ is also a totally isotropic $\K$-subspace of $\R^{d + 1}$. Note that by construction, $E_2$ is isomorphic to $E^*$ via the map \eqref{implicitidentification}. Since $E$ is totally isotropic, $E\subgp E^\perp$ and thus $E\cap E_2 = \{\0\}$.

Let $E_3 = E^\perp\cap E_2^\perp = (E + E_2)^\perp$. Since 
$\form|_{E + E_2}$ 
is nonsingular, we have $(E + E_2)\cap E_3 = \{\0\}$ and thus $\R^{d + 1} = E\oplus E_2\oplus E_3$. It follows that $\HD(E_3) = d + 1 - \HD(E) - \HD(E_2) = d + 1 - 2m$, and in particular $m\leq(d + 1)/2$. Let $(\ff_i)_{i = m}^{d - m}$ be a $\K$-basis for $E_3$, and let $\matrix$ be the $(d + 1)\times(d + 1)$ matrix whose columns are given by $\ff_0,\ldots,\ff_d$, so that $\matrix(\ee_i) = \ff_i$ for $i = 0,\ldots,d$. Then $\matrix\in\GL_{d + 1}(\K)$ by the above-mentioned decomposition $\R^{d + 1} = E\oplus E_2\oplus E_3$. (i) and (ii) follow immediately.
\end{proof}

Note that it follows from the above proposition that $\Rrank$ is always less than or equal to $\frac{d+1}2$. Also, if $\form$ has coefficients in $\Q$ then $\Qrank \geq \frac{d-3}2$ unless $\Qrank = \Rrank$. Indeed, without loss of generality suppose that $\form$ is $\Qrank$-normalized, and let $\w\form$ be the remainder of $\form$. If $\Qrank \ne \Rrank$, then $\w\form$ represents zero over $\R$. Since $\w\form$ is a quadratic form in $d + 1 - 2\Qrank$ variables, if $d + 1 - 2\Qrank \geq 5$, by Meyer's theorem $\w\form$ represents zero over $\Q$. This would contradict the definition of $\Qrank$. So $d + 1 - 2\Qrank \leq 4$; rearranging gives $\Qrank \geq \frac{d-3}2$.




\smallskip
Another consequence of Proposition \ref{propositionrenormalization} is a nice characterization of quadratic forms 
rationally equivalent to the exceptional quadratic form $\form_0$ defined in  \eqref{form0}. Recall that the \emph{determinant} $\det(Q)$ of a quadratic form $\form:\R^{d+1}\to\R$ is the determinant of the linear map $\phi_\form:\R^{d+1}\to (\R^{d+1})^*\equiv\R^{d+1}$ defined by $\xx\mapsto B_\form(\xx,\cdot)$.

\begin{lemma}
\label{exceptional}
The following are equivalent for a  
rational quadratic form $\form$ in $4$ variables with $\Qrank\ge 1$:
\begin{itemize}
\item[(i)] $\form\sim\form_0$;
\item[(ii)] $\Qrank = \Rrank = 2$;
\item[(iii)] $\det(Q)$ is a square  of a rational number.
\end{itemize}
\end{lemma}
\begin{proof} Note that for any $\phi\in \GL_4(\R)$ it holds  that $\det(\form\circ\phi) = \det(\form)\det(\phi)^2$. In particular, if $\form_1$ and $\form_2$ are equivalent over $\Q$, then $\det(\form_1)$ is a square if and only if $\det(\form_2)$ is. Thus the implication   (i) $\Rightarrow$ (iii)  follows immediately upon calculating that $\det(\form_0) = 1/16$. 

For the implication (iii) $\Rightarrow$ (ii), suppose that $\det(\form)$ is a square. By Proposition \ref{propositionrenormalization}, we may without loss of generality assume that $\form$ is $1$-normalized. In this case, we have $\det(\form) = -(1/4)\det(\w\form)$ where $\w\form$ is the remainder of $\form$. By the well-known canonical form of quadratic forms, we may without loss of generality assume that $\w\form(\xx) = a_1 x_1^2 + a_2 x_2^2$ for some $a_1,a_2\in\Q$. Then $-\det(\w\form) = -a_1 a_2$ is a square. Thus $\bb := (0,a_2,\sqrt{-a_1 a_2},0)\in\Q^4$, and $\R\ee_0 + \R\bb$ is a totally isotropic subspace of dimension $2$, proving that $\Qrank = 2$.

Finally, the implication (ii) $\Rightarrow$ (i) is a straightforward consequence of Proposition \ref{propositionrenormalization}.
\end{proof}

A convenient fact about $m$-normalized quadratic forms is that any element of $\GL_m(\R)$ extends to an element of $\SL_{d + 1}(\R)$ which preserves every $m$-normalized quadratic form. Specifically, given a quadratic form $\nform:\R^{d + 1}\to\R$, let
\[
\orth(\nform) = \left\{g\in\SL_{d + 1}^\pm(\R): \nform\circ g = \nform\right\}.
\]
Then a direct computation yields the following:
\begin{observation}
Fix $m\leq (d + 1)/2$ and $\matrix\in\GL_m(\R)$. {Define the \emph{reverse} of the matrix $\matrix$ to be the matrix whose $(i,j)$th entry is equal to the $(m - j,m - i)$th entry of $\matrix^{-1}$, and denote this matrix by $\widehat\matrix$. Visually, $\widehat\matrix$ is $\matrix^{-1}$ flipped along the northeast-southwest diagonal.} Let
\begin{equation}
\label{gmatrix}
g_\matrix = \left[\begin{array}{ccc}
\matrix &&\\
& I_{d + 1 - 2m} &\\
&& \widehat\matrix
\end{array}\right].
\end{equation}
Then $g_\matrix\in\orth(\nform)$ for every $m$-normalized quadratic form $\nform$.
\end{observation}

Next, for each $m\leq \frac{d + 1}2$ and $\tt\in\R^m$, 
let
\begin{equation}
\label{gtdef1}
g_\tt = {g_{\diag(e^{-t_0},\ldots,e^{-t_{m - 1}})}} = \left[\begin{array}{ccccccc}
e^{-t_0} &&&&&&\\
& \ddots &&&&&\\
&& e^{-t_{m - 1}} &&&&\\
&&& I_{d + 1 - 2m} &&&\\
&&&& e^{t_{m - 1}} &&\\
&&&&& \ddots &\\
&&&&&& e^{t_0}
\end{array}\right].
\end{equation}
Of particular importance will be the case $m = 1$, in which case
\begin{equation}
\label{gtdef2}
g_t = \left[\begin{array}{ccc}
e^{-t} &&\\
& I_{d - 1} &\\
&& e^t
\end{array}\right].
\end{equation}

%
%
A simple computation immediately yields the following observation, which will turn out to be quite useful:

\begin{observation}
\label{observationgt}
For $t\geq 0$ and $\xx\in\R^{d + 1}$,
\begin{equation}
\label{gt}
\dist(\xx,\LL_1) \leq \|g_t(\xx)\|,
\end{equation}
where $\LL_1$ is as in \eqref{Lmdef}.
\end{observation}

\subsection{The space of lattices; Mahler's compactness criterion}
\label{subsectionmahler}
As stated in the introduction, our main tool for proving theorems concerning intrinsic approximation on $M_\form$ is a correspondence principle between approximations of a point in $M_\form$ and dynamics in the space of lattices. We will describe this correspondence principle in Section \ref{sectioncorrespondence} below, while here we introduce the space of lattices which we are interested in, namely the space of \emph{$\form$-arithmetic} lattices.

\begin{definition}
Fix a quadratic form $\form:\R^{d + 1}\to\R$. A lattice $\Lambda\subgp\R^{d + 1}$ is \emph{$\form$-arithmetic} if $\form(\Lambda)\subset\Z$. (Symmetrically, we may also say that $\form$ is $\Lambda$-\emph{arithmetic}.) The set of 
$\form$-arithmetic lattices will be denoted {by} $\Omega_\form$, while the set of all 
lattices in $\R^{d + 1}$ will be denoted  {by}  $\Omega_d$.
\end{definition}
\begin{observation}
A quadratic form is $\Z^{d + 1}$-arithmetic if and only if its coefficients are integral.
\end{observation}
Clearly, $\Omega_\form$ is preserved by the action of $\orth(\form)$. If ${\Lambda}\in\Omega_\form$ is fixed, we denote its stabilizer by $\orth(\form;{\Lambda})$ and its orbit by $\Omega_{\form,{\Lambda}}$. We will implicitly identify $\Omega_{\form,{\Lambda}}$ with the homogeneous space $\orth(\form)/\orth(\form;{\Lambda})$ via the map $g\orth(\form;{\Lambda})\mapsto g{\Lambda}$. This automatically endows $\Omega_{\form,{\Lambda}}$ with a topological structure and, since $\orth(\form)$ is unimodular and $\orth(\form;{\Lambda})$ is discrete, a Haar measure, which we will denote by $\mu_{\form,{\Lambda}}$.

Viewing $\Omega_{\form,{\Lambda}}$ as a homogeneous space could conceivably give it a different topology than viewing it as a subset of $\Omega_d$, which has its own topology from its identification with $\GL_{d + 1}(\R)/\GL_{d + 1}(\Z)$ coming from the map $g\GL_{d + 1}(\Z)\mapsto g(\Z^{d + 1})$. Fortunately, it turns out that these topologies are identical:
\begin{proposition}
\label{propositioninclusionproper}
The inclusion map $\Omega_{\form,{\Lambda}}\to\Omega_d$ is proper and continuous, when both spaces are endowed with the topologies coming from the identification with their corresponding homogeneous spaces. Consequently, the topology on $\Omega_{\form,{\Lambda}}$ is unambiguous.
\end{proposition}
\begin{proof}
The continuity of the inclusion map follows directly from the continuity of the inclusion map from $\orth(\form)$ to $\GL_{d + 1}(\R)$. Let us show that the inclusion map is proper. Let $(\Lambda_n)_1^\infty$ be a sequence in $\Omega_{\form,{\Lambda}}$ converging to a point $\Lambda_0\in\Omega_d$. Then there exist $\GL_{d + 1}(\R)\ni g_n\to g_0\in \GL_{d + 1}(\R)$ such that $\Lambda_n = g_n(\Z^{d + 1})$ for all $n\geq 0$. This implies that for all $n\geq 1$, $\form_n := \form\circ g_n$ is a $\Z^{d + 1}$-arithmetic quadratic form, and $\form_n\to \form_0 := \form\circ g_0$. Since the space of $\Z^{d + 1}$-arithmetic quadratic forms is discrete (being identical to the space of quadratic forms with coefficients in $\Z$), we have $\form_n = \form_0$ for all sufficiently large $n$. (Thus \emph{a posteriori} $\form_0$ is $\Z^{d + 1}$-arithmetic, or equivalently $\Lambda_0 \in \Omega_\form$.) For $n$ satisfying $\form_n = \form_0$, we have $h_n := g_n g_0^{-1}\in \orth(\form)$; in particular $\Lambda_0 = h_n^{-1}(\Lambda_n) \in \Omega_{\form,{\Lambda}}$. On the other hand $\Lambda_n = h_n\Lambda_0$ and $h_n\to h_0 = \id$; this implies that $\Lambda_n\to \Lambda_0$ in the topology on $\Omega_{\form,{\Lambda}}$ coming from its identification with the homogeneous space $\orth(\form)/\orth(\form;{\Lambda})$.
\end{proof}
We now recall Mahler's famous compactness criterion, and deduce an analogue in the context of quadratic forms. For $\Lambda\in\Omega_d$ let
\begin{equation}
\label{mindist}
\mindist(\Lambda) := \min_{\pp\in\Lambda\smallsetminus\{\0\}} \|\pp\|.
\end{equation}
\begin{theorem}[Mahler's compactness criterion, {\cite[Theorem 2]{Mahler_compactness}}]
\label{theoremmahler}
A set $S\subset\Omega_d$ is precompact if and only if $\mindist$ is bounded from below on $S$, and the covolumes of all lattices in $S$ are uniformly bounded from above.
\end{theorem}
For $\Lambda\in\Omega_\form$ let
\begin{equation*}
\label{mindistQ}
\mindist_\form(\Lambda) = \min_{\pp\in\Lambda\cap L_\form\smallsetminus\{\0\}}\|\pp\|.
\end{equation*}
We let $\mindist_\form(\Lambda) = \infty$ if $\Lambda\cap L_\form\smallsetminus\{\0\} = \emptyset$.
\begin{observation}
\label{observationmindistQ}
If we let
\[
\|\form\| = \max_{\|\xx\| = \|\yy\| = 1}|B_\form(\xx,\yy)|
\]
then $\min(\mindist_\form,1/\sqrt{\|\form\|}) \leq \mindist \leq \mindist_\form$.
\end{observation}
\begin{proof}
For $\pp\in\Lambda\smallsetminus L_\form$, $\|\pp\| \geq \sqrt{|\form(\pp)|/\|\form\|} \geq 1/\sqrt{\|\form\|}$.
\end{proof}
\begin{corollary}[Analogue of Mahler's compactness criterion]
\label{corollarymahlerquadratic}
Fix ${\Lambda}\in\Omega_\form$. A set $S\subset\Omega_{\form,{\Lambda}}$ is precompact if and only if $\mindist_\form$ is bounded from below on $S$.
\end{corollary}
\begin{proof}
By Observation \ref{observationmindistQ}, $\mindist_\form$ is bounded from below on $S$ if and only if $\mindist$ is bounded from below on $S$. But by Theorem \ref{theoremmahler}, since the covolumes of all lattices in $\Omega_{\form,{\Lambda}}$ are the same, $\mindist$ is bounded from below if and only if $S$ is precompact in the topology of $\Omega_d$. By Proposition \ref{propositioninclusionproper}, this occurs if and only if $S$ is precompact in the topology of $\Omega_{\form,{\Lambda}}$. (Here we use not only the fact that the topology on $\Omega_{\form,{\Lambda}}$ is the one induced from $\Omega_d$, but also the fact that the inclusion map is proper, and consequently $\Omega_{\form,{\Lambda}}$ is closed in $\Omega_d$.)
\end{proof}

%
%
%


\subsection{Relation to Kleinian lattices}
\label{subsectionkleinian}

In this subsection, we describe the relation between the intrinsic Diophantine approximation of a quadric hypersurface $M_\form$ satisfying $\Qrank = \Rrank = 1$ and the approximation of points in the boundary of $d$-dimensional hyperbolic space $\H^d$ by parabolic fixed points in a lattice $\Gamma\subgp\Isom(\H^d)$ which depends on the quadric hypersurface $M_\form$. Since the latter situation is well-studied, this correspondence can be used to immediately prove the theorems of \6\ref{subsectionquadratic} in the case $\Qrank = \Rrank = 1$. (However, our proofs of the theorems of \6\ref{subsectionquadratic} in the general case are not dependent on assuming $\Rrank > 1$, so this subsection can be skipped without any loss of generality.)

Let $\form:\R^{d + 1}\to\R$ be a quadratic form with integer coefficients satisfying $\Qrank = \Rrank = 1$. Then the signature of $\form$ is either $(d,1)$ or $(1,d)$. Without loss of generality, we will suppose that its signature is $(d,1)$. The \emph{hyperboloid model of hyperbolic geometry} is the set
\[
\H^d := \{\xx\in\R^{d + 1} : \form(\xx) = -1\}
\]
with the Riemannian metric 
$\form|_{\H^d}$ (its positive-definiteness is guaranteed by the fact that the signature of $\form$ is $(d,1)$). The \emph{hyperbolic distance} is given by the formula
\[
\cosh\dist(\xx,\yy) = |B_\form(\xx,\yy)|.
\]
Note that by Sylvester's law of inertia, up to isometry the space $(\H^d,\dist)$ does not depend on $\form$, but only on $d$. For the equivalence of the hyperboloid model with other standard models of hyperbolic geometry, see e.g.\ \cite{CFKP}. The \emph{boundary} of $\H^d$, denoted $\del\H^d$, is defined to be the boundary of $[\H^d]$ in $\P_\R^d$. Observe that $\del\H^d = M_\form$. A \emph{horoball} in $\H^d$ is a set of the form
\[
\{\xx\in\H^d : \busemann_{[\rr]}(\zz,\xx) > t\},
\]
where $\zz\in\H^d$, $[\rr]\in \del\H^d$, $t\in\R$, and $\busemann_{[\rr]}$ denotes the \emph{Busemann function}
\[
\busemann_{[\rr]}(\zz,\xx) = \lim_{[\yy]\to[\rr]}[\dist(\yy,\zz) - \dist(\yy,\xx)].
\]
Such a horoball is said to be \emph{centered} at the point $[\rr]$. The isometry group of $\H^d$ is given by
\[
\Isom(\H^d) = \orth(\form).
\]
Since $\form$ has integer coefficients, the subgroup
\[
\Gamma := \orth(\form;\Z) := \orth(\form) \cap \GL_{d + 1}(\Z)
\]
is a lattice in $\orth(\form)$ \cite[Theorem 7.8]{BorelHarish-Chandra1}. Let $P_\Gamma\subset\del\H^d$ denote the set of parabolic fixed points of $\Gamma$.

We now state the relation between intrinsic approximation of $M_\form$ and approximation of $\del\H^d$ by $P_\Gamma$:

\begin{proposition}
\label{propositionkleinian}
~
\begin{itemize}
\item[(i)] There exists a $\Gamma$-invariant disjoint family of horoballs $(H_{[\rr]})_{[\rr]\in \P_\Q^d\cap M_\form}$ such that for each $[\rr]\in \P_\Q^d\cap M_\form$, $H_{[\rr]}$ is centered at $[\rr]$ and 
\begin{equation}
\label{heightasymp}
H_\std([\rr]) \asymp_\times e^{\dist(\zz,H_{[\rr]})},
\end{equation}
where $\zz\in\H^d$ is fixed.
\item[(ii)] $\P_\Q^d\cap M_\form = P_\Gamma$.
\end{itemize}
\end{proposition}
Using Proposition \ref{propositionkleinian}, one may translate \cite[Theorems 1 and 4]{StratmannVelani}, \cite[Theorem C]{Stratmann3}, and \cite[Theorem 2]{MelianPestana} (see also \cite{FSU4} and the references therein for subsequent generalizations) into the context of quadratic forms, yielding the results of \6\ref{subsectionquadratic} in the case $\Qrank = \Rrank = 1$. Details are left to the reader.

\begin{proof}[Proof of \text{(i)}]
Fix $\epsilon > 0$, and for each $[\rr]\in \P_\Q^d\cap M_\form$ let
\[
H_{[\rr]} = \{\xx\in\H^d : |B_\form(\xx,\rr)| < \epsilon\},
\]
where $\rr$ is the unique primitive integral representative of $[\rr]$. The fact that $H_{[\rr]}$ is a horoball centered at $[\rr]$ follows from the following well-known formula for the Busemann function in the hyperboloid model:
\[
\busemann_{[\rr]}(\xx,\yy) = \log\frac{|B_\form(\xx,\rr)|}{|B_\form(\yy,\rr)|}\cdot
\]
Since $\P_\Q^d\cap M_\form$ and $\form$ are both invariant under $\Gamma$, it is clear that the collection $(H_{[\rr]})_{[\rr]\in \P_\Q^d\cap M_\form}$ is $\Gamma$-invariant. Next, we will show that the collection $(H_{[\rr]})_{[\rr]\in \P_\Q^d\cap M_\form}$ is disjoint for $\epsilon$ sufficiently small. Indeed, suppose $\xx\in H_{[\rr_1]}\cap H_{[\rr_2]}$, and apply $g\in\orth(\form)$ such that $g(\xx) = \ww$, where $\ww\in\H^d$ is fixed. Then $|B_\form\big(\ww,g(\rr_i)\big)| < \epsilon$, where $\rr_i$ is the primitive integral representative of $[\rr_i]$. On the other hand, since $\form$ has signature $(d,1)$ and $\form(\ww) = -1$, we have
\begin{equation}
\label{forallr}
|B_\form(\ww,\rr)| \asymp_\times \|\rr\| \text{ for all }\rr\in L_\form.
\end{equation}
Thus $\|g(\rr_i)\| \lesssim_\times \epsilon$, and so $|B_\form(\rr_1,\rr_2)| \lesssim_\times \|\form\|\epsilon^2$. Thus for $\epsilon$ sufficiently small, $|B_\form(\rr_1,\rr_2)| < 1/2$. On the other hand $B_\form(\rr_1,\rr_2)\in\Z/2$ since $\form$ has integer coefficients, so $B_\form(\rr_1,\rr_2) = 0$. Since $\Qrank = 1$, this implies $[\rr_1] = [\rr_2]$.

Since the horoballs $(H_{[\rr]})_{[\rr]\in \P_\Q^d\cap M_\form}$ are disjoint open subsets of the connected set $\H^d$, there exists $\zz\in\H^d\smallsetminus\bigcup_{[\rr]} H_{[\rr]}$. Now fix $[\rr]\in\P_\Q^d\cap M_\form$, and we will demonstrate \eqref{heightasymp}. Letting $\xx\in\del H_{[\rr]}$ be arbitrary, we calculate
\[
e^{\dist(\zz,H_{[\rr]})} = e^{\busemann_{[\rr]}(\zz,\xx)} = \frac{|B_\form(\zz,\rr)|}{\epsilon}\cdot
\]
Combining with \eqref{forallr} yields \eqref{heightasymp}.
\end{proof}
\begin{proof}[Proof of \text{(ii)}]
Suppose that $[\rr]$ is a parabolic fixed point of $\Gamma$, say $g([\rr]) = [\rr]$ for some parabolic $g\in\Gamma$. Then the line representing $[\rr]$ is precisely the set
\[
\{\xx\in\R^{d + 1} : g(\xx) = \xx\},
\]
which is a rational subspace of $\R^{d + 1}$. Consequently $[\rr]\in\P_\Q^d\cap M_\form$.

Conversely, suppose that $[\rr]\in\P_\Q^d\cap M_\form$. As above we fix $\zz\in\H^d\smallsetminus\bigcup_{[\rr]} H_{[\rr]}$. Since the collection $(H_{[\rr]})_{[\rr]\in \P_\Q^d\cap M_\form}$ is $\Gamma$-invariant, this implies $g(\zz)\notin H_{[\rr]}$ for all $g\in\Gamma$. In particular, $[\rr]$ cannot be a conical limit point of $\Gamma$ (see e.g. \cite[\63.2]{Bowditch_geometrical_finiteness} for the definition). But since $\Gamma$ is a lattice, every point of $\del\H^d$ is either a conical limit point or a parabolic fixed point (e.g. \cite[\64]{Bowditch_geometrical_finiteness}). Thus $[\rr]\in P_\Gamma$.
\end{proof}

\section{The correspondence principle}
\label{sectioncorrespondence}

In this section we introduce the correspondence principle alluded to in the introduction. It is an intrinsic approximation analogue of the so-called Dani Correspondence for ambient approximation \cite{Dani4, DavenportSchmidt1,DavenportSchmidt2, KleinbockMargulis2, KleinbockMargulis}. A special case can be found in \cite[Theorem 1.5]{KleinbockMerrill}.

Fix $d\geq 2$, and let $\form:\R^{d + 1}\to\R$ be a nonsingular quadratic form with integer coefficients. Suppose that 
$\Qrank\geq 1$. By Proposition \ref{propositionrenormalization}, there exists a matrix $\matrix\in\GL_{d + 1}(\Q)$ such that $\nform := \form\circ\matrix$ is $\Qrank$-normalized. 
Let $\Lambda_* = \matrix^{-1}(\Z^{d + 1})$. Note that $\Lambda_*$ is commensurable with $\Z^{d + 1}$ and that $\Lambda_*\in\Omega_\nform$. Moreover, the $\Q$-ranks of $\form$ and $\nform$ are identical, and the same goes for the $\R$-ranks, so denoting these ranks by $\Qrank$ and $\Rrank$ will not cause ambiguity.

Consider the maps $\pi_1:\orth(\nform)\to M_\form$ and $\pi_2:\orth(\nform)\to \Omega_{\nform,\Lambda_*}$ defined by
\begin{equation}
\label{pi1pi2}
\begin{split}
\pi_1(g) &= \matrix\circ g([\ee_0]), \\
\pi_2(g) &= g^{-1}\Lambda_* = (\matrix\circ g)^{-1}(\Z^{d + 1}).
\end{split}
\end{equation}
Now fix $g\in\orth(\nform)$, and let \begin{equation}
\label{gxlambda}[\xx] = \pi_1(g)\text{ and }\Lambda = \pi_2(g).
\end{equation} The first version of the correspondence principle gives a relation between the following entities:
\begin{itemize}
\item[(A)] Rational points in $\P_\Q^d\cap M_\form$ which are close to $[\xx]$.
\item[(B)] Points in $\Lambda_\prim\cap L_\nform$ which are close to $\LL_1$. Here $\Lambda_\prim$ denotes the set of primitive vectors of $\Lambda$, and $\LL_1 = \R\ee_0$ is as in \eqref{Lmdef}.
\item[(C)] Pairs $(t,\qq)$, where $\qq\in g_t\Lambda_\prim\cap L_\nform$ is close to $\{\0\}$.
\end{itemize}

\begin{lemma}[Correspondence principle, form 1]
\label{lemmacorrespondence1}
Let $g$, $[\xx]$, and $\Lambda$ be as in \eqref{gxlambda}. Then
\begin{itemize}
\item[(i)] $\pp\mapsto\matrix\circ g([\pp])$ is a bijection between $\Lambda_\prim\cap L_\nform$ and $\P_\Q^d\cap M_\form$.
\item[(ii)] Fix $\pp\in\Lambda_\prim\cap L_\nform$, and let $[\rr] = \matrix\circ g([\pp])$. Then
\begin{equation}
\label{correspondence1A}
\dist([\rr],[\xx]) \asymp_{\times,g} \frac{\dist(\pp,\LL_1)}{\|\pp\|} \text{ and } H_\std([\rr]) \asymp_{\times,g} \|\pp\|.
\end{equation}
In particular, if $\psi:(0,\infty) \to (0,\infty)$ is a regular function (cf.\ Definition \ref{definitionregular}), then
\begin{equation}
\label{correspondence1B}
\frac{\dist([\rr],[\xx])}{\psi\circ H_\std([\rr])} \asymp_{\times,g,\psi} \frac{\dist(\pp,\LL_1)}{\|\pp\|\psi(\|\pp\|)}\cdot
\end{equation}
In each case, the implied constant can be made independent of $g$ if $g$ is constrained to lie in a bounded subset of $\orth(\nform)$.
\item[(iii)] Fix $\pp\in\Lambda\cap L_\nform\smallsetminus\{\0\}$. such that $|p_0| = \|\pp\|$ (i.e.\ $|p_0| \geq \dist(\pp,\LL_1)$). For $t \geq 0$,
\begin{equation}
\label{correspondence1C}
\max\left(\dist(\pp,\LL_1) , \frac{\|\pp\|}{e^t} \right) \leq \|g_t(\pp)\| \lesssim_\times \max\left(\dist(\pp,\LL_1) , \frac{\|\pp\|}{e^t} , \frac{e^t \dist(\pp,\LL_1)^2}{\|\pp\|}\right).
\end{equation}
In particular, letting $t(\pp) = \log\big(\|\pp\|/\dist(\pp,\LL_1)\big)$ we have
\begin{equation}
\label{correspondence1D}
\|g_{t(\pp)}(\pp)\| \asymp_\times \dist(\pp,\LL_1).
\end{equation}
\end{itemize}
\end{lemma}
\begin{proof}
Part (i) is straightforward. Regarding part (ii), formula \eqref{correspondence1A} is perhaps elucidated by the calculation
\begin{align*}
\dist([\rr],[\xx]) &= \dist\big(\phi\circ g([\pp]),\phi\circ g([\ee_0])\big) \asymp_{\times,g} \dist([\pp],[\ee_0]) \asymp_\times \frac{\dist(\pp,\LL_1)}{\|\pp\|}\\
H_\std([\rr]) &= \|\phi\circ g(\pp)\| \asymp_{\times,g} \|\pp\|.
\end{align*}
Formula \eqref{correspondence1B} follows from \eqref{correspondence1A} together with the regularity of $\psi$; since $H_\std([\rr]) \asymp_{\times,g} \|\pp\|$, we have $\psi\circ H_\std([\rr]) \asymp_{\times,g} \psi(\|\pp\|)$, and \eqref{correspondence1B} follows upon combining with the first part of \eqref{correspondence1A}.

We proceed to the proof of (iii). The first inequality of \eqref{correspondence1C} is an immediate consequence of the definition of $g_t$. To demonstrate the second inequality of \eqref{correspondence1C}, let $\qq = g_t(\pp)$, and write $\qq = (q_0,\ldots,q_d)$. Then $|q_1|,\ldots,|q_{d - 1}| \leq \dist(\pp,\LL_1)$, while $|q_0| = \|\pp\|/e^t$. To bound $|q_d|$, we use the fact that $\qq \in L_\nform$, which means that
\[
\nform(\qq) = q_0 q_d + \w\nform(q_1,\ldots,q_{d - 1}),
\]
where $\w\nform$ is the remainder of $\nform$. Rearranging, we have
\[
|q_d| = \frac{|\w\nform(q_1,\ldots,q_{d - 1})|}{|q_0|} \leq \frac{\|\w\nform\|\cdot\|(q_1,\ldots,q_{d - 1})\|^2}{|q_0|} \lesssim_\times \frac{\dist(\pp,\LL_1)^2}{|q_0|} = \frac{e^t \dist(\pp,\LL_1)^2}{\|\pp\|}\cdot
\qedhere\]
\end{proof}


The second version of the correspondence principle depends on a function $\psi:(0,\infty)\to(0,\infty)$, and may be stated as follows:

\begin{lemma}[Correspondence principle, form 2]
\label{lemmacorrespondence2}
Let $g$, $[\xx]$, and $\Lambda$ be as  in \eqref{gxlambda}, and assume that $[\xx]$ is irrational (equivalently,  that $\Lambda\cap\LL_1 = \{\0\}$). Let $\psi:(0,\infty)\to(0,\infty)$ be a regular 
function such that the map $q\mapsto q\psi(q)$ is nonincreasing and tends to zero. 
Then
\begin{equation}
\label{correspondence2}
\liminf_{\substack{[\rr]\to [\xx] \\ [\rr]\in \P_\Q^d\cap M_\form}} \frac{\dist([\rr],[\xx])}{\psi\circ H_\std([\rr])} \asymp_{\times,g,\psi} \liminf_{\substack{[\pp]\to [\ee_0] \\ \pp\in\Lambda_\prim\cap L_\nform}} \frac{\dist(\pp,\LL_1)}{\|\pp\|\psi(\|\pp\|)} \asymp_{\times,\psi} \liminf_{t\to\infty} \frac{e^{-t}}{\psi\big(e^t \mindist_\nform(g_t\Lambda)\big)}\cdot
\end{equation}
\end{lemma}
\begin{proof}
The first asymptotics follows directly from (i) and (ii) of Lemma \ref{lemmacorrespondence1}. The second asymptotics can be rewritten in a more convenient form using the function $\Psi(q) := q\psi(q)$:
\begin{equation}
\label{correspondence2A}
\liminf_{\substack{[\pp]\to [\ee_0] \\ \pp\in\Lambda_\prim\cap L_\nform}} \frac{\dist(\pp,\LL_1)}{\Psi(\|\pp\|)} \asymp_\times \liminf_{t\to\infty} \frac{\mindist_\nform(g_t\Lambda)}{\Psi\big(e^t \mindist_\nform(g_t\Lambda)\big)}\cdot
\end{equation}
To demonstrate the $\lesssim$ direction of \eqref{correspondence2A}, for each $t\geq 0$ choose $\pp_t\in\Lambda_\prim\cap L_\nform$ such that $\mindist_\nform(g_t\Lambda) = \|g_t(\pp_t)\|$. Then by \eqref{correspondence1C}, we have
\[
\dist(\pp_t,\LL_1) \leq \mindist_\nform(g_t\Lambda) \text{ and } \|\pp_t\| \leq e^t \mindist_\nform(g_t\Lambda)
\]
and thus
\begin{equation}
\label{intermediate}
\frac{\dist(\pp_t,\LL_1)}{\Psi(\|\pp_t\|)} \leq \frac{\mindist_\nform(g_t\Lambda)}{\Psi\big(e^t \mindist_\nform(g_t\Lambda)\big)}
\end{equation}
Here we have used the fact that the function $\Psi$ is nonincreasing. Next, suppose we have a sequence $t_k\to\infty$ such that $\lim_{k\to\infty}\frac{e^{-t_k}}{\psi\big(e^{t_k} \mindist_\nform(g_t\Lambda)\big)} < \infty$. Since $\Psi(q) \to 0$ as $q\to\infty$, it follows that $\mindist_\nform(g_{t_k}\Lambda) \to 0$. In particular
\[
\dist(\pp_{t_k},\LL_1) \to 0.
\]
Since $\Lambda\cap\LL_1 = \{\0\}$, this implies that the set $\{\pp_{t_k} : k\in\N\}$ is infinite. Combining with \eqref{intermediate} yields the $\lesssim$ direction of \eqref{correspondence2A}.

To demonstrate the $\gtrsim$ direction of \eqref{correspondence2A}, suppose that $\pp_k\in\Lambda_\prim\cap L_\nform$ is a sequence such that $[\pp_k]\to [\ee_0]$. For each $k$, let $t_k = t(\pp_k)$ be defined as in (iii) of Lemma \ref{lemmacorrespondence1}. Since $[\pp_k]\to [\ee_0]$, we have $t_k\to \infty$. On the other hand, by \eqref{correspondence1D} we have
\begin{align*}
\mindist_\nform(g_{t_k}\Lambda) &\leq_\pt  \|g_{t_k}(\pp_k)\| \asymp_\times \dist(\pp_k,\LL_1)\\
e^{t_k}\mindist_\nform(g_{t_k}\Lambda) &\lesssim_\times e^{t_k}\dist(\pp_k,\LL_1) = \|\pp_k\|
\end{align*}
and so
\[
\frac{\mindist_\nform(g_{t_k}\Lambda)}{\Psi\big(e^t \mindist_\nform(g_{t_k}\Lambda)\big)} \lesssim_\times \frac{\dist(\pp_k,\LL_1)}{\Psi(\|\pp_k\|)}\cdot
\]
Letting $k\to\infty$ finishes the proof.
\end{proof}


The following corollary is a direct analogue of Dani's correspondence between bounded orbits and badly approximable vectors/matrices \cite[Theorem 2.20]{Dani4}. 
\begin{corollary}
\label{corollaryBAboundedorbits}
Let $g$, $[\xx]$, and $\Lambda$ be as  in \eqref{gxlambda}. Then the following are equivalent:
\begin{itemize}
\item[(A)] $[\xx]$ is intrinsically badly approximable, i.e.
\[
[\xx]\in\BA_{M_\form}.
\]
\item[(B)]
\[
\inf_{\pp\in\Lambda\cap L_\form\smallsetminus\{\0\}} \dist(\pp,\LL_1) > 0.
\]
\item[(C)] The orbit
\[
(g_t\Lambda)_{t\geq 0}
\]
is bounded in $\Omega_\nform$.
\end{itemize}
\end{corollary}
\begin{proof}
Clearly all the above statements are false if $[\xx]$ is irrational. Otherwise, let $\CC$ be the class of all regular 
functions $\psi$ such that the map $q\mapsto q\psi(q)$ is nonincreasing and tends to zero. Then (A) is equivalent to the assertion that the left hand side of \eqref{correspondence2} is positive for all $\psi\in\CC$, (B) is equivalent to the assertion that the middle of \eqref{correspondence2} is positive for all $\psi\in\CC$, and (C) is equivalent (by Corollary \ref{corollarymahlerquadratic}) to the assertion that the right hand side of \eqref{correspondence2} is positive for all $\psi\in\CC$.
\end{proof}

\begin{remark}
It is somewhat annoying that Lemma \ref{lemmacorrespondence2} requires the assumption that $q\psi(q)\to 0$ as $q\to\infty$, so that the Dirichlet function $\psi = \psi_1$ is ruled out. (If we were allowed to use $\psi = \psi_1$, then the proof of Corollary \ref{corollaryBAboundedorbits} could be made even simpler -- just consider $\psi = \psi_1$ rather than all functions $\psi\in\CC$.) However, this assumption is necessary, as can be seen as follows. Arguing as in \cite[Proof of Corollary 3.5]{KleinbockMerrill}, one can show that 
there exists $C > 0$ such that $\mindist_\nform(\Lambda) \leq C$ for all $\Lambda\in\Omega_{\nform,\Lambda_*}$. 
(Indeed, otherwise one can take a sequence $\Lambda_n\in\Omega_{\nform,\Lambda_*}$ with $\mindist_\nform(\Lambda_n)\to\infty$; such a sequence cannot have a convergent subsequence, yet it is precompact in view of Corollary \ref{corollarymahlerquadratic}.)
This $C$ is a uniform upper bound on the right hand side of \eqref{correspondence2} when $\psi = \psi_1$. However, we know that when $\Qrank \neq \Rrank$, then there is no uniform upper bound on the left hand side of \eqref{correspondence2}; this follows from Theorem \ref{theoremdirichletquadratic}(ii) below. Thus the left and right hand sides cannot be asymptotic.\Footnote{A closer analysis shows that when $\psi = \psi_1$, the left and right hand sides of \eqref{correspondence2} are not necessarily asymptotic even when both of them are close to $0$.}
\end{remark}

Using Corollary \ref{corollaryBAboundedorbits} we can now prove Theorem \ref{theoremBAquadratic}.
\begin{theorem}
\label{theoremBAquadratic}
Let $M_\form\subset\P_\R^d$ be a nonsingular rational quadric hypersurface. Then $\HD(\BA_{M_\form}) = \HD(M_\form)$. In particular, the Dirichlet function $\psi_1$ is optimal.
\end{theorem}

\begin{proof} First observe that $\BA_{M_\form} =  M_\form$ if $\Qrank = 0$; thus it suffices to consider the case $\Qrank \ge 1$.
Let $\BA_{\Omega_\nform}\subset\Omega_\nform$ denote the set of lattices in $\Omega_{\nform}$ whose orbit under the $g_t$ flow is bounded. By \cite[Theorem 5.2]{KleinbockMargulis3}, we have $\HD(\BA_{\Omega_\nform}) = \HD(\Omega_\nform)$. On the other hand, by Corollary \ref{corollaryBAboundedorbits}, we have $\BB := \pi_2^{-1}(\BA_{\Omega_\nform}) = \pi_1^{-1}(\BA_{M_\form})$.

Since $\pi_2$ is a fibration whose fibers are isomorphic to $\Stab(\Lambda_*)$, the set $\BB = \pi_2^{-1}(\BA_{\Omega_\nform}) \subset \orth(\nform)$ has the same local structure as the product $\BA_{\Omega_\nform} \times \Stab(\Lambda_*) \subset \Omega_\nform \times \Stab(\Lambda_*)$. Now, since $\Stab(\Lambda_*)$ is a manifold, its Hausdorff dimension and upper box dimension are equal. (We refer to \cite[p.38]{Falconer_book} for the definition of upper box dimension.) So by \cite[Corollary 7.4]{Falconer_book}, we have $\HD\big(A\times\Stab(\Lambda_*)\big) = \HD(A) + \dim\big(\Stab(\Lambda_*)\big)$ for all $A\subset \Omega_\nform$. Taking the cases $A = \BA_{\Omega_\nform}$ and $A = \Omega_\nform$ and using the fact that Hausdorff dimension is a local property, we have
\begin{align*}
\HD(\BB) &= \HD\big(\BA_{\Omega_\nform}\big) + \dim(\Stab(\Lambda_*)\big),& \HD\big(\orth(\nform)\big) &= \HD(\Omega_\nform) + \dim\big(\Stab(\Lambda_*)\big).
\end{align*}
A similar argument gives
\begin{align*}
\HD(\BB) &= \HD\big(\BA_{M_\form}\big) + \dim\big(\Stab([\ee_0])\big),& \HD\big(\orth(\nform)\big) &= \HD(M_\form) + \dim\big(\Stab([\ee_0])\big).
\end{align*}
Thus since $\HD(\BA_{\Omega_\nform}) = \HD(\Omega_\nform)$, we have $\HD(\BA_{M_\form}) = \HD(M_\form)$.
\end{proof}



Under the assumption that $q\psi(q)\to 0$ as $q\to\infty$, Lemma \ref{lemmacorrespondence2} can be used to dynamically describe the sets $\A_{M_\form}(\psi)$ and $\WA_{M_\form}(\psi)$:
\begin{corollary}
\label{corollarycorrespondencesets}
Let $\psi:(0,\infty)\to(0,\infty)$ be a regular continuous function such that the map $q\mapsto q\psi(q)$ is nonincreasing and tends to zero, let
\[
r_\psi(t) := e^{-t} \psi^{-1}(e^{-t})
\]
(this is well defined for large enough $t$),
and let
\begin{equation}
\label{defarpsi}
\A(r_\psi,\Omega_{\nform,\Lambda_*}) := \{\Lambda\in\Omega_{\nform,\Lambda_*} : \mindist_\nform(g_t\Lambda) \leq r_\psi(t) \text{ for an unbounded set of  $t\geq 0$}\}.
\end{equation}
Then for every compact set $\KK\subset\orth(\nform)$, there exists $C > 0$ (depending on $\psi$ and $\KK$) such that
\begin{equation}
\label{correspondence4}
\pi_1^{-1}\big(\A_{M_\form}(\psi/C)\big)\cap \KK \subset \pi_2^{-1}\big(\A(r_\psi,\Omega_{\nform,\Lambda_*})\big)\cap \KK \subset \pi_1^{-1}\big(\A_{M_\form}(C\psi)\big)\cap \KK.
\end{equation}
Consequently,  if $g$, $[\xx]$, and $\Lambda$  are as  in \eqref{gxlambda}, then  
$[\xx] \in\WA_{M_\form}(\psi) \smallsetminus (\P^d_\Q \cap M_\form)$ if and only if
\[
\Lambda \in \WA(r_\psi,\Omega_{\nform,\Lambda_*}):= \bigcap_{\epsilon > 0}\A(\epsilon r_\psi,\Omega_{\nform,\Lambda_*})
\]
\end{corollary}
\begin{proof}
Given $g\in\orth(\nform)$ and $[\xx],\Lambda$ as in \eqref{gxlambda}, write $C([\xx])$ for the left hand side of \eqref{correspondence2} and write $C(\Lambda)$ for the right hand side of \eqref{correspondence2}. Then
\[
C([\xx]) < \alpha \;\;\Rightarrow\;\; [\xx]\in \A_{M_\form}(\alpha\psi) \;\;\Rightarrow\;\; C([\xx]) \leq \alpha
\]
and
\[
C(\Lambda) < 1 \;\;\Rightarrow\;\; \Lambda\in \A(r_\psi,\Omega_{\nform,\Lambda_*}) \;\;\Rightarrow\;\; C(\Lambda) \leq 1.
\]
The conclusion follows. The ``consequently" part follows from the regularity of $\psi$ and the elementary computation $r_{\epsilon\psi}(t)  = e^{-t} \psi^{-1}(e^{-t}/\epsilon)$.
\end{proof}

In applying the correspondence principle, 
the following observations happen to be useful:

\begin{observation}
\label{observationcompact}
There exists a compact set $\KK\subset\orth(\nform)$ such that $\pi_1(\KK) = M_\form$.
\end{observation}
\begin{proof}
This follows from the facts that $M_\form$ is compact, $\orth(\nform)$ is locally compact, and $\pi_1$ is open and surjective.
\end{proof}

We remark that the corresponding assertion is not true for $\pi_2$, since $\Omega_{\nform,\Lambda_*}$ is not compact by Corollary \ref{corollarymahlerquadratic}.

Now let $\mu_\nform$ and $\mu_{\nform,\Lambda_*}$ denote   Haar measures on $\orth(\nform)$ and $\Omega_{\nform,\Lambda_*}$, respectively.

\begin{observation}
\label{observationlebesgue}
The measures\Footnote{Note that the measures $\pi_1[\mu_\nform]$ and $\pi_2[\mu_\nform]$ are not $\sigma$-finite; in fact, they are $\{0,\infty\}$-valued.} $\lambda_{M_\form}$ and $\pi_1[\mu_\nform]$ are mutually absolutely continuous. The measures $\mu_{\nform,\Lambda_*}$ and $\pi_2[\mu_\nform]$ are mutually absolutely continuous.
\end{observation}

We remark that Corollary \ref{corollaryBAboundedorbits}, the ergodicity of the $g_t$-action on $\Omega_{R,\Lambda_*}$, and the above observation allow one to conclude that the set $\BA_{M_\form}$ is $\lambda_{M_\form}$-null. This is a special case of a more general Khintchine-type result -- namely  Theorem \ref{theoremkhinchinquadratic}.


\section{A Dirichlet-type theorem}
\label{sectionproofdirichlet}

In this section we prove the following:
\begin{theorem}[Dirichlet-type theorem for quadric hypersurfaces]
\label{theoremdirichletquadratic}
Fix $d\geq 2$, and let $M_\form$ be a nonsingular rational quadric hypersurface in $\P_\R^d$ 
with $\Qrank \ge 1$. Then:
\begin{itemize}
\item[(i)] $\psi_1$ is Dirichlet for intrinsic approximation on $M_\form$.
\item[(ii)] $\psi_1$ is uniformly Dirichlet if and only if $\Qrank = \Rrank$.
\item[(iii)] The following are equivalent:
\begin{itemize}
\item[(A)] $\Qrank = \Rrank = 1$.
\item[(B)] There exist $C,\numberq_0 > 0$ such that for all $[\xx]\in M_\form$ and for all $\numberq\geq\numberq_0$ there exists $[\rr]\in \P_\Q^d\cap M_\form$ such that
\begin{equation}
\label{strongdirichlet}
H_\std([\rr]) \leq \numberq\text{ and }\dist([\rr],[\xx]) \leq \frac{C}{\sqrt{H_\std([\rr]) \numberq}}\cdot
\end{equation}
\item[(C)] The set
\[
\{[\xx]\in M_\form : \exists\, C,\numberq_0 > 0 \all\, \numberq\geq\numberq_0 \;\; \exists\, [\rr]\in \P_\Q^d\cap M_\form \text{ satisfying 
\eqref{strongdirichlet}}\}
\]
has positive $\lambda_{M_\form}$-measure.
\end{itemize}
\end{itemize}
\end{theorem}

Except for the forward direction of (ii) (i.e.\ uniformly Dirichlet implies $\Qrank = \Rrank$), which we will prove separately (see p.\pageref{pageuniformlydirichlet}), all of these results are consequences of the following theorem together with the correspondence principle,\Footnote{However, the correspondence principle cannot be used to deduce Theorem \ref{theoremdirichletquadratic2} from Theorem \ref{theoremdirichletquadratic} (or similarly, Theorem \ref{theoremkhinchinquadratic2} from Theorem \ref{theoremkhinchinquadratic}), due to the lack of an analogue of Observation \ref{observationcompact} for $\pi_2$. Similar considerations prevent the forwards direction of Theorem \ref{theoremdirichletquadratic}(ii) from being deduced from an appropriate analogue in the space of lattices.} namely Lemma \ref{lemmacorrespondence1}(i,ii) and Observations \ref{observationcompact} and \ref{observationlebesgue}. Details are left to the reader.

\begin{theorem}
\label{theoremdirichletquadratic2}
Fix $d\geq 2$, let $\nform$ be a nonsingular  quadratic form on $\R^{d + 1}$ with  $\Qrank\geq 1$ which is $\Qrank$-normalized. Fix $\Lambda_*\in\Omega_\nform$ commensurable to $\Z^{d + 1}$. Then:
\begin{itemize}
\item[(i)] For all $\Lambda\in\Omega_{\nform,\Lambda_*}$, there exists $C_\Lambda > 0$ such that infinitely many $\pp\in\Lambda\cap L_\nform\smallsetminus\{\0\}$ satisfy
\begin{equation}
\label{psi1dirichlet2}
\dist(\pp,\LL_1) \leq C_\Lambda.
\end{equation}
\item[(ii)] If $\Qrank = \Rrank$, then the constant $C_\Lambda$ in \eqref{psi1dirichlet2} can be made independent of $\Lambda$.
\item[(iii)] The following are equivalent:
\begin{itemize}
\item[(A)] $\Qrank = \Rrank = 1$.
\item[(B$'$)] There exist $C,\numberq_0 > 0$ such that for all $\Lambda\in\Omega_{\nform,\Lambda_*}$ and for all $\numberq\geq\numberq_0$ there exists $\pp\in \Lambda\cap L_\nform\smallsetminus\{\0\}$ with $\|\pp\| \leq \numberq$ such that
\begin{equation}
\label{strongdirichlet2}
\dist(\pp,\LL_1) \leq C\sqrt{\frac{\|\pp\|}{\numberq}}\cdot
\end{equation}
\item[(C$'$)] The set
\[
\{\Lambda\in\Omega_{\nform,\Lambda_*} : \exists\, C,\numberq_0 > 0 \all\, \numberq\geq\numberq_0 \;\; \exists\, \pp\in \Lambda\cap L_\nform\smallsetminus\{\0\} \text{ satisfying $\|\pp\| \leq \numberq$ and \eqref{strongdirichlet2}}\}
\]
has positive $\mu_{\nform,\Lambda_*}$-measure.
\end{itemize}
\end{itemize}
\end{theorem}

\begin{proof}[Proof of \text{(i)}]
We require the following preliminary result:

\begin{lemma}
\label{lemmadense}
Let $\form$ be a nonsingular quadratic form on $\R^{d + 1}$, and fix $\Lambda\in\Omega_\form$ satisfying $\Lambda\cap L_\form\smallsetminus\{\0\}\neq\emptyset$. Then
\[
\Span(\Lambda\cap L_\form) = \R^{d + 1}.
\]
\end{lemma}
\begin{subproof}
After applying a matrix (namely one whose columns form a basis of $\Lambda$), we may without loss of generality assume that $\Lambda = \Z^{d + 1}$.  
The assumption $\Lambda\cap L_\form \smallsetminus \{\0\}\neq\emptyset$ will then imply that $\Qrank\geq 1$, and, by applying 
Proposition \ref{propositionrenormalization}, we may without loss of generality assume that $\form$ is 1-normalized and $\Lambda$ is commensurable with $\Z^{d + 1}$.
Then clearly \eq{rational}{r_0\ee_0,r_d\ee_d\in\Lambda\cap L_\form\text{ for some nonzero }r_0,r_d\in\Q.} On the other hand, for each $i = 1,\ldots,d - 1$, we have
\[
\ee_i + \form(\ee_i)\ee_0 - \ee_d \in 
L_\form
\]
by direct calculation. Since $\Lambda$ is $Q$-arithmetic and commensurable with $\Z^{d + 1}$, 
it follows that $$r_i\big(\ee_i + \form(\ee_i)\ee_0 - \ee_d \big)\in \Lambda$$ for some nonzero $r_i\in\Q$; hence, in view of  \equ{rational}, $\ee_i\in \Span(\Lambda\cap L_\form)$.
\end{subproof}

For $t\geq 0$, let $g_t\in\orth(\nform)$ be as in \eqref{gtdef2}. Applying Corollary \ref{corollarymahlerquadratic} to the lattices $(g_t \Lambda)_{t\geq 0}$, we see that one of the following two cases holds:
\begin{itemize}
\item[Case 1:] There exists a sequence $t_n\to\infty$ and a sequence $g_{t_n}(\Lambda\cap L_\nform)\ni g_{t_n}(\pp_n)\to 0$. In this case, for all sufficiently large $n$, \eqref{gt} implies that $\pp_n$ satisfies \eqref{psi1dirichlet2}. If the set $\{\pp_n : n\in\N\}$ is infinite, this completes the proof. Otherwise, there exists $\pp\in\Lambda$ such that $\pp_n = \pp$ for arbitrarily large $n$. In particular, $g_{t_{n_k}}(\pp) \to 0$ for some increasing sequence $(n_k)_1^\infty$. Comparing with \eqref{gtdef2}, we see that $\pp\in\LL_1$. Since the vectors $n\pp$ ($n\in\Z$) all satisfy \eqref{psi1dirichlet2}, this completes the proof.
\item[Case 2:] There exists a sequence $t_n\to\infty$ such that $g_{t_n}\Lambda \to \w\Lambda \in \Omega_{\nform,\Lambda_*}$. In this case, by Lemma \ref{lemmadense} we have $\w\Lambda\cap L_\nform\nsubset \LL_1^\perp$, where $\LL_1^\perp$ denotes the set of vectors $\form$-orthogonal to $\ee_1$ as in Definition \ref{def_orthogonal}. Thus we may fix $\w\pp\in\w\Lambda\,\cap\, L_\nform\smallsetminus \LL_1^\perp$. Since $g_{t_n}\Lambda \to \w\Lambda$, there is a sequence $g_{t_n}\Lambda\ni g_{t_n}(\pp_n)\to \w\pp$. Let $C_\Lambda = 2\|\w\pp\|$; then for all sufficiently large $n$, \eqref{gt} implies that $\pp_n$ satisfies \eqref{psi1dirichlet2}. If the set $\{\pp_n : n\in\N\}$ is infinite, this completes the proof. Otherwise, there exists $\pp\in\Lambda$ such that $\pp_n = \pp$ for arbitrarily large $n$. In particular, $e^{t_{n_k}} \dist(\pp,\LL_1^\perp) \to \dist(\w\pp,\LL_1^\perp)\neq 0$ for some increasing sequence $(n_k)_1^\infty$. This is clearly a contradiction.
\qedhere\end{itemize}
\end{proof}
\begin{proof}[Proof of \text{(ii)}]
We first need to define the codiameter of a discrete subgroup.
\begin{definition}
\label{definitioncodiam}
The \emph{codiameter} of a discrete 
subgroup $\Gamma \subgp \R^{d+1}$, written $\Codiam(\Gamma)$, is the diameter of the quotient space $\Span(\Gamma)/\Gamma$. 
\end{definition}
We require the following lemma. 
\begin{replemma}{lemmaCnew}
There exists $C_1 > 0$ such that for every $\Lambda\in \Omega_{\nform,\Lambda_*}$, there exists a totally isotropic $\Lambda$-rational\Footnote{A subspace $V\subgp\R^{d + 1}$ is \emph{$\Lambda$-rational} if $\Span(\Lambda\cap V) = V$.} subspace $V\subgp\R^{d + 1}$ of dimension $\Qrank$ satisfying $\Codiam(V\cap\Lambda)\leq C_1$.
\end{replemma}
The proof of Lemma \ref{lemmaCnew} requires reduction theory, so we delay its proof until Section \ref{sectionreductiontheory}.

\smallskip

Let $C_1$ be as in Lemma \ref{lemmaCnew}. Fix $\Lambda\in\Omega_{\nform,\Lambda_*}$. For each $t\geq 0$, applying Lemma \ref{lemmaCnew} to the lattice $g_t\Lambda\in\Omega_{\nform,\Lambda_*}$ yields a totally isotropic $g_t\Lambda$-rational subspace $V_t\subgp\R^{d + 1}$ of dimension $\Qrank$ satisfying
\begin{equation}
\label{codiam}
\Codiam(V_t\cap g_t\Lambda)\leq C_1.
\end{equation}
At this point we divide the proof into two cases:
\begin{itemize}
\item[Case 1:] $\LL_1\subgp V_t$ for some $t\geq 0$. In this case, since the set $S := \{\xx\in g_{-t}(V_t) : \dist(\xx,\LL_1) \leq C_1\}$ has infinite volume in the vector space $g_{-t}(V_t)$, by Minkowski's theorem it contains infinitely many lattice points $\pp\in \Lambda\cap S$. Note that each such $\pp$ is in $L_\nform$ since $V_t$ is totally isotropic. On the other hand, \eqref{psi1dirichlet2} is clearly satisfied (with $C_\Lambda = C_1$ independent of $\Lambda$). This completes the proof.
\item[Case 2:] $\LL_1\nsubset V_t$ for all $t\geq 0$. Fix $t\geq 0$. Note that if $V_t\subgp \LL_1^\perp$, then $V_t + \LL_1$ is a totally isotropic vector space of dimension $\Qrank + 1 = \Rrank + 1 > \Rrank$, a contradiction. Thus $V_t\nsubset\LL_1^\perp$. Fix a unit vector $\vv_t \in V_t$ which is perpendicular to $V_t\cap \LL_1^\perp$ with respect to the Euclidean quadratic form $\EE_{d + 1} = \sum_0^d x_i^2$. By \eqref{codiam}, there exists $g_t(\pp_t) \in V_t\cap g_t\Lambda$ satisfying $\|g_t(\pp_t) - 2C_1 \vv_t\| \leq C_1$. \eqref{gt} implies that $\pp_n$ satisfies \eqref{psi1dirichlet2}, with $C_\Lambda = 3C_1$ independent of $\Lambda$. If the set $\{\pp_t : t\geq 0\}$ is infinite, this completes the proof. Otherwise, there exists $\pp\in\Lambda$ such that $\pp_t = \pp$ for arbitrarily large $t$. However, for all $t$ we have $g_t(\pp_t)\in V_t\smallsetminus(V_t\cap\LL_1^\perp) = V_t\smallsetminus\LL_1^\perp$, and thus $\pp\notin\LL_1^\perp$. This implies that $\|g_t(\pp)\|\to\infty$, a contradiction.
\qedhere\end{itemize}
\end{proof}
\begin{proof}[Proof of \text{(iii)}]
For the purpose of this proof, we introduce a new system of coordinates on $\R^{d + 1}$. For $\xx\in\R^{d + 1}$ let
\begin{align*}
H(\xx) &= |x_0|\\
W(\xx) &= \|(x_1,\ldots,x_{d - 1})\|\\
L(\xx) &= |x_d|.
\end{align*}
We will think of the letters $H$, $W$, and $L$ as being short for ``height'', ``width'', and ``length'', respectively. Note that for $t\in\R$,
\begin{align*}
H(g_t\xx) &= e^{-t} H(\xx)\\
W(g_t\xx) &= W(\xx)\\
L(g_t\xx) &= e^t L(\xx).
\end{align*}
In other words, for $t\geq 0$, applying $g_t$ decreases height and increases length while leaving width fixed. Moreover,
\begin{align*}
\|\xx\| &= \max\big(H(\xx),W(\xx),L(\xx)\big)\\
\dist(\xx,\LL_1) &= \max\big(W(\xx),L(\xx)\big).
\end{align*}
If $\xx\in L_\nform$, then
\begin{equation}
\label{Linequality}
H(\xx) L(\xx) = |\w\nform(x_1,\ldots,x_{d - 1})| \leq \|\w\nform\| W^2(\xx),
\end{equation}
where $\w\nform$ is the remainder of $\nform$.

We will now rephrase the Diophantine condition on a lattice $\Lambda\in\Omega_{\nform,\Lambda_*}$ described in (B$'$) and (C$'$) of Theorem \ref{theoremdirichletquadratic2}(iii) as a dynamical condition on the same lattice $\Lambda$. Precisely,

\begin{observation}
Fix $C,\numberq_0 \geq 1$ with $\numberq_0 > C^2$, and fix $\Lambda\in\Omega_{\nform,\Lambda_*}$. Then \text{(1) \implies (2) \implies (3)}:
\begin{itemize}
\item[(1)] For all $t\geq \frac12 \log(\numberq_0)$, there exists $\qq\in g_t\Lambda\cap L_\nform\smallsetminus\{\0\}$ satisfying $\|\qq\| \leq C$ and $W(\qq) \leq \sqrt{C H(\qq)}$.
\item[(2)] For all $\numberq\geq\numberq_0$, there exists $\pp \in \Lambda\cap L_\nform\smallsetminus\{\0\}$ with $\|\pp\|\leq\numberq$ satisfying \eqref{strongdirichlet2}.
\item[(3)] For all $t\geq\log(\numberq_0)$, there exists $\qq\in g_t\Lambda\cap L_\nform\smallsetminus\{\0\}$ satisfying $\|\qq\|\leq C^2\max(1,\|\w\nform\|)$ and $W(\qq) \leq C\sqrt{H(\qq)}$.
\end{itemize}
\end{observation}
\begin{subproof}[\text{(1) \implies (2)}]
Fix $\numberq\geq\numberq_0$, and let $t = \log(\numberq/C) \geq \frac12 \log(\numberq_0)$. Let $\qq\in g_t\Lambda\cap L_\nform\smallsetminus\{\0\}$ be as in (1), and let $\pp = g_{-t}(\qq) \in \Lambda\cap L_\nform\smallsetminus\{\0\}$. Then
\[
\|\pp\| \leq e^t \|\qq\| \leq \frac{\numberq}{C}C = \numberq.
\]
To demonstrate \eqref{strongdirichlet2}, we bound $W(\pp)$ and $L(\pp)$. First of all,
\begin{equation}
\label{Wbound}
W(\pp) = W(\qq) \leq \sqrt{C H(\qq)} = \sqrt{C \frac{H(\pp)}{\numberq/C}} = C \sqrt\frac{H(\pp)}{\numberq}.
\end{equation}
On the other hand, we have
\[
L(\pp) = \frac{L(\qq)}{\numberq/C} \leq \frac{C}{\numberq/C} = \frac{C^2}{\numberq},
\]
which implies
\[
L(\pp) = \sqrt{L(\pp)}\sqrt{L(\pp)} \leq \sqrt{L(\pp)}\sqrt\frac{C^2}{\numberq} = C\sqrt\frac{L(\pp)}{\numberq}.
\]
Combining with \eqref{Wbound} demonstrates \eqref{strongdirichlet2}.
\end{subproof}
\begin{subproof}[\text{(2) \implies (3)}]
Fix $t\geq\log(\numberq_0)$, and let $\numberq = e^t\geq\numberq_0$. Let $\pp \in \Lambda\cap L_\nform\smallsetminus\{\0\}$ be as in (2), and let $\qq = g_t(\pp) \in g_t\Lambda\cap L_\nform\smallsetminus\{\0\}$. Then
\[
H(\qq) = e^{-t}H(\pp) \leq e^{-t} \numberq = 1.
\]
On the other hand, \eqref{strongdirichlet2} is written in terms of height, width, and length as
\[
\max\big(W(\pp), L(\pp)\big) \leq C \sqrt\frac{\max\big(H(\pp),W(\pp),L(\pp)\big)}{\numberq},
\]
and since $\numberq\geq \numberq_0 > C^2$, the case where the maximum is $W(\pp)$ or $L(\pp)$ cannot occur. Thus
\[
\max\big(W(\pp), L(\pp)\big) \leq C \sqrt\frac{H(\pp)}{\numberq}.
\]
In particular,
\[
W(\qq) = W(\pp) \leq C \sqrt\frac{H(\pp)}{e^t} = C\sqrt{H(\qq)}.
\]
Since $\qq\in L_\nform$, \eqref{Linequality} gives
\[
L(\qq) \leq \|\w\nform\|\frac{W^2(\qq)}{H(\qq)} \leq \|\w\nform\|\frac{C^2 H(\qq)}{H(\qq)} = C^2 \|\w\nform\|.
\]
Thus $\|\qq\| = \max\big(H(\qq), W(\qq), L(\qq)\big) \leq \max(1,C,C^2\|\nform\|) \leq C^2\max(1,\|\nform\|)$.
\end{subproof}

For each $C > 0$ consider the set
\[
\FF_C := \{\Lambda\in\Omega_{\nform,\Lambda_*} : \exists\, \qq\in \Lambda\cap L_\nform\smallsetminus\{\0\} \text{ such that } \|\qq\| \leq C, \;\; W(\qq) \leq \sqrt{C H(\qq)} \}.
\]
Then (B$'$) and (C$'$) of Theorem \ref{theoremdirichletquadratic2}(iii) are equivalent to the following conditions, respectively:
\begin{itemize}
\item[(B$''$)] There exists $C > 0$ such that for all $\Lambda\in\Omega_{\nform,\Lambda_*}$ and for all $t\geq C$, $g_t\Lambda\in\FF_C$.
\item[(C$''$)] The set
\[
\{\Lambda\in\Omega_{\nform,\Lambda_*} : \exists\, C > 0 \all\, t\geq C \;\; g_t\Lambda\in\FF_C\} = \bigcup_{C > 0}\liminf_{t\to\infty}g_{-t}(\FF_C)
\]
has positive $\mu_{\nform,\Lambda_*}$-measure.
\end{itemize}
Now (B$''$) is clearly equivalent to the following:
\begin{itemize}
\item[(B$'''$)] There exists $C > 0$ such that $\FF_C = \Omega_{\nform,\Lambda_*}$.
\end{itemize}
We claim that (C$''$) is also equivalent to (B$'''$). Indeed, it is clear that (B$'''$) implies (C$''$). Conversely, if (C$''$) holds, then by Moore's ergodicity theorem \cite[Theorem III.2.1]{BekkaMayer},\Footnote{If $d = 3$ and $\Rrank = 2$, then the group $G = \orth(\form)$ is not simple (being isomorphic to $\orth(2,2)$), so one should use \cite[Theorem III.2.5]{BekkaMayer} rather than \cite[Theorem III.2.1]{BekkaMayer}. Note that the fact that the group $(g_t)_{t\in\R}$ is totally noncompact in $G$ follows from the inequality $(\pi_i)'(\zz)\neq\0$ proven on p.\pageref{pagerefpiz} of the present paper.} the set $\FF_C$ has full $\mu_{\nform,\Lambda_*}$ measure, where $C$ is large enough so that the $(g_t)$-invariant set $\liminf_{t\to\infty} g_{-t}(\FF_C)$ has positive measure. But since $\FF_C$ is closed, this implies (B$'''$).

\smallskip
To complete the proof, we must show that (B$'''$) is equivalent to (A).
\begin{subproof}[Proof of \text{(A) \implies (B$'''$)}]
Since $\Rrank = 1$, the remainder $\w\nform$ does not represent zero over $\R$, i.e.\ it is either positive definite or negative definite. Without loss of generality suppose that it is positive definite. Then $\sqrt{\w\nform}$ is a norm on $\R^{d - 1}$, so there exists $K > 0$ such that
\[
\w\nform(\xx) \geq \frac{1}{K}W^2(\xx) \all\, \xx\in\R^{d - 1}.
\]
Then for all $\xx\in L_\nform$,
\begin{equation}
\label{Linequality2}
W^2(\xx) \leq K\w\nform(x_1,\ldots,x_{d - 1}) = -Kx_0 x_d = K H(\xx) L(\xx),
\end{equation}
providing an asymptotic converse to \eqref{Linequality}.

Let $C_1 > 0$ be as in Lemma \ref{lemmaCnew}. Fix $\Lambda\in\Omega_{\nform,\Lambda_*}$, and we will show that $\Lambda\in\FF_{C_1 K}$. Indeed, by Lemma \ref{lemmaCnew} there exists $\qq\in \Lambda\cap L_\nform\smallsetminus\{\0\}$ satisfying $\|\qq\| \leq C_1$. Then \eqref{Linequality2} gives
\[
W(\qq) \leq \sqrt{K H(\qq) L(\qq)} \leq \sqrt{C_1 K H(\qq)},
\]
demonstrating that $\Lambda\in\FF_{C_1 K}$.
\end{subproof}
\begin{subproof}[Proof of \text{(B$'''$) \implies (A)}]

\begin{claim}
We may without loss of generality\Footnote{Here we abandon the assumption that $\Lambda_*$ is commensurable to $\Z^{d + 1}$.} suppose that $\nform$ is $\Rrank$-normalized and that $\Lambda_*\cap \LL_\Qrank = \Z^{d + 1}\cap\LL_\Qrank$.
\end{claim}
\begin{subproof}
Let $E_\Q$ be a $\Lambda_*$-rational totally isotropic subspace of $\R^{d + 1}$ of dimension $\Qrank$. Let $E_\R\supset E_\Q$ be a totally isotropic subspace of $\R^{d + 1}$ of dimension $\Rrank$. By Proposition \ref{propositionrenormalization}, there is a matrix $\matrix_1\in\GL_{d + 1}(\R)$ such that $\nform' := \nform\circ\matrix_1$ is $\Rrank$-normalized and $\matrix_1^{-1}(E_\R) = \LL_\Rrank$. In particular, $\Gamma := \matrix_1^{-1}(\Lambda_*\cap E_\Q)\subset \LL_\Rrank$. Let $\matrix_2\in\GL_\Rrank(\R)$ send $\Gamma$ to $\Z^{d + 1}\cap\LL_\Qrank$. Let $g_{\matrix_2}$ be defined by the equation \eqref{gmatrix}, so that $g_{\matrix_2}\in \orth(\nform')$. Then $g_{\matrix_2}^{-1}(\Gamma) = \Z^{d + 1}\cap\LL_\Qrank$. Letting $\matrix = \matrix_1\circ g_{\matrix_2}^{-1}$, we have $\matrix^{-1}(\Lambda_*\cap E_\Q) = \Z^{d + 1}\cap\LL_\Qrank$, or equivalently $\matrix^{-1}(\Lambda_*)\cap \LL_\Qrank = \Z^{d + 1}\cap\LL_\Qrank$. Let $\Lambda_*' = \matrix^{-1}(\Lambda_*)$, and observe that $\nform' = \nform\circ\matrix$. Then $\nform'$ is $\Rrank$-normalized and $\Lambda_*'\cap \LL_\Qrank = \Z^{d + 1}\cap\LL_\Qrank$. On the other hand, both conditions (A) and (B$'''$) are unaffected by replacing $\nform$ and $\Lambda_*$ with $\nform'$ and $\Lambda_*'$, respectively.
\end{subproof}

Now suppose (A) fails, i.e.\ $\Rrank > 1$. Fix $t\geq 0$, and let $\tt = (t,\ldots,t)\in\R^\Qrank$. Then
\[
\Lambda_t := g_\tt\Lambda_* \in \Omega_{\nform,\Lambda_*}.
\]
\begin{claim}
\label{claimLambdat}
If $\pp\in\Lambda_t\cap L_\nform$ satisfies $\|\pp\| < e^t/(2\|\nform\|)$, then $\pp\in \Gamma_t := \Lambda_t \cap \LL_\Qrank$.
\end{claim}
\begin{subproof}
For each $i = 0,\ldots,\Qrank - 1$, we have $\ee_i\in\Z^{d + 1}\cap\LL_\Qrank \subset\Lambda_*\cap L_\nform$, and thus $g_\tt(\ee_i) = e^{-t}\ee_i \in\Lambda_t\cap L_\nform$. Since $\Lambda_t$ is $\nform$-arithmetic, we have
\begin{equation}
\label{BinZ}
B_\nform\big(\pp,g_\tt(\ee_i)\big) \in \frac{\Z}{2}\cdot
\end{equation}
On the other hand,
\[
\big|B_\nform\big(\pp,g_\tt(\ee_i)\big)\big| \leq \|\nform\| \cdot \|\pp\| \cdot \|g_\tt(\ee_i)\| < \|\nform\| \left(\frac{e^t}{2\|\nform\|}\right) e^{-t} = \frac12\cdot
\]
Combining with \eqref{BinZ}, we see that
\[
B_\nform\big(\pp,g_\tt(\ee_i)\big) = e^{-t} B_\nform(\pp,\ee_i) = 0.
\]
It follows that the $\Lambda_t$-rational subspace $\LL_\Qrank + \R\pp$ is totally isotropic, and so by the maximality of $\Qrank$, we have $\pp\in\LL_\Qrank$.
\end{subproof}
Now let $\matrix_t\in\orth(\EE_\Rrank)$ satisfy $\matrix_t(\Gamma_t)\cap\LL_1 = \{\0\}$, where $\EE_\Rrank$ is the Euclidean metric on $\R^\Rrank$. Such a choice is possible since by assumption $\Rrank > 1$. Let $g_{\matrix_t}$ be given by \eqref{gmatrix}, so that $g_{\matrix_t}\in\orth(\nform)\cap \orth(\EE_{d + 1})$. Let $\Lambda_t' = g_{\matrix_t}\Lambda_t$.

Let
\[
\gamma = \left[\begin{array}{ccc}
&& 1\\
& I_{d - 1} &\\
1 &&
\end{array}\right],
\]
so that
\[
\gamma(\FF_C) = \{\Lambda\in\Omega_{\nform,\Lambda_*} : \exists\, \qq\in \Lambda\cap L_\nform\smallsetminus\{\0\} \;\; \|\qq\| \leq C, \;\; W(\qq) \leq \sqrt{C L(\qq)} \}.
\]
We claim that for all $C > 0$, there exists $t\geq 0$ such that $\Lambda_t'\notin\gamma(\FF_C)$; in particular $\FF_C\propersubset \Omega_{\nform,\Lambda_*}$. Indeed, fix $C$ and $t$, and suppose we have $\qq = g_{\matrix_t}(\pp)\in\Lambda_t'\cap L_\nform\smallsetminus\{\0\}$ with $\|\pp\| \asymp_\times \|\qq\|\leq C$ and $W(\qq) \leq \sqrt{C L(\qq)}$. If $t$ is large enough (depending on $C$), then by Claim \ref{claimLambdat} we have $\pp\in\Gamma_t$ and thus $\qq\in\LL_\Rrank\smallsetminus\LL_1$. In particular, $L(\qq) = 0$ but $W(\qq) > 0$. This is a contradiction. Thus $\FF_C\propersubset \Omega_{\nform,\Lambda_*}$ for all $C > 0$, so (B$'''$) fails.
\end{subproof}
This completes the proof of Theorem \ref{theoremdirichletquadratic2}.
\end{proof}

We complete the proof of Theorem \ref{theoremdirichletquadratic} by demonstrating the forwards direction of (ii).

\begin{proof}[Proof of Theorem \ref{theoremdirichletquadratic}, forwards direction of \text{(ii)}]\label{pageuniformlydirichlet}
Let $V_\Q$ be a maximal isotropic $\Q$-subspace of $\R^{d + 1}$, and let $V_\R$ be a maximal isotropic $\R$-subspace of $\R^{d + 1}$ such that $V_\Q\propersubset V_\R$. Then $[V_\Q]\propersubset [V_\R]$. By contradiction, suppose that $\psi_1$ is uniformly Dirichlet. This is equivalent to the existence of a constant $C > 0$ such that for all $[\xx]\in M_\form$, there exist infinitely many $\rr\in\Z^{d + 1}\cap L_\form$ satisfying
\begin{equation}
\label{drLxC}
\dist(\rr,\LL_{[\xx]}) \leq C,
\end{equation}
where $\LL_{[\xx]} = \R\xx$.

Fix $[\xx]\in [V_\R]\smallsetminus [V_\Q]\subset M_\form$. Since $[\xx]\notin [V_\Q]$, only finitely many $\rr\in V_\Q\cap\Z^{d + 1}$ can satisfy \eqref{drLxC}, so there exists $\rr\in\Z^{d + 1}\cap L_\form\smallsetminus V_\Q$ satisfying \eqref{drLxC}. Let $\xx$ be the projection of $\rr$ onto $\LL_{[\xx]}$, so that
\begin{equation}
\label{xrbound}
\|\xx - \rr\| = \dist(\rr,\LL_{[\xx]}) \leq C.
\end{equation}
Let $\bb_1,\ldots,\bb_\Qrank$ be a basis of $V_\Q\cap\Z^{d + 1}$. Since $V_\R$ is totally isotropic and $\xx\in V_\R$, we have $B_\form(\xx,\bb_i) = 0$ for all $i = 1,\ldots,\Qrank$. Thus
\begin{align*}
|B_\form(\rr,\bb_i)| &= |B_\form(\xx - \rr,\bb_i)|\\
&\leq \|B_\form\|\cdot\|\xx - \rr\|\cdot\|\bb_i\| \leq N := \left\lceil C\|B_\form\|\max_{i = 0}^{\Qrank - 1}\|\bb_i\| \right\rceil,
\end{align*}
and so since $\form$ is $\Z^{d + 1}$-arithmetic,
\[
\zz := \big(B_\form(\rr,\bb_i)\big)_{i = 0}^{\Qrank - 1} \in \{-N,\ldots,N\}^\Qrank.
\]
On the other hand, since $\rr\notin V_\Q$, the maximality of $V_\Q$ implies that $V_\Q + \R\rr$ is not isotropic (it is clearly a $\Q$-subspace). Thus $B_\form(\rr,\bb_i)\neq 0$ for some $i = 1,\ldots,\Qrank$, i.e.
\[
\zz\neq\0.
\]
Choose real numbers $c_1,\ldots,c_\Qrank$ linearly independent over $\Q$, and let $\ss = \sum_{i = 1}^\Qrank c_i\bb_i\in V_\Q$. Let $[\xx_m]\tendsto m[\ss]$ with $[\xx_m]\in [V_\R]\smallsetminus [V_\Q]$. For each $m$, let $\rr_m$, $\xx_m$, and $\zz_m$ be defined as above, with the additional stipulation that $\|\rr_m\| \geq m$ (this is possible since there were infinitely many possible choices for $\rr_m$). Then for each $m\in\N$ we have
\[
|B_\form(\rr_m,\ss)| = |\zz_m\cdot\cc|,
\]
where $\cc = (c_i)_{i = 0}^{\Qrank - 1}$. Thus
\[
|B_\form(\rr_m,\ss)| \in \big\{|\zz\cdot\cc| : \zz\in\{-N,\ldots,N\}^\Qrank\smallsetminus\{\0\}\big\},
\]
which implies $|B_\form(\rr_m,\ss)|\geq\epsilon$ for some $\epsilon > 0$ independent of $m$. Let $t_m = \pm\|\xx_m\|/\|\ss\|$; since $[\xx_m]\tendsto m [\ss]$ we have
\[
\left\|\ss - \frac{\xx_m}{t_m}\right\|\tendsto m 0
\]
after choosing the appropriate $\pm$ signs to define the $t_m$s. Now
\begin{align*}
\epsilon t_m &\leq |B_\form(\rr_m,t_m\ss)|\\
&= |B_\form(\rr_m - \xx_m,t_m\ss)| \since{$\xx_m,\ss\in V_\R$}\\
&\leq |B_\form(\rr_m - \xx_m,t_m\ss - \xx_m)| + |B_\form(\rr_m - \xx_m,\xx_m)| \noreason\\
&= |B_\form(\rr_m - \xx_m,t_m\ss - \xx_m)| + \frac{1}{2}\big|\form(\rr_m) - \form(\xx_m) - \form(\rr_m - \xx_m)\big| \noreason\\
&= |B_\form(\rr_m - \xx_m,t_m\ss - \xx_m)| + \frac{1}{2}\big|\form(\rr_m - \xx_m)\big| \since{$\rr_m,\xx_m\in L_\form$}\\
&\leq \|\form\|\cdot\|\rr_m - \xx_m\|\left[\|t_m\ss - \xx_m\| + \frac{1}{2}\|\rr_m - \xx_m\|\right] \noreason\\
&\leq C\|\form\|(C/2 + \|t_m\ss - \xx_m\|). \by{\eqref{xrbound}}
\end{align*}
Dividing by $t_m$ we have
\[
\epsilon \lesssim_\times \frac{1}{t_m} + \left\|\ss - \frac{\xx_m}{t_m}\right\| \tendsto m 0,
\]
a contradiction.
\end{proof}

\begin{remark}
\label{remarkquadraticsingular}
The hypothesis of nonsingularity can be dropped from parts (i) and (ii) of Theorem \ref{theoremdirichletquadratic}, if the hypothesis that $\P_\Q^d\cap M_\form\neq\emptyset$ is replaced by the stronger hypothesis that $\Z^{d + 1}$ intersects $L_\form\smallsetminus(\R^{d + 1})^\perp$.
\end{remark}
\begin{proof}
Any singular quadratic form is conjugate to a quadratic form $\form:\R^{d + 1}\to\R$ of the form
\[
\form(x_0,\ldots,x_d) = \w\form(x_0,\ldots,x_m),
\]
where $\w\form$ is a nonsingular quadratic form on $\R^{m + 1}$ for some $m < d$. In particular, $L_\form = L_{\w\form}\times\R^{d - m}$. Note that the hypothesis on $\form$ guarantees that $\P_\Q^m\cap M_{\w\form}\neq\emptyset$.

Fix $[\xx]\in M_\form$ and a representative $\xx = (\xx^{(1)},\xx^{(2)})\in L_\form$. Suppose first that $\xx^{(1)}\neq 0$, and let $\rr^{(1)}\in \Z^{m + 1}\cap L_{\w\form}$ be such that
\begin{equation}
\label{r1x1}
\dist(\rr^{(1)},\R\xx^{(1)})\leq C_{[\xx^{(1)}]}.
\end{equation}
Then there exists $t\in\R$ so that $\|\rr^{(1)} - t\xx^{(1)}\| \leq C_{[\xx^{(1)}]}$. Choose $\rr^{(2)}\in\Z^{d - m}$ so that $\|\rr^{(2)} - t\xx^{(2)}\| \leq 1$. Then
\begin{equation}
\label{quadraticsingular}
\|(\rr^{(1)},\rr^{(2)}) - t\xx\| \leq C_{[\xx^{(1)}]} + 1.
\end{equation}
Now by Theorem \ref{theoremdirichletquadratic}(i) applied to $\w\form$, there exist infinitely many $\rr^{(1)}\in\Z^{m + 1}\cap L_{\w\form}$ satisfying \eqref{r1x1}; thus there exist infinitely many pairs $(\rr^{(1)},\rr^{(2)})$ satisfying \eqref{quadraticsingular}.

On the other hand, if $\xx^{(1)} = 0$, let $\rr^{(1)} = 0$ and for each $t\in\R$ choose $\rr^{(2)}$ satisfying $\|\rr^{(2)} - t\xx^{(2)}\| \leq 1$; then \eqref{quadraticsingular} holds. Letting $t\to\infty$, there exist infinitely many pairs $(\rr^{(1)},\rr^{(2)})$ satisfying \eqref{quadraticsingular}.

Finally, if $\Qrank = \Rrank$, then by using Theorem \ref{theoremdirichletquadratic}(ii) in place of Theorem \ref{theoremdirichletquadratic}(i), the above argument shows that the implied constant is independent of $\xx$.
\end{proof}

\begin{remark}
The same technique cannot be used to remove the nonsingularity hypothesis from Theorem \ref{theoremkhinchinquadratic} below. Indeed, if we suppose that $[\xx^{(1)}]\in \A_{\psi,M_{\w\form}}$ for some $\psi$, then $C_{[\xx^{(1)}]}$ will be replaced by $C H_\std([\rr])\psi\circ H_\std([\rr])$ in \eqref{quadraticsingular}, but the second term (namely $1$) will not be changed. Thus the overall bound is no better than if we did not know that $[\xx^{(1)}]\in \A_{\psi,M_{\w\form}}$.
\end{remark}

\begin{remark}
\label{remarknondense}
The hypothesis that $M_\form$ is rational certainly cannot be dropped from Theorem \ref{theoremdirichletquadratic}. Indeed, Theorem \ref{theoremdirichletquadratic}(i) implies that the set $\P_\Q^d\cap M_\form$ is dense in $M_\form$ whenever $M_\form$ is a nonsingular rational quadric hypersurface in $\P_\R^d$ satisfying $\P_\Q^d\cap M_\form\neq \emptyset$. By contrast, if $\form$ is a quadratic form which is not a scalar multiple of any quadratic form with integer coefficients, then $\P_\Q^d\cap M_\form$ is not dense in $M_\form$.
\end{remark}
\begin{proof}
Let $\pi:\R\to\Q$ be a $\Q$-linear map, and let $\nform:\R^{d + 1}\to\R$ be the unique quadratic form so that $\nform = \pi\circ\form$ on $\Q^{d + 1}$. Then for $\rr\in\Q^{d + 1}$, $\form(\rr) = 0$ implies $\nform(\rr) = 0$; thus $\P_\Q^d\cap M_\form\subset M_{\nform}$. If $\P_\Q^d\cap M_\form$ is dense in $M_\form$, then $M_\form\subset M_{\nform}$, and so $\form$ is a scalar multiple of $\nform$. But $\nform$ has rational coefficients, and is therefore a scalar multiple of a quadratic form with integer coefficients.
\end{proof}

\section{Khintchine-type theorems and counting of rational points}
\label{sectionkhinchinquadratic}

Recall that in the classical setting, the convergence case of Khintchine's theorem follows directly from the Borel--Cantelli lemma combined with estimates for the number of rational points whose height is less than a fixed number $\numberq$. So in the case of intrinsic approximation one must find upper bounds on expressions of the form
\[
N_M(\numberq) := \#\big\{[\rr]\in \P_\Q^d\cap M: H_\std([\rr])\leq \numberq\big\},
\]
where $M\subset \P_\R^d$ is an arbitrary manifold. Such bounds have been considered extensively in the case where $M$ is algebraic in \cite{Browning_book}. We will pay special attention to the following result due to D.\ R.\ Heath-Brown. Recall that $\form$ is a rational quadratic form in $d+1$ variables, 
$\dim(M_\form) = d-1$, and $\form_0$ is the exceptional quadratic form on $\R^4$ defined in \eqref{form0}.

\begin{theorem}[{\cite[Theorems 5, 6, 7,  8 and remarks afterwards]{Heath-Brown_quadratic}}]
\label{theoremheathbrown}
Let $M_\form\subset\P_\R^d$ be a nonsingular rational quadric hypersurface 
with $\Qrank \ge 1$. Then
\begin{equation}
\label{heathbrown}
N_{M_\form}(\numberq) \asymp_\times
\begin{cases}
\numberq^{d-1} & \form\not\sim\form_0;\\
\numberq^2\log \numberq & \form\sim\form_0.
\end{cases}
\end{equation}
\end{theorem}

In order to clarify the relation between the above paraphrased version of Heath-Brown's results with  with the original theorems,  we make the following comments:
\begin{itemize}
\item[1.]  \cite[Theorems 5, 6, 7, and 8]{Heath-Brown_quadratic} provide asymptotics with an error term for the weighted sum 
$$
N(
F, w) =
N(F, w, P) :=
\sum_{\xx\in\Z^{d+1} \cap F^{-1}(0)}
w(P^{-1}\xx),$$  where $F$ is a rational quadratic form in $d+1$ variables, and $w$  a function on $\R^{d+1}$ which is required to be $\CC^\infty$. However to estimate $N_{M_\form}(\numberq)$ 
one must let $w = \one_{B(\0,1)}$. Since $w_0 = \one_{B(\0,1)}$ can be approximated from above and below by $\CC^\infty$ functions $w_n$ in a way such that the singular integrals $\sigma_\infty(F,w_n)$ approach $\sigma_\infty(F,w_0) \in (0,\infty)$ as $n\to\infty$, \cite[Theorems 5, 6, 7, and 8]{Heath-Brown_quadratic} will still hold for $w_0 = \one_{B(\0,1)}$, but  without an estimate on the error term; namely, we have
\[
\lim_{P\to\infty}\frac{N(F,w_0,P)}{\text{leading term}} = 1
\]
for each result in \cite{Heath-Brown_quadratic}. In Theorem \ref{theoremheathbrown} we have stated only the weaker conclusion that the left hand side is bounded from above and below (in limsup and liminf respectively).
\item[2.] According to \cite[Theorems 5, 6, 7, and 8]{Heath-Brown_quadratic}, the number of  integer vectors on quadric hypersurfaces $\form^{-1}(0)$ of $\R^{d+1}$ inside the ball of radius $\numberq$ is up to a multiplicative constant asymptotically equal to 
\begin{equation}
\label{heathbrown}
\begin{cases}
\numberq^{d-1} & \text{ if }d\ge 4\hskip .795in\text{ (Theorem 5)};\\
\numberq^2 & \text{ if }d= 3\text{ and }\form\not\sim\form_0\text{ (Theorem 6)};\\
\numberq^2\log \numberq & \text{ if }d = 3\text{ and }\form\sim\form_0 \text{ (Theorem 7)};
\\
\numberq \log \numberq & \text{ if }d= 2\hskip .8in \text{ (Theorem 8).}
\end{cases}
\end{equation}
Note however that our goal is to count rational points on  $M_\form$, which correspond to \emph{primitive} integer vectors on $Q^{-1}(0)$. The relation between counting primitive vectors and counting all lattice vectors is clarified in \cite{Heath-Brown_quadratic} after the theorems are stated. In particular, Theorems 5, 6 and 7 lead to equivalent results for
counting of primitive vectors,
which the only change is that the leading term is divided by
a constant. However the situation with Theorem 8 is different: in view of \cite[Corollary 2]{Heath-Brown_quadratic}, for the count of primitive integer vectors the factor $\log\numberq$ in the last line of \eqref{heathbrown} disappears.
\item[3.] In \cite{Heath-Brown_quadratic}, it is shown that the modified singular series $\sigma^*$ is positive and finite if and only if the equation $\form = 0$ has nontrivial solutions in every $p$-adic field. Since the forms we deal with satisfy $\P_\Q^d\cap M_\form\neq\emptyset$, the equation $\form = 0$ has nontrivial solutions over $\Q$, and so certainly over every $p$-adic field.
\end{itemize}

For any nonincreasing function $\psi:\N\to(0,\infty)$, we may write
\[
\A_{M_\form} (\psi) \subset \limsup_{\substack{\numberq\to\infty \\ \numberq\in 2^\N}} \bigcup_{\substack{[\rr]\in \P_\Q^d\cap M_\form \\ H_\std([\rr])\leq 2\numberq}} B\big([\rr],\psi(\numberq)\big).
\]
Combining with \eqref{heathbrown} and using the Hausdorff--Cantelli lemma \cite[Lemma 3.10]{BernikDodson}, one can immediately deduce the following corollary:

\begin{corollary}
\label{corollarykhinchinquadratic}
Let $M_\form\subset\P_\R^d$ be a nonsingular rational quadric hypersurface 
with $\Qrank \ge 1$. Fix a positive $s\leq {d-1}$, and let $\psi:\N\to(0,\infty)$ be nonincreasing. If the series
\begin{equation}
\label{convergence3}
\begin{cases}
\sum_{\numberq\in 2^\N} \numberq^{d-1} \psi^s(\numberq) & \form\not\sim\form_0 \\
\sum_{\numberq\in 2^\N} \numberq^2\log \numberq \psi^s(\numberq) & \form\sim\form_0
\end{cases}
\end{equation}
converges, then $\HH^s\big(\A_{M_\form} (\psi)\big) = 0$.
\end{corollary}
The case $s = d-1$ corresponds to Lebesgue measure.

\smallskip
Based on the above, one would expect that Khintchine's theorem for quadric hypersurfaces would state that the converse of Corollary \ref{corollarykhinchinquadratic} holds when $s = d-1$ (possibly with some additional assumptions on $\psi$). However, we instead have the following:


\begin{theorem}[Khintchine-type theorem for quadric hypersurfaces]
\label{theoremkhinchinquadratic}
Let $M_\form\subset\P_\R^d$ be a nonsingular rational quadric hypersurface with $\Qrank \ge 1$. Fix $\psi:\N\to(0,\infty)$, and suppose that $\psi$ is regular (see Definition \ref{definitionregular})
and that the function $q\mapsto q\psi(q)$ is nonincreasing.
Then $\A_{M_\form}(\psi)$ has full Lebesgue measure 
if 
the series
\begin{equation}\label{loglog}
\begin{cases}
\sum_{\numberq\in 2^\N} \numberq^{d-1} \psi^{d-1}(\numberq) & \form\not\sim\form_0 \\
\sum_{\numberq\in 2^\N} \numberq^2\log\log \numberq\, \psi^2(\numberq) & \form\sim\form_0
\end{cases}
\end{equation}
diverges; otherwise, $\A_{M_\form}(\psi)$ is Lebesgue null.
\end{theorem}

In other words, whenever $\form \not\sim \form_0$, the above intuition is correct: Theorem \ref{theoremkhinchinquadratic} then says that when $\form \not\sim \form_0$, the converse to the standard Borel--Cantelli argument holds for the collection of sets defining $\A_{M_\form}(\psi)$. On the other hand, the series \eqref{loglog} does not agree with \eqref{convergence3} when $\form \sim \form_0$, and so philosophically there is some nontrivial relation between the sets appearing in the definition of $\A_{M_{\form_0}}(\psi)$. A description of this relation is given in Section \ref{sectionspecialtype} (see in particular Remark \ref{remarknontrivialrelation}), where an elementary proof of the convergence case of Theorem \ref{theoremkhinchinquadratic} for the manifold $M_{\form_0}$ is given. 


\smallskip

Using the Mass Transference Principle of Beresnevich and Velani \cite[Theorem 2]{BeresnevichVelani}, one can immediately deduce the following:\Footnote{The dimension $s > 0$ may be replaced by a dimension function $f$; we omit the statement for brevity.}
\begin{theorem}[The Jarn\'ik--Besicovitch theorem for quadric hypersurfaces]
\label{theoremjarnikquadratic}
Fix $0 < s < d-1$. Let $\psi:\N\to(0,\infty)$ be regular, and suppose that $q\mapsto q^{d-1}\psi^s(q)$ is nonincreasing. If the series
\begin{equation}
\label{loglog2}
\begin{cases}
\sum_{\numberq\in 2^\N} \numberq^{d-1} \psi^s(\numberq) & \form\not\sim\form_0 \\
\sum_{\numberq\in 2^\N} \numberq^2\log\log \numberq\, \psi^s(\numberq) & \form\sim\form_0
\end{cases}
\end{equation}
diverges, then $\HH^s\big(\A_{M_\form} (\psi)\big) = \infty$.
\end{theorem}
This, in particular, computes the Hausdorff dimension of the set of $\psi_\cee$-approximable points of $M_\form$, see
\eqref{jarnikquadratic}.

It follows from Corollary \ref{corollarykhinchinquadratic} 
that for $\form \not\sim \form_0$, convergence of \eqref{loglog2} implies $\HH^s\big(\A_{M_\form}(\psi)\big) = 0$. However, in the case of the exceptional quadratic form $\form_0$, there is a discrepancy between \eqref{loglog2} and the series \eqref{convergence3} appearing in Corollary \ref{corollarykhinchinquadratic}, and the former may converge while the latter diverges. In this case, we do not know the value of $\HH^s\big(\A_{M_\form}(\psi)\big) $. However, the coarser Hausdorff dimension result \eqref{jarnikquadratic} holds regardless.
For reasons explained in Remark \ref{remarknontrivialrelation}, the authors conjecture that Theorem \ref{theoremjarnikquadratic} remains true if \eqref{loglog2} is replaced by \eqref{convergence3}.

Note also that
if $q^2\psi(q)\to 0$, then all $\psi$-good rational approximations of points in $M_\form$ are intrinsic, meaning that $\A_{M_\form}(\psi) = \A_d(\psi)\cap M_\form$ \cite[Lemma 4.1.1]{Drutu}. Consequently, for such $\psi$, Theorem \ref{theoremjarnikquadratic} may be rephrased in terms of ambient approximation. The rephrased result has been proven in the case $\form_\aff(\xx) = x_1^2 + x_2^2$ by Dickinson and Dodson \cite[Theorem 1]{DickinsonDodson}, and in the case where $\form\not\sim\form_0$ by Dru\c tu \cite[Theorem 4.5.7]{Drutu}.\Footnote{Although the hypothesis $\form\not\sim\form_0$ does not appear explicitly in Dru\c tu's theorem, it is required by her standing assumption that the lattice $\Gamma$ is irreducible (cf.\ \cite[\62.5,\64.5]{Drutu}), since when $\form\sim\form_0$, $\Gamma$ is reducible (see p.\pageref{pagerefpiz}). \label{footnotedrutu}}

Note that Theorem \ref{theoremkhinchinquadratic} is analogous to the main result of \cite{GorodnikShah}, the difference being that we are considering intrinsic approximation and the authors of \cite{GorodnikShah} are considering a specific type of extrinsic approximation. Also, it is likely that the techniques of  Dru\c tu \cite{Drutu} can be used to prove Theorem \ref{theoremkhinchinquadratic} in the case $\form\not\sim\form_0$ via the use of ubiquitous systems as considered in \cite{BDV}. 
On the other hand, Dru\c tu's methods do not apply to the exceptional quadric hypersurface $M_{\form_0}$ (cf. Footnote \ref{footnotedrutu}). 
We opt to use the machinery of Kleinbock and Margulis \cite{KleinbockMargulis} to establish Theorem \ref{theoremkhinchinquadratic}.

Theorem \ref{theoremkhinchinquadratic} can be deduced directly from the following theorem together with the correspondence principle (Corollary \ref{corollarycorrespondencesets} and Observation \ref{observationlebesgue}). As before, details are left to the reader.\Footnote{It is helpful to notice that the convergence/divergence of the series \eqref{loglog} is unaffected by the substitution $\psi\mapsto C\psi$, where $C > 0$ is a constant. Also, the fact that the assumption $q\psi(q) \to 0$ appears in Corollary \ref{corollarycorrespondencesets} but not Theorem \ref{theoremkhinchinquadratic} can be remedied by the observation that $\BA_{M_\form}$ has measure zero, which follows either from applying Theorem \ref{theoremkhinchinquadratic} to any function $\psi$ satisfying the hypotheses and such that the series \eqref{loglog} diverges, or by the argument at the end of Section \ref{sectioncorrespondence}.}

\begin{theorem}
\label{theoremkhinchinquadratic2}
Fix $d\geq 2$, let $\nform$ be a nonsingular $\Qrank$-normalized quadratic form on $\R^{d + 1}$, and fix $\Lambda_*\in\Omega_\nform$ commensurable to $\Z^{d + 1}$. Let $\psi:(0,\infty)\to(0,\infty)$ be a continuous function, and suppose that $q\mapsto q\psi(q)$ is nonincreasing. 
Let $r_\psi:(0,\infty)\to(0,\infty)$ and $\A_\nform(\psi) = \A(r_\psi,\Omega_{\nform,\Lambda_*})$ be defined as in Corollary \ref{corollarycorrespondencesets}, see  \eqref{defarpsi}. Then $\A_\nform(\psi)$ has full measure with respect to $\mu_{\nform,\Lambda_*}$ if 
\eqref{loglog} diverges; otherwise, $\A_\nform(\psi)$ is null with respect to $\mu_{\nform,\Lambda_*}$.
\end{theorem}

The proof of Theorem \ref{theoremkhinchinquadratic2} will occupy Sections \ref{sectionkhinchinproof} and \ref{sectionreductiontheory}. 

\section{Proof of Theorem \ref{theoremkhinchinquadratic2} modulo a volume computation}
\label{sectionkhinchinproof}

In the current section, we reduce Theorem \ref{theoremkhinchinquadratic2} to a statement about the asymptotic behavior of the measure $\mu_{\nform,\Lambda_*}$. Namely, we will deduce Theorem \ref{theoremkhinchinquadratic2} as a corollary of one of the main results of \cite{KleinbockMargulis}, which we now recall.

\begin{definition}
Let $(X,\dist_X)$ be a metric space, let $\mu$ be a (finite Borel) measure on $X$, and let $\Delta:X\to\R$ be a continuous function. For each $z\in\R$ let
\[
S_{\Delta,z} = \{x\in X: \Delta(x) \geq z\} \text{ and }\Phi_\Delta(z) = \mu(S_{\Delta,z}).
\]
$\Phi_\Delta$ is called the \emph{tail distribution function} of $\Delta$. We say that $\Delta$ is \emph{distance-like} if
\begin{itemize}
\item[(I)] $\Delta$ is uniformly continuous, and
\item[(II)] $\Phi_\Delta$ is regular (see Definition \ref{definitionregular}).
\end{itemize}
\end{definition}

Let $G$ be a connected semisimple center-free Lie group without compact factors, and let $\Gamma\subgp G$ be a lattice. By \cite[Theorem 5.22]{Raghunathan}, one can find connected normal subgroups $G_1,\ldots,G_\ell\leq G$ such that $G$ is the direct product of $G_1,\ldots,G_\ell$, $\Gamma_i := G_i\cap\Gamma$ is an irreducible lattice in $G_i$ for each $i = 1,\ldots,\ell$, and $\prod_{i = 1}^\ell \Gamma_i$ has finite index in $\Gamma$. Of course, if $\Gamma$ is irreducible, then we have $\ell = 1$, $G_1 = G$, and $\Gamma_1 = \Gamma$. Let $\pi_1,\ldots,\pi_\ell$ denote the projections from $G$ to the factors $G_i$.

\begin{theorem}[{\cite[Theorem 1.7(a)]{KleinbockMargulis}}]
\label{theoremkleinbockmargulisdiscrete}
Fix $G,\Gamma,G_1,\ldots,G_\ell$ as above. Let $\mfg$ denote the Lie algebra of $G$, and let $\zz\in\mfg$ be an element of a Cartan subalgebra of $\mfg$. Suppose that $(\pi_i)'(\zz)\neq \0$ for all $i = 1,\ldots,\ell$. (If $G$ is simple, this just amounts to saying that $\zz\neq\0$.) Let $X = G/\Gamma$, let $\mu_X$ be normalized Haar measure on $X$, let $\dist_G$ be a right-invariant Riemannian metric on $G$, let $\dist_X$ be the quotient of $\dist_G$ by $\Gamma$, and let $\Delta:X\to\R$ be a distance-like function.\Footnote{We remark that whether or not $\Delta$ is distance-like is independent of the choice of the right-invariant Riemannian metric $\dist_G$, since any two such metrics $\dist_1,\dist_2$ satisfy $\dist_1 \asymp_\times \dist_2$.} If $(z_t)_1^\infty$ is a sequence in $\R$, then
\begin{equation}
\label{kleinbockmargulisdiscrete}
\mu_X\big(\{x\in X: e^{t\zz}(x) \in S_{\Delta,z_t} \text{ for infinitely many $t\in\N$}\}\big) = \begin{cases}
0 \text{ if } \sum_{t = 1}^\infty\Phi_\Delta(z_t) < \infty\\
1 \text{ if } \sum_{t = 1}^\infty \Phi_\Delta(z_t) = \infty
\end{cases}.
\end{equation}
\end{theorem}

\begin{remark}
In \cite[Theorem 1.7(a)]{KleinbockMargulis}, $\Gamma$ is assumed to be irreducible, and $\zz$ is simply assumed to be a nonzero vector in $\mfa$. However, in \cite[\610.3]{KleinbockMargulis}, the authors of \cite{KleinbockMargulis} describe how to modify their proof to include the case where $\Gamma$ is reducible. Incorporating those modifications leads to the above theorem.
\end{remark}

For the purposes of this paper, it will be more convenient to deal with the following ``continuous'' version of Theorem \ref{theoremkleinbockmargulisdiscrete}:

\begin{theorem}
\label{theoremkleinbockmargulis}
Let $G,\Gamma,\mfa,\zz,X,\mu_X,\Delta$ be as in Theorem \ref{theoremkleinbockmargulisdiscrete}. If $z:(0,\infty)\to(0,\infty)$ is nondecreasing, then
\begin{equation}
\label{kleinbockmargulis}
\mu_X\big(\{x\in X: e^{t\zz}(x) \in S_{\Delta,z(t)} \text{ for arbitrarily large $t > 0$}\}\big) = \begin{cases}
0 \text{ if } \sum_{t = 1}^\infty\Phi_\Delta\circ z(t) < \infty\\
1 \text{ if } \sum_{t = 1}^\infty \Phi_\Delta\circ z(t) = \infty
\end{cases}.
\end{equation}
\end{theorem}
\begin{proof}[Proof of Theorem \ref{theoremkleinbockmargulis} using Theorem \ref{theoremkleinbockmargulisdiscrete}]
Let $z_t^{(1)} = z(t)$, and let $z_t^{(2)} = z(t) - C$ for some $C > 0$. To complete the proof it suffices to demonstrate the following:

\begin{itemize}
\item[(i)] $\sum_{t = 1}^\infty \Phi_\Delta(z_t^{(i)}) < \infty$ if and only if $\sum_{t = 1}^\infty \Phi_\Delta\circ z(t) < \infty$,
\item[(ii)] $e^{t\zz}(x)\in S_{\Delta,z_t^{(1)}}$ for infinitely many $t\in\N$ implies $e^{t\zz}(x)\in S_{\Delta,z(t)}$ for arbitrarily large $t > 0$, and
\item[(iii)] If $C$ is large enough, $e^{t\zz}(x)\in S_{\Delta,z(t)}$ for arbitrarily large $t > 0$ implies $e^{t\zz}(x)\in S_{\Delta,z_t^{(2)}}$ for infinitely many $t\in\N$.
\end{itemize}
Indeed, (i) follows from the fact that $\Phi_\Delta$ is regular (since $\Delta$ is assumed distance-like), and (ii) is obvious, so we turn to (iii). Suppose that $e^{t\zz}(x)\in S_{\Delta,z(t)}$ for some $t$, and let $t' = \lfloor t\rfloor$. Then
\[
\dist_X\big(e^{t'\zz}(x),e^{t\zz}(x)\big) \leq C_1
\]
for some constant $C_1 > 0$; since $\Delta$ is uniformly continuous, there exists $C = C_2 > 0$ independent of $t$ so that $\big|\Delta\big(e^{t'\zz}(x)\big) - \Delta\big(e^{t\zz}(x)\big)\big| \leq C_2$. On the other hand, since $z$ is nondecreasing, $z_{t'}^{(2)} \leq z(t) - C$; it follows that $e^{t'\zz}(x)\in S_{\Delta,z_{t'}^{(2)}}$.
\end{proof}

Let $\orth(\nform)_0$ denote the identity component of $\orth(\nform)$. We claim that Theorem \ref{theoremkhinchinquadratic2} follows from applying Theorem \ref{theoremkleinbockmargulis} with
\begin{equation}
\label{reductions}
\begin{alignedat}{2}
G &= \orth(\nform)_0, &
\Gamma &= \orth(\nform;\Lambda_*)\cap \orth(\nform)_0,\\
X &= G/\Gamma \equiv \Omega_{\nform,\Lambda_*},\hspace{.2 in} &
\Delta &= -\log\mindist, \ \text{where $\delta$ is as in \eqref{mindist},} \\
\zz &= \frac{\del}{\del t}g_t \Big|_{t = 0} = \left[\begin{array}{lll}
-1 &&\\
& \0_{d - 1} &\\
&& 1
\end{array}\right], \text{ and} \hspace{-1 in} &\\
z(t) &= -\log r_\psi(t).
\end{alignedat}
\end{equation}
Obviously, the verification of this claim consists of two parts: showing that the hypotheses of Theorem \ref{theoremkleinbockmargulis} are satisfied, and showing that Theorem \ref{theoremkhinchinquadratic2} follows from the conclusion of Theorem \ref{theoremkleinbockmargulis}.

\begin{proof}[Verification of the hypotheses]
The verification of hypotheses is mostly a consequence of well-known facts; we leave the details to the reader, proving only the following statements:

\begin{itemize}\label{pagerefpiz}
\item[1.] \textsl{$(\pi_i)'(\zz)\neq\0\all\, i$.} To see this, first note that the group $G$ is isomorphic to $\orth(p,q)_0$, where $p = \Rrank$ and $q = d + 1 - \Rrank$. Now $\orth(p,q)_0$ is simple as long as $p + q\geq 3$ and $(p,q)\notin \{(4,0),(2,2),(0,4)\}$; if $(p,q)\in\{(4,0),(2,2),(0,4)\}$, then $\orth(p,q)_0$ is only semisimple. In our case, $1\leq p\leq q$ and $p + q = d + 1 \geq 3$, so $G$ is simple unless $p = q = 2$. If $G$ is simple, there is nothing to prove, so assume that $p = q = 2$. Then by Proposition \ref{propositionrenormalization}, $G\equiv \orth(2,2)_0$ is conjugate in $\SL_4(\R)$ to $\orth(\form_0)_0$, where
\[
\form_0(\xx) = x_0 x_3 - x_1 x_2
\]
is the exceptional quadratic form; moreover, it is readily seen that $\orth(\form_0)_0 = \SL_2(\R)\times\SL_2(\R)$, where $G\times H$ denotes the set of all matrices of the form $g\otimes h$, where $g\in G$ and $h\in H$. (Cf. the ``product structure'' of $M_{\form_0}$ described in Section \ref{sectionspecialtype}). Write $G = \matrix\big(\SL_2(\R)\times\SL_2(\R)\big)$ for some matrix $\matrix\in\SL_4(\R)$. Then the factors of $G$ are given by the formulas $G_1 = \matrix\big(\SL_2(\R)\times I\big)$, $G_2 = \matrix\big(I\times \SL_2(\R)\big)$.\Footnote{If $\Gamma$ is irreducible, then there will actually be only one factor, namely $G$, and so as before there is nothing to prove. (In fact, this happens if and only if $\Qrank = 1$.)} The tangent spaces are given by the formulas $\mfg_1 = \matrix\big(\sl_2(\R)\times I\big)$, $\mfg_2 = \matrix\big(I\times \sl_2(\R)\big)$. Now any element of either of these tangent spaces has eigenvalues $\lambda,\lambda,-\lambda,-\lambda$ for some $\lambda\in\R$. On the other hand, the eigenvalues of $\zz$ are $1,0,0,-1$. Thus $\zz\notin \mfg_1,\mfg_2$. It follows that $(\pi_i)'(\zz)\neq\0\all\, i$.
\item[2.] \label{pagerefDelta} \textsl{$\Delta$ is uniformly continuous.} To see this, fix $g\in G$ and $\Lambda\in X$; then for all $\rr\in\Lambda$, we have $\|g\rr\| \leq \|g\|\cdot\|\rr\|$, where $\|g\|$ is the operator norm of $g$. Taking the minimum over $\rr\in\Lambda\smallsetminus\{\0\}$ gives
\[
\mindist(g\Lambda) \leq \|g\| \mindist(\Lambda),
\]
or equivalently $\Delta(\Lambda) \leq \Delta(g\Lambda) + \log\|g\|$. A symmetric argument gives $\Delta(g\Lambda) \leq \Delta(\Lambda) + \log\|g^{-1}\|$. Since $\log\|g\|,\log\|g^{-1}\| \leq \dist_G(\id,g)$ for all $g$, it follows that \begin{equation}\label{lip}{\Delta\text{ is $1$-Lipschitz.}}\end{equation}

\item[3.] \textsl{$\Phi_\Delta$ is regular.} This will be a consequence of the following asymptotic formula for $\Phi_\Delta(z)$, whose proof will occupy Section \ref{sectionreductiontheory}, and which we will make further use of below:
\end{itemize}

\begin{repproposition}{propositionkDL}
For $z$ large enough,
\[
\Phi_\Delta(z) \asymp_\times \begin{cases}
e^{-(d - 1)z} & \nform\not\sim\form_0 \\
e^{-2z}z & \nform\sim\form_0
\end{cases}.
\]
\end{repproposition}
\noindent This completes the verification of the hypotheses of Theorem \ref{theoremkhinchinquadratic2}.
\end{proof}

\begin{proof}[Completion of the proof]
First, we rewrite \eqref{kleinbockmargulis} using \eqref{reductions}:
\begin{align*}
&
\begin{cases}
0 \text{ if } \sum_{t = 1}^\infty\Phi_\Delta\big(-\log r_\psi(t)\big) < \infty\\
1 \text{ if } \sum_{t = 1}^\infty \Phi_\Delta\big(-\log r_\psi(t)\big) = \infty
\end{cases}\\
&= \mu_{\nform,\Lambda_*}\big(\{\Lambda\in \Omega_{\nform,\Lambda_*}: g_t\Lambda \in S_{-\log\mindist,-\log r_\psi(t)} \text{ for arbitrarily large $t > 0$}\}\big)\\
&= \mu_{\nform,\Lambda_*}\big(\{\Lambda\in \Omega_{\nform,\Lambda_*}: \mindist(g_t\Lambda) \leq r_\psi(t) \text{ for arbitrarily large $t > 0$}\}\big)\\
&= \mu_{\nform,\Lambda_*}\big(\A_\nform(\psi)\big).
\end{align*}
So to complete the proof, it suffices to show that the series
\begin{equation}
\label{rpsilog}
\sum_{t = 1}^\infty\Phi_\Delta\big(-\log r_\psi(t)\big)
\end{equation}
is asymptotic to \eqref{loglog}. First of all, by Proposition \ref{propositionkDL}, we have
\[
\eqref{rpsilog} \asymp_\times \begin{cases}
\sum_{t = 1}^\infty r_\psi(t)^{d - 1} & \nform\not\sim\form_0 \\
\sum_{t = 1}^\infty r_\psi(t)^2 \big(-\log r_\psi(t)\big) & \nform\sim\form_0
\end{cases}.
\]
Let
\begin{equation}
\label{indicator}
n = \begin{cases}
0 & \nform\not\sim\form_0\\
1 & \nform\sim\form_0
\end{cases}.
\end{equation}
Then we can write both \eqref{loglog} and \eqref{rpsilog} in a uniform manner:
\begin{align*}
\eqref{loglog} &=_\pt \sum_{\numberq\in 2^\N} \numberq^{d-1} \log^n\log \numberq \,\psi^{d-1} (\numberq) \\
\eqref{rpsilog} &\asymp_\times \sum_{t = 1}^\infty r_\psi(t)^{d-1}  \big(-\log r_\psi(t)\big)^n.
\end{align*}
Since $\psi$ is regular, each of these series is asymptotic to its corresponding integral, that is,
\begin{align*}
\eqref{loglog} &\asymp_{\plus,\times} \int_0^\infty (2^x)^{d-1}  \log^n\log(2^x)\,\psi^{d-1} (2^x)\;\dee x\\
\eqref{rpsilog} &\asymp_{\plus,\times} \int_0^\infty r_\psi(t)^{d-1}  \big(-\log r_\psi(t)\big)^n\;\dee t.
\end{align*}
Let $\Psi(\numberq) = \numberq^{d-1} \log^n\log \numberq$. In the following integrals, we omit the finite limit of integration since it is irrelevant for determining whether or not the integral converges. The reader should think of the finite limit of integration as being some arbitarily large number.
\begin{align*}
\eqref{loglog}
&\asymp_{\plus,\times} \int^\infty (2^x)^{d-1}  \log^n\log(2^x)\,\psi^{d-1} (2^x)\;\dee x\\
&\asymp_{\times\phantom{,\plus}} \int^\infty \numberq^{d-1}  \log^n\log \numberq \,\psi^{d-1} (\numberq) \;\frac{\dee \numberq}{\numberq}\\
&=_{\phantom{\times,\plus}} \int^\infty R \psi^{d-1} (\numberq)\;\frac{\dee R}{\numberq\Psi'(\numberq)} & \text{(letting $\numberq = \Psi^{-1}(R)$)}\\
&\asymp_{\times\phantom{,\plus}} \int^\infty \psi^{d-1} \big(\Psi^{-1}(R)\big)\;\dee R. \since{$\Psi(\numberq)\asymp_\times \numberq\Psi'(\numberq)$}
\end{align*}
We shall now resort to the following lemma:
\begin{lemma}
\label{lemmacongruent}
Let $f:[c,\infty)\to(0,\infty)$ be a strictly decreasing continuous function. Then
\[
\int_c^\infty f(x)\;\dee x + cf(c) = \int_0^{f(c)} f^{-1}(x)\;\dee x.
\]
\end{lemma}
\begin{subproof}
The regions whose areas are represented by these integrals are congruent to each other via the map $(x,y)\mapsto (y,x)$.
\end{subproof}
Applying this lemma with $f = \psi^{d-1} \circ\Psi^{-1}$, we continue our calculation:
\begin{align*}
\eqref{loglog}
&\asymp_{\plus,\times} \int_0 \Psi\left(\psi^{-1}\big(U^{\frac1{d-1} }\big)\right)\;\dee U \by{Lemma \ref{lemmacongruent}}\\
&\asymp_{\times\phantom{,\plus}} \int^\infty \Psi\big(\psi^{-1}(e^{-t})\big)e^{-{(d-1)} t}\;\dee t &\text{(letting $U = e^{-{(d-1)} t}$)}\\
&=_{\phantom{\plus,\times}} \int^\infty r_\psi(t)^{d-1} \log^n\log(\psi^{-1}\big(e^{-t})\big)\;\dee t.
\end{align*}
Comparing with \eqref{rpsilog}, we see that we have proven Theorem \ref{theoremkhinchinquadratic2} in the case $n = 0$, and also for all functions $\psi$ satisfying
\begin{equation}
\label{ETSloglog}
\log\log\psi^{-1}(e^{-t}) \asymp_{\plus,\times} -\log r_\psi(t).
\end{equation}
\begin{remark}
For the remainder of the proof, we could require $n = 1$ and thus $d -1 =2$ to simplify notation somewhat. However, we prefer to keep the original notation.
\end{remark}

For each $c > 0$, let $\psi_{1,c}$ be defined by the equation
\[
r_{\psi_{1,c}}(t) = \frac{1}{t^c},
\]
i.e.
\[
\psi_{1,c}^{-1}(x) = \frac{1}{x(-\log x)^c}\cdot
\]
Then
\begin{align*}
-\log r_{\psi_{1,c}}(t) &= c\log t  \asymp_\times \log t\\
\log\log\psi_{1,c}^{-1}(e^{-t}) &= \log(t + c\log t ) \asymp_\plus \log t .
\end{align*}
This yields the following:
\begin{claim}
\label{claimkhinchinquadratic}
Fix $c_1 > c_2 > 0$. Then Theorem \ref{theoremkhinchinquadratic2} holds for any function $\psi_{1,c_1} \leq \psi \leq \psi_{1,c_2}$.
\end{claim}
\begin{subproof}
We have $\psi_{1,c_1}^{-1} \leq \psi^{-1} \leq \psi_{1,c_2}^{-1}$ and $r_{\psi_{1,c_1}} \leq r_\psi \leq r_{\psi_{1,c_2}}$, and thus
\[
\log\log\psi^{-1}(e^{-t}) \asymp_\plus \log t \asymp_\times -\log r_\psi(t),
\]
i.e.\ \eqref{ETSloglog} holds.
\end{subproof}

\begin{remark}
This completes the proof of Theorem \ref{theoremkhinchinquadratic2} for the case of most ``reasonable'' functions $\psi$, for example if $\psi$ can be written in terms of the elementary operations together with exponents and logs. Such a $\psi$ is always comparable to every function $\psi_{1,c}$ \cite[Chapter III]{Hardy}. On the other hand, if $c_1 > \frac1{d-1} > c_2 > 0$, then \eqref{loglog} converges with $\psi = \psi_{1,c_1}$ but diverges with $\psi = \psi_{1,c_2}$. If $\psi \lesssim_\times \psi_{1,c_1}$, then $\A_\nform(\psi) \subset \A_\nform(C\psi_{1,c_1})$ for some $C > 0$, implying that $\mu_{\nform,\Lambda_*}\big(\A_\nform(\psi)\big) = 0$. Similarly, if $\psi \gtrsim_\times \psi_{1,c_2}$ then $\mu_{\nform,\Lambda_*}\big(\A_\nform(\psi)\big) = 1$. Finally, if $\psi_{1,c_1} \lesssim_\times \psi \lesssim_\times \psi_{1,c_2}$, then Claim \ref{claimkhinchinquadratic} gives the desired result.
\end{remark}
We now proceed to prove the general case of Theorem \ref{theoremkhinchinquadratic2}, using Claim \ref{claimkhinchinquadratic}. Fix $c_1 > \frac1{d-1}  > c_2 > c_3 > 0$.

\begin{claim}
We can without loss of generality assume $\psi \geq \psi_{1,c_1}$.
\end{claim}
\begin{subproof}
Suppose that the theorem is true for all $\psi\geq\psi_{1,c_1}$, and let $\psi$ be arbitrary. Let $\psi' = \max(\psi,\psi_{1,c_1})$. Note that \eqref{loglog} converges for $\psi = \psi'$ if and only if it converges for $\psi = \psi$. Applying the known case of the theorem, we have
\[
\mu_{\nform,\Lambda_*}\big(\A_\nform(\psi')\big) = \begin{cases}
0 & \text{\eqref{loglog} converges}\\
1 & \text{\eqref{loglog} diverges}
\end{cases}.
\]
On the other hand, we have
\[
\A_\nform(\psi') = \A_\nform(\psi) \cup \A_\nform(\psi_{1,c_1}).
\]
Since the latter set has measure zero, the measures of $\A_\nform(\psi')$ and $\A_\nform(\psi)$ are equal.
\end{subproof}

So from now on, we assume $\psi \geq \psi_{1,c_1}$. If $\psi \leq \psi_{1,c_3}$, then this completes the proof (of Theorem \ref{theoremkhinchinquadratic2}). So we will assume that $\psi(q) > \psi_{1,c_3}(q)$ for arbitrarily large $q$.

\begin{claim}
Fix $\numberq_2$ for which $\psi(\numberq_2) > \psi_{1,c_3}(\numberq_2)$, and let $\numberq_1 < \numberq_2$ be the largest value for which $\psi(\numberq_1) \leq \psi_{1,c_2}(\numberq_1)$. Then
\begin{equation}
\label{Q1Q2}
\int_{\numberq_1}^{\numberq_2} \numberq^{d-1} \log^n\log \numberq\,\psi_{1,c_2}^{d-1}(\numberq)\;\frac{\dee \numberq}{\numberq} \gtrsim_\times \log^{(c_2/c_3)(1 - (d-1) c_2)} \numberq_2  - C\log^{1 - (d-1) c_2} \numberq_2 
\end{equation}
for some constant $C > 0$.
\end{claim}
\begin{proof}
Since $q\mapsto q\psi(q)$ is assumed to be nondecreasing, we have
\[
\numberq_1\psi_{1,c_2}(\numberq_1) \geq \numberq_1\psi(\numberq_1) \geq \numberq_2\psi(\numberq_2) > \numberq_2\psi_{1,c_3}(\numberq_2).
\]
On the other hand,
\[
\psi_c(q) \asymp_\times \frac{1}{q\log^c q},
\]
so
\begin{equation*}
\label{c2c3}
\log^{c_2} \numberq_1 \lesssim_\times \log^{c_3} \numberq_2.
\end{equation*}
Now
\begin{align*}
\int_{\numberq_1}^{\numberq_2} \numberq^{d-1} \log^n\log \numberq\,\psi_{1,c_2}^{d-1}(\numberq)\;\frac{\dee \numberq}{\numberq}
&\asymp_\times \int_{\numberq_1}^{\numberq_2} \frac{\log^n\log \numberq}{\log^{(d-1) c_2} \numberq }\;\frac{\dee \numberq}{\numberq}\\
&= \int_{\log \numberq_1}^{\log \numberq_2} \frac{\log^n t}{t^{(d-1) c_2}}\;\dee t\\
&\geq \int_{\log\numberq_1 }^{\log \numberq_2} t^{-(d-1)c_2} \;\dee t\\
&\asymp_\times \log^{1 - (d-1) c_2}\numberq_2  - \log^{1 - (d-1) c_2} \numberq_1\\
&\gtrsim_\times \log^{(c_2/c_3)(1 - (d-1) c_2)} \numberq_1 - C\log^{1 - (d-1) c_2} \numberq_1.
\end{align*}
\QEDmod\end{proof}
Since the right hand side of \eqref{Q1Q2} tends to infinity as $\numberq_2\to\infty$, the existence of infinitely large values of $\numberq_2$ for which the hypotheses of the claim are satisfied implies that
\[
\int^\infty \numberq^{d-1} \log^n\log \numberq\, \min\big(\psi(\numberq),\psi_{1,c_2}(\numberq)\big)^{d-1}\;\frac{\dee \numberq}{\numberq} = \infty,
\]
i.e.\ \eqref{loglog} diverges for $\psi = \min(\psi,\psi_{1,c_2})$. Thus by Claim \ref{claimkhinchinquadratic}, we have
\[
\mu_{\nform,\Lambda_*}\left(\A_\nform\big(\min(\psi,\psi_{1,c_2})\big)\right) = 1.
\]
But since $\A_\nform(\psi)\supset \A_\nform\big(\min(\psi,\psi_{1,c_2})\big)$, this completes the proof of Theorem \ref{theoremkhinchinquadratic2}.
\end{proof}


\section{Estimating the measure $\mu_{\nform,\Lambda_*}$}
\label{sectionreductiontheory}

In this section we estimate $\int\varphi\;\dee\mu_{\nform,\Lambda_*}$ for any function $\varphi:\Omega_{\nform,\Lambda_*}\to\Rplus$. Our main tools will be the generalized Iwasawa decomposition (Theorem \ref{theoremgeneralizediwasawa}) and the reduction theory of algebraic groups (Theorem \ref{theoremreduction2}). We first prove a theorem for general algebraic groups, and then specialize to the case $G = \orth(\nform)_0$.

We will need the following notation: if $X$ is a metric space with distance  $\dist_X$, $\varphi$ is a nonnegative continuous function on $X$ and $C>0$, we define 
\begin{equation}
\label{phiCdef}
\varphi^{(C)}(x) := \max_{\dist_X(x' , x)\leq C}\varphi(x'), \;\; \varphi_{(C)}(x) := \min_{\dist_X(x' , x)\leq C}\varphi(x').
\end{equation}

Let $G$ be a semisimple algebraic group. Let $P\subgp G$ be a parabolic subgroup, and let $P = M A N$ be a Langlands decomposition of $P$. Let $\mfg$, $\mfp$, $\mfm$, $\mfa$, and $\mfn$ denote the corresponding Lie algebras. Let $K\subgp G$ be a maximal compact subgroup whose Lie algebra $\mfk$ is orthogonal to $\mfa$ with respect to the Killing form.


\begin{theorem}[Generalized Iwasawa decomposition, {\cite[Proposition 8.44]{Knapp}}]
\label{theoremgeneralizediwasawa}
Let $\rho_P$ be the modular function of $P$. Then given any Haar measures $\mu_K$, $\mu_M$, $\mu_A$, $\mu_N$ on $K$, $M$, $A$, $N$ respectively, the measure $\mu_G$ given by
\[
\int_G \Phi\,\dee\mu_G := \int_{K\times M\times {A}\times N} \rho_P(a) \Phi({kman}) \;\dee (\mu_K\times\mu_M\times\mu_{A}\times\mu_N)(k,m,a,n),
\]
where $\Phi$ is a measurable function on $G$, is a Haar measure on $G$.
\end{theorem}

Now suppose that $G$ is $\Q$-algebraic and that $P\subgp G$ is a minimal parabolic $\Q$-subgroup. Let $\Gamma\subgp G$ be a lattice commensurable to $G_\Z$.

\begin{definition}
\label{definitioncoarsedomain}
A set $\FF\subset G$ is a \emph{coarse fundamental domain} for $\Gamma$ if
\begin{itemize}
\item[(I)] $\FF \, \Gamma = G$, and
\item[(II)] $\#\{\gamma\in\Gamma: \FF\gamma \cap\FF\neq\emptyset\} < \infty$.
\end{itemize}
\end{definition}
\noindent Consider the set
\begin{equation}
\label{contractingweyl}
A^+ := \{a\in A : 
\Adj_a|_{\mfn}
 \text{ is contracting}\}.
\end{equation}
Here $\Adj_a$ denotes the adjoint action of $a$.

\begin{theorem}[Reduction theory for arithmetic groups, {\cite[Proposition 2.2]{Leuzinger} or \cite[Theorem 16.9]{Morris}}]
\label{theoremreduction1}
There exist precompact open sets $M_0\subset M$ and $N_0\subset N$ and a finite set $F\subset G_\Q$ such that
\begin{equation}
\label{Fdomain}
\FF := K M_0 A^+ N_0 F
\end{equation}
is a coarse fundamental domain for $\Gamma$.
\end{theorem}

Let $\dist_G$ denote a right-invariant Riemannian metric on $G$. Let $X = G/\Gamma$, and consider the metric $\dist_X(x,x') = \min_{g\Gamma = x, g'\Gamma = x'}\dist_G(g,g')$. We note that $\dist_X$ is a Riemannian metric on $X$. Let $\mu_X$ denote the normalized Haar measure on $X$. 

\begin{theorem}
\label{theoremreduction2}
There exist $C > 0$ and a finite set $F\subset G_\Q$ such that for any function $\varphi: X\to\Rplus$, we have
\[
\int_{A^+} \rho_P(a) \sum_{f\in F} \varphi_{(C)}(af\Gamma) \;\dee\mu_{A}(a)
\lesssim_\times \int \varphi \;\dee\mu_X
\lesssim_\times \int_{A^+} \rho_P(a)\sum_{f\in F}\varphi^{(C)}(af\Gamma) \;\dee\mu_{A}(a).
\]
\end{theorem}
\begin{proof}
Let $M_0\subset M$, $N_0\subset N$ and $F\subset G_\Q$ be as in Theorem \ref{theoremreduction1}, and let $\FF$ be given by \eqref{Fdomain}. Let $\FF_0 = K M_0 A^+ N_0$, so that $\FF = \FF_0 F$. Then
\begin{align*}
\int_{\FF_0} \sum_{f\in F}\varphi(gf\Gamma)\;\dee\mu_G(g)
&\lesssim_\times \int_\FF \varphi(g\Gamma)\;\dee\mu_G(g) \since{$\#(F) < \infty$}\\
&\lesssim_\times \int \varphi \;\dee\mu_X \by{(II) of Definition \ref{definitioncoarsedomain}}\\
&\leq_\pt \int_\FF \varphi(g\Gamma)\;\dee\mu_G(g) \by{(I) of Definition \ref{definitioncoarsedomain}}\\
&\leq_\pt \int_{\FF_0} \sum_{f\in F} \varphi(gf\Gamma)\;\dee\mu_G(g).
\end{align*}
Let $\Phi(g) = \sum_{f\in F}\varphi(gf\Gamma)$, so that
\begin{equation}
\label{reduction1}
\int \varphi \;\dee\mu_X \asymp_\times \int_{\FF_0} \Phi \;\dee\mu_G.
\end{equation}
Now by Theorem \ref{theoremgeneralizediwasawa},
\begin{equation}
\label{iwasawa}
\int_{\FF_0} \Phi \;\dee\mu_G
= \int_{K\times M_0\times A^+\times N_0} \rho_P(a) \Phi(kman) \;\dee(\mu_K\times\mu_M\times\mu_{A}\times\mu_N)(k,m,a,n).
\end{equation}
Now let
\begin{equation}
\label{precompact}
C = \max\big\{\dist_G\big(\id,km(ana^{-1})\big) : k\in K, m\in M_0, a\in A^+, n\in N_0\big\}.
\end{equation}
Since $N$ is contracted by the adjoint action of $A^+$, the set $\{ana^{-1}: a\in A^+, n\in N_0\}$ is precompact and thus $C < \infty$. For $k\in K$, $m\in M_0$, $a\in A^+$, and $n\in N_0$ fixed, we have
\[
\dist_G(a,kman) = \dist_G(a,km(ana^{-1})a) \leq C
\]
and thus
\[
\Phi(kman) = \Phi(km(ana^{-1})a) \in [\Phi_{(C)}(a),\Phi^{(C)}(a)].
\]
Thus by \eqref{iwasawa},
\begin{equation}
\label{reduction2}
\begin{split}
&\int_{K\times M_0\times A^+\times N_0} \rho_P(a) \Phi_{(C)}(a) \;\dee(\mu_K\times\mu_M\times\mu_{A}\times\mu_N)(k,m,a,n)\\
&\leq \int_{\FF_0} \Phi \;\dee\mu_G\\
&\leq \int_{K\times M_0\times A^+\times N_0} \rho_P(a) \Phi^{(C)}(a) \;\dee(\mu_K\times\mu_M\times\mu_{A}\times\mu_N)(k,m,a,n).
\end{split}
\end{equation}
Now since $K$, $M_0$, and $N_0$ are open and precompact we have
\begin{equation}
\label{reduction3}
\begin{split}
&\int_{K\times M_0\times A^+\times N_0} \rho_P(a) \Phi^{(C)}(a) \;\dee(\mu_K\times\mu_M\times\mu_{A}\times\mu_N)(k,m,a,n)\\
&\asymp_\times \int_{A^+} \rho_P(a) \Phi^{(C)}(a) \;\dee\mu_{A}(a),
\end{split}
\end{equation}
and similarly for $\Phi_{(C)}$. Combining \eqref{reduction1}, \eqref{reduction2}, and \eqref{reduction3} completes the proof.
\end{proof}

Next we apply Theorem \ref{theoremreduction2} to the case where $G = \orth(\nform)_0$ for some quadratic form $\nform:\R^{d + 1}\to\R$. 
Suppose that $\Lambda_*$ is an $R$-arithmetic lattice  commensurable with $\Z^{d + 1}$. Then  $X:=\Omega_{\nform,{\Lambda_*}}\cong G/\Gamma $,  where\linebreak $\Gamma = \orth(\nform;\Lambda_*)$, see \eqref{reductions}. 
In view of Proposition \ref{propositioninclusionproper}, it is properly embedded into the space $\Omega_d$ of all lattices in $\R^{d+1}$. We are going to consider functions $\varphi:\Omega_{\nform,\Lambda_*}\to\Rplus$ which are restrictions of functions on $\Omega_d$ satisfying an additional property defined below.
\begin{definition}
\label{definitionmonotonic}
A function $\varphi:\Omega_
d\to\Rplus$ is \emph{monotonic} if $\Lambda_1\subgp\Lambda_2$ implies $\varphi(\Lambda_1)\leq\varphi(\Lambda_2)$.
\end{definition}

\begin{theorem}
\label{theoremreduction3}
Let $\nform:\R^{d + 1}\to\R$ be a $\Qrank$-normalized quadratic form, and suppose that 
$\Lambda_*\in\Omega_d$ is commensurable with $\Z^{d + 1}$. Let
\begin{align*}
\ss &=\left[\begin{array}{c}
d - 1 \\
d - 3 \\
\vdots \\
d + 1 - 2\Qrank
\end{array}\right] \in \R^\Qrank.
\end{align*}
There exists $C > 0$ such that for any monotonic function $\varphi:\Omega_d
\to\Rplus$ we have
\[
\int_{\tt\in\mfa^+} e^{-\ss\cdot\tt} \varphi_{(C)}(g_\tt\Lambda_*) \;\dee\tt
\lesssim_\times \int_{X
} \varphi \;\dee\mu_{X
}
\lesssim_\times \int_{\tt\in\mfa^+} e^{-\ss\cdot\tt} \varphi^{(C)}(g_\tt\Lambda_*)\;\dee\tt.
\]
\end{theorem}

We remark that even though we integrate $\varphi$ over $X = \Omega_{\nform,{\Lambda_*}}$, it is assumed to be a function on $\Omega_d$; 
 in particular, the functions $ \varphi_{(C)}$, $ \varphi^{(C)}$ are defined with respect to the Riemannian distance on $\Omega_d\cong\GL_{d + 1}(\R)/\GL_{d + 1}(\Z)$. 
\begin{proof}
Let $G = \orth(\nform)_0$ and let $\Gamma = \orth(\nform;\Lambda_*)\cap\orth(\nform)_0$. Then $G$ is a semisimple $\Q$-algebraic group, and $\Gamma$ is commensurable with $G_\Z$. For $\tt\in \R^\Qrank$, let $\Phi(\tt) = g_\tt$ be as in \eqref{gtdef1}, so that $\Phi:\R^\Qrank\to G$ is a homomorphism. Let $A = \Phi(\R^\Qrank)$. Then the Lie algebra $\mfa$ of $A$ is isomorphic to $\R^\Qrank$ via the map $\Phi'(\0)$. In our notation, we will not distinguish between $\mfa$ and $\R^\Qrank$.

Let
\begin{equation}
\label{qtorusplusdef}
\mfa^+ = \{\tt\in\R^\Qrank:t_0 > t_1 > \ldots > t_{\Qrank - 1} > 0 \} \subset \mfa,
\end{equation}
and let $A^+ = \exp(\mfa^+)$. Then $A$ is a maximal $\Q$-split torus, and $A^+$ is as in \eqref{contractingweyl}. Fix $a\in A^+$, and let $N\subgp G$ and $P\subgp G$ be the groups
\begin{align*}
N &:= \left\{g\in G: a^n g a^{-n} \tendsto n 0\right\}\\
P &:= \left\{g\in G: (a^n g a^{-n})_1^\infty \text{ is bounded} \right\},
\end{align*}
i.e.\ $N$ is the group of elements contracted by $A^+$, and $P$ is the group of elements stabilized by $A^+$. Then $P$ is a minimal parabolic $\Q$-subgroup of $G$ whose Langlands decomposition is $P = M A N$ for some reductive group $M\subgp P$. Moreover, $A^+$ is given by the formula \eqref{contractingweyl}. So by Theorem \ref{theoremreduction2}, there exist $C > 0$ and a finite set $F\subset G_\Q$ such that for any $\varphi:\Omega_{\nform,\Lambda_*}\to\Rplus$, we have
\begin{equation}
\label{reduction4}
\begin{split}
\int_{\tt\in \mfa^+} \rho_P(g_\tt) \sum_{f\in F} \varphi_{(C)}(g_\tt f\Lambda_*) \;\dee\tt
&\lesssim_\times 
\int_{X
} \varphi \;\dee\mu_{X
}\\
&\lesssim_\times \int_{\tt\in \mfa^+} \rho_P(g_\tt) \sum_{f\in F} \varphi^{(C)}(g_\tt f\Lambda_*) \;\dee\tt.
\end{split}
\end{equation}
Here we remark that formally Theorem \ref{theoremreduction2} produces \eqref{reduction4} with $ \varphi_{(C)}$, $ \varphi^{(C)}$  replaced by $ (\varphi_X)_{(C)}$, $ (\varphi_X)^{(C)}$ respectively, where the latter are defined  with respect to the Riemannian distance on $X$. But since we clearly have $ \varphi_{(C)}\le (\varphi_X)_{(C)}$ and  $ \varphi^{(C)}\ge  (\varphi_X)^{(C)}$,  \eqref{reduction4} follows.

\begin{claim}
For some $C' > 0$,
\begin{equation}
\label{ETSreduction}
\sum_{f\in F}\varphi^{(C)}(g_\tt f\Lambda_*) \lesssim_\times \varphi^{(C')}(g_\tt\Lambda_*).
\end{equation}
\end{claim}
\begin{subproof}
For $f\in F\subset G_\Q$ fixed, $f\Lambda_*$ is commensurable with $\Lambda_*$, and thus $\frac{1}{N_f}\Lambda_* \subgp f\Lambda_* \subgp N_f\Lambda_*$ for some $N_f\in\N$. In particular, since $\varphi$ is monotonic
\[
\varphi^{(C)}(g_\tt f\Lambda_*) \leq \varphi^{(C)}(g_\tt N_f\Lambda_*)  = \varphi^{(C)}( N_f g_\tt\Lambda_*) \leq \varphi^{(C + \log N_f)}(g_\tt\Lambda_*),
\]
where the last inequality follows since the distance on  $\Omega_d$ is defined via a Riemannian metric on $GL_{d+1}(\R)$. 
Thus \eqref{ETSreduction} holds with $C' = C + \log\max_{f\in F} N_f $.
\end{subproof}
A similar argument shows that
\[
\sum_{f\in F}\varphi_{(C)}(g_\tt f\Lambda_*) \gtrsim_\times \varphi_{(C')}(g_\tt\Lambda_*).
\]
Thus \eqref{reduction4} becomes
\begin{equation}
\label{reduction5}
\begin{split}
\int_{\tt\in \mfa^+} \rho_P(g_\tt) \varphi_{(C')}(g_\tt\Lambda_*)\;\dee\tt
&\lesssim_\times \int \varphi\;\dee\mu_{\nform,\Lambda_*}\\
&\lesssim_\times \int_{\tt\in \mfa^+} \rho_P(g_\tt) \varphi^{(C')}(g_\tt\Lambda_*)\;\dee\tt.
\end{split}
\end{equation}
\begin{claim}
$\rho_P(g_\tt) = e^{-\ss\cdot\tt}$. (Here and hereafter $\ss\cdot\tt$ denotes $\sum_{i = 1}^{\Qrank - 1} s_i t_i$.)

\end{claim}
\begin{subproof}
It is well-known (see e.g.\ \cite[(8.38)]{Knapp}\Footnote{The sign difference between \cite[(8.38)]{Knapp} and the present formula is due to Knapp's convention of assuming that $\mfn$ is the union of the positive root spaces, while we assume that $\mfn$ is the union of the negative root spaces (cf.\ \eqref{contractingweyl}).})
that $\rho_P(g_\tt) = e^{-\rho(\tt)}$, where $\rho$ is the sum of the positive roots of $A$, counting multiplicity.

So to demonstrate the claim, we must show that $\rho(\tt) = \ss\cdot\tt$. One verifies that the positive roots of $A$ are of the form
\begin{align*}
\lambda_{i,j,\pm} &:= \ee_i^* \pm \ee_j^* & i < j < \Qrank&\\
\lambda_i &:= \ee_i^* & i < \Qrank&,
\end{align*}
with corresponding root spaces
\begin{align*}
\mfg_{\lambda_{i,j,-}} &= \R\big( \ee_j \cdot\ee_i^* - \ee_{d - i} \cdot \ee_{d - j}^*\big)\\
\mfg_{\lambda_{i,j,+}} &= \R\big( \ee_{d - j}\cdot\ee_i^* - \ee_{d - i}\cdot\ee_j^*\big) \\
\mfg_{\lambda_i} &= \{\xx\cdot\ee_i^* - \ee_{d - i}\cdot 2B_{\w\nform}(\xx,\cdot) : (x_\Qrank,\ldots,x_{d - \Qrank}) \in \R^{d + 1 - 2\Qrank}\}.
\end{align*}
In particular, the multiplicity of the root $\lambda_{i,j,\pm}$ is $1$, and the multiplicity of the root $\lambda_i$ is $(d + 1 - 2\Qrank)$. Thus
\begin{align*}
\rho &= \sum_{j = 1}^{\Qrank - 1} \sum_{i = 0}^{j - 1} [(\ee_i^* + \ee_j^*) + (\ee_i^* - \ee_j^*)] + \sum_{i = 0}^{\Qrank - 1}(d + 1 - 2\Qrank)\ee_i^*\\
&= \sum_{i = 0}^{\Qrank - 1} \left[2(\Qrank - i - 1) +(d + 1 - 2\Qrank)\right]\ee_i^*
= \sum_{i = 0}^{\Qrank - 1} [d - 2i - 1]\ee_i^*.
\end{align*}

\end{subproof}
This completes the proof of Theorem \ref{theoremreduction3}.
\end{proof}

Finally, we use Theorem \ref{theoremreduction3} to complete the proof of Theorem \ref{theoremkhinchinquadratic}. Recall that $\Delta$ denotes the function $\Delta = -\log\mindist: \Omega_{\nform,\Lambda_*}\to\R$ (cf.\ \eqref{reductions}), where $\mindist$ is defined by \eqref{mindist},
and that for $z\in\R$,
\[
S_{\Delta,z} = \{\Lambda\in \Omega_{\nform,\Lambda_*} : \Delta(\Lambda)\geq z\}.
\]

\begin{proposition}
\label{propositionkDL}
For $z$ large enough,
\[
\Phi_\Delta(z) := \mu_{\nform,\Lambda_*}(S_{\Delta,z}) \asymp_\times \begin{cases}
e^{-(d - 1)z} & \nform\not\sim\form_0 \\
e^{-2z}z & \nform\sim\form_0
\end{cases}.
\]
\end{proposition}
\begin{proof}
Clearly $\mindist(\Lambda) = \min_{\pp\in\Lambda \smallsetminus\{\0\}}\|\pp\|$ and $\Delta = -\log\mindist$ can be extended to $\Omega_d$ using the same definition.
For each $z\in\R$, define 
$$\varphi_z := 1_{\{\Lambda\in \Omega_{d} : \Delta(\Lambda)\geq z\}}.$$ Then the restriction of $\varphi_z$ to $\Omega_{\nform,\Lambda_*}$ is the characteristic function of $S_{\Delta,z}$, so that $$\Phi_\Delta(z) = \int_{\Omega_{\nform,\Lambda_*}}\varphi_z\;\dee\mu_{\nform,\Lambda_*}.$$ Observe that $\varphi_z$ is monotonic in the sense of Definition \ref{definitionmonotonic}, with $X = \Omega_{d}$. Thus by Theorem \ref{theoremreduction3}, there exists $C > 0$ independent of $z$ such that
\[
\int_{\tt\in\mfa^+} e^{-\ss\cdot\tt} (\varphi_z)_{(C)}(g_\tt\Lambda_*)\;\dee\tt
\lesssim_\times \Phi_\Delta(z)
\lesssim_\times \int_{\tt\in\mfa^+} e^{-\ss\cdot\tt} (\varphi_z)^{(C)}(g_\tt\Lambda_*) \;\dee\tt.
\]
Since $\Delta$ is $1$-Lipschitz 
(see \eqref{lip}), we have
\[
(\varphi_z)_{(C)} \geq \varphi_{z + C} \text{ and } (\varphi_z)^{(C)} \leq \varphi_{z - C},
\]
and so
\[
f(z + C)
\lesssim_\times \Phi_\Delta(z)
\lesssim_\times f(z - C),
\]
where
\[
f(z) := \int_{\tt\in\mfa^+} e^{-\ss\cdot\tt} \varphi_z(g_\tt\Lambda_*)\;\dee\tt.
\]
Thus to complete the proof it suffices to show that
\begin{equation}
\label{ETSkDL}
f(z) \asymp_\times \begin{cases}
e^{-(d - 1)z} & \nform\not\sim\form_0 \\
e^{-2z}z & \nform\sim\form_0
\end{cases}.
\end{equation}
Indeed, observe that that for $\tt\in\mfa^+$, the smallest vector in $g_\tt(\Z^{d + 1})$ is $g_\tt(\ee_0) = e^{-t_0}\ee_0$. Thus $\Delta(g_\tt\Z^{d + 1}) = t_0$. On the other hand, since $\Lambda_*$ is commensurable with $\Z^{d + 1}$, we have $\frac1N \Z^{d + 1}\subgp\Lambda_* \subgp N\Z^{d + 1}$ for some $N\in\N$, which implies $|\Delta(g_\tt\Lambda_*) - \Delta(g_\tt\Z^{d + 1})| \leq \log N$ for all $\tt$. It follows that $\Delta(g_\tt\Lambda) \asymp_\plus t_0$, and so
\[
\varphi_z(g_\tt\Z^{d + 1}) \asymp_\times \begin{cases}
1 & t_0 \geq z\\
0 & \text{otherwise}
\end{cases}.
\]
Therefore
\[
f(z) \asymp_\times \int_{\substack{t_0 > t_1 > \cdots > t_{\Qrank - 1} > 0 \\ t_0 > z}}\, e^{-\ss\cdot\tt} \;\dee\tt.
\]
\begin{claim}
For $x \geq 1$,
\[
\int_{x > t_1 > \cdots > t_{\Qrank - 1} > 0} e^{-\ss\cdot\tt} \;\dee\tt
\asymp_\times \begin{cases}
1 & \nform\not\sim\form_0\\
x & \nform\sim\form_0
\end{cases}.
\]
\end{claim}
\begin{subproof}
If $\Qrank = 1$, then the domain of integration is zero-dimensional, making the statement trivial. Thus suppose $\Qrank\geq 2$. If $d = 3$, then Proposition \ref{propositionrenormalization} implies that $\nform\sim\form_0$. So if $\nform\not\sim\form_0$, then $d\geq 4$ and in particular $s_1 = d - 3 > 0$. Since $s_i \geq 0$ for all $i$, we have
\[
\int_{t_1 > \cdots > t_{\Qrank - 1} > 0} e^{-\ss\cdot\tt} \;\dee\tt \leq \int_{t_1 > \cdots > t_{\Qrank - 1} > 0} e^{-s_1 t_1} \;\dee\tt \leq \int_{t_1,\ldots,t_{\Qrank - 1} > 0} e^{-\frac{s_1}{\Qrank - 1} \sum_{i = 1}^{\Qrank - 1}t_i} \;\dee\tt < \infty,
\]
demonstrating the upper bound. The lower bound is trivial, so this completes the proof if $\nform\not\sim\form_0$.

Now suppose that $\nform\sim\form_0$. Then $s_1 = 0$, and
\[
\int_{x > t_1 > \cdots > t_{\Qrank - 1} > 0} e^{-\ss\cdot\tt}\;\dee\tt = \int_{x > t_1 > 0} 1 \;\dee t_1 = x.
\]
\end{subproof}
\noindent Let $n$ be given by \eqref{indicator}, so that
\[
\int_{x > t_1 > \cdots > t_{\Qrank - 1} > 0} e^{-\ss\cdot\tt} \;\dee\tt \asymp_\times x^n.
\]
Integrating over $t_0 > z$ gives
\[
f(z) \asymp_\times \int_{t_0 > z} e^{-s_0 t_0} t_0^n \;\dee t_0 \asymp_\times e^{-s_0 z} z^n = e^{-(d - 1) z} z^n,
\]
demonstrating \eqref{ETSkDL}.
%
\end{proof}

We end this section by proving a lemma which was needed in the proof of Theorem \ref{theoremdirichletquadratic}(ii,iii). Recall the definition of codiameter given in Definition \ref{definitioncodiam}:

\begin{lemma}
\label{lemmaCnew}
There exists $C_1 > 0$ such that for every $\Lambda\in \Omega_{\nform,\Lambda_*}$, there exists a totally isotropic $\Lambda$-rational subspace $V\subgp\R^{d + 1}$ of dimension $\Qrank$ satisfying $\Codiam(V\cap\Lambda)\leq C$. 
\end{lemma}
\begin{proof}
Let $G$, $\Gamma$, $A$, $A^+$, $N$, $P$, and $M$ be as in the proof of Theorem \ref{theoremreduction3}. Let $M_0\subset M$, $N_0\subset N$, and $F\subset G_\Q$ be as in Theorem \ref{theoremreduction1}, and let $\FF$ be given by \eqref{Fdomain}. Then for every $\Lambda\in \Omega_{\nform,\Lambda_*}$, we can write $\Lambda = g\Lambda_*$ for some $g\in\FF$. Write
\[
g = kmanf = km(ana^{-1}) af,
\]
where $k\in K$, $m\in M_0$, $a\in A^+$, $n\in N_0$, and $f\in F$. Write $h = km(ana^{-1})$, so that
\[
\Lambda = haf\Lambda_*.
\]
We recall (cf.\ \eqref{precompact}) that
\[
\dist_G(\id,h) \leq C
\]
for some $C > 0$ independent of $\Lambda$.

Let $V_0 = \LL_\Qrank$, and let $V = h(V_0)$. We observe that $V_0$ is a totally isotropic $af\Lambda_*$-rational subspace of $\R^{d + 1}$ of dimension $\Qrank$, and thus $V$ is a totally isotropic $\Lambda$-rational subspace of $\R^{d + 1}$ of dimension $\Qrank$.

Since $a$ is contracting on $V_0$, we have $\Codiam(V_0\cap af\Lambda_*) \leq \Codiam(V_0\cap f\Lambda_*)$. On the other hand, $\Codiam(V_0\cap f\Lambda_*) \asymp_\times 1$ since $f$ ranges over a finite set. Thus
\[
\Codiam(V\cap \Lambda) \leq e^{\dist_G(\id,h)}\Codiam(V_0\cap af\Lambda_*) \lesssim_\times e^C.
\]
This completes the proof.
\end{proof}

\section{The exceptional quadric hypersurface}
\label{sectionspecialtype}

Recall that the \emph{exceptional quadric hypersurface} is the hypersurface $M_{\form_0}$ defined by the exceptional quadratic form \eqref{form0}.
This hypersurface occupies an interesting place in the theory of intrinsic Diophantine approximation on quadric hypersurfaces developed in this paper. To begin with, it has ``more rational points than expected''. Specifically, according to Theorem \ref{theoremheathbrown}
\begin{equation}
\label{specialtypecount}
N_{M_{\form_0}}(\numberq) \asymp_\times \numberq^2 \log \numberq,
\end{equation}
rather than $N_{M_\form}(\numberq) \asymp_\times \numberq^2$, which holds when $\form$ is a quadratic form on $\R^4$ which is not equivalent to $\form_0$. Nevertheless, these ``extra points'' do not appear to affect either the Dirichlet- or Khintchine-type theorems of these manifolds in quite the way one would expect. With regards to the Dirichlet-type theorem, the extra points have no effect at all, and the optimal Dirichlet function for $M_\form$ is always $\psi_1$, independent of whether or not $\form\sim\form_0$. On the other hand, the extra points do affect the Khintchine-type theorem, but not as expected: they introduce a factor of $\log\log \numberq$ into the series \eqref{loglog}, rather than a factor of $\log \numberq$ as a naive application of the Borel--Cantelli lemma would predict.

It is natural to ask whether these extraordinary properties of the exceptional quadric hypersurface are due to special algebraic properties. This turns out to be the case; in this section we make this special structure explicit, and use this explicitness to derive elementary proofs both of \eqref{specialtypecount} and of the convergence case of Theorem \ref{theoremkhinchinquadratic} for the manifold $M_{\form_0}$.


We begin by describing the special algebraic property which leads to the results outlined above: The manifold $M_{\form_0}$ is isomorphic to $\P_\R^1\times\P_\R^1$, with the isomorphism given by the Segre embedding $\Phi:\P_\R^1\times\P_\R^1\to\P_\R^3$ defined by the formula $\Phi([\xx],[\yy]) = [\xx\otimes\yy]$, or, more explicitly,
\[
\Phi([(x_0,x_1)],[(y_0,y_1)]) = [(x_0y_0,x_0y_1,x_1y_0,x_1y_1)].
\]
Thus $M_{\form_0}$ has a ``product structure''.   This explains why the lattice $\orth(\form_0;\Z)\cap \orth(\form_0)_0$ factors as $\SL_2(\Z)\times\SL_2(\Z)$; each factor of $\SL_2(\Z)$ acts on a different copy of $\P_\R^1$. Note that the natural metric on $M_{\form_0}$ is compatible with the   distance inherited from $\P_\R^3$ under the Segre embedding.

We also remark that the product structure of $M_{\form_0}$ is consistent with its Diophantine structure. More precisely, the set of intrinsic rationals $\P_\Q^3\cap M_{\form_0}$ factors as $\P_\Q^1\times\P_\Q^1$; moreover, for $[\pp],[\qq]\in\P_\Q^1$,
\begin{equation}
\label{heightproduct}
H_\std\big(\Phi([\pp],[\qq])\big) = H_\std([\pp]) \cdot H_\std([\qq]).
\end{equation}
\begin{remark}
According to formula \eqref{heightproduct}, the Diophantine triple
\[
(\iota_3^{-1}(M_{\form_0}),\Q^3\cap \iota_3^{-1}(M_{\form_0}),H_\std)
\]
is locally isomorphic to the Diophantine triple $(\R^2,\Q^2,H_{\typewriter prod})$ considered in \cite{FishmanSimmons3}. For example, applying the affine corollary of Theorem \ref{theoremdirichletquadratic} to the hypersurface $M_{\form_0}$ yields an alternate proof of the case $\Theta = {\typewriter prod}$, $d = 2$ of \cite[Theorem 1.2]{FishmanSimmons3}.
\end{remark}

We are now ready to begin proving statements about the manifold $M_{\form_0}$ by using the decomposition $M_{\form_0} \equiv \P_\R^1\times\P_\R^1$. We begin by computing the number of rationals up to a given height:

\begin{proof}[An elementary proof of \eqref{specialtypecount}]
It is well-known that
\begin{equation}
\label{PQ1asymp}
\#\{[\pp]\in\P_\Q^1 : \numberq/2 < H_\std([\pp]) \leq \numberq\} \asymp_\times \#\{\pp\in\P_\Q^1 : H_\std([\pp]) \leq \numberq\} \asymp_\times \numberq^2.
\end{equation}
Now by \eqref{heightproduct},
\begin{align*}
\noreason N_{M_\form}(2^N) &=_\pt \#\{([\pp],[\qq])\in (\P_\Q^1)^2: H_\std([\pp])\cdot H_\std([\qq]) \leq 2^N\} \noreason \\
&=_\pt \sum_{n = 0}^N \#\left\{([\pp],[\qq])\in (\P_\Q^1)^2: 2^{n - 1} < H_\std([\pp])\leq 2^n, H_\std([\qq]) \leq \frac{2^N}{H_\std([\pp])}\right\} \noreason\\
&\asymp_\times \sum_{n = 0}^N \sum_{\substack{[\pp]\in \P_\Q^1 \\ 2^{n - 1} < H_\std([\pp]) \leq 2^n}} \left(\frac{2^N}{H_\std([\pp])}\right)^2 \by{\eqref{PQ1asymp}}\\
&\asymp_\times \sum_{n = 0}^N (2^{N - n})^2\#\{[\pp]\in\P_\Q^1 : 2^{n - 1} < H_\std([\pp]) \leq 2^n\} \noreason\\
&\asymp_\times \sum_{n = 0}^N (2^N)^2 \by{\eqref{PQ1asymp}}\\
&=_\pt (2^N)^2 (N + 1) \asymp_\times (2^N)^2 \log(2^N), \noreason
\end{align*}
demonstrating \eqref{specialtypecount} in the case $\numberq\in 2^\N$. The general case follows from a standard approximation argument.
\end{proof}
Next, we give an elementary proof of the convergence case of Theorem \ref{theoremkhinchinquadratic} for the manifold $M_{\form_0}$. This proof will give insight as to why in this case Theorem \ref{theoremkhinchinquadratic} does not simply state the converse of the (naive) Borel--Cantelli lemma; cf.\ Remark \ref{remarknontrivialrelation}.
\begin{remark}
In the following proof, we will assume that $\psi$ is regular, but we do not need to assume that $q\mapsto q\psi(q)$ is nonincreasing, as was assumed in the proof of Theorem \ref{theoremkhinchinquadratic}.
\end{remark}

\begin{proof}[Proof of the convergence case of Theorem \ref{theoremkhinchinquadratic} assuming $\form = \form_0$]
Let $\lambda$ denote normalized Lebesgue measure on $\P_\R^1$, and note that
\[
\lambda_{M_\form} \asymp_\times \Phi(\lambda\times\lambda).
\]
Let
\[
\A_\psi = \left\{([\xx],[\yy])\in (\P_\R^1)^2 :
\begin{aligned}
\text{there exist infinitely many $([\pp],[\qq])\in (\P_\Q^1)^2$ such that}\\
\dist([\pp],[\xx]),\dist([\qq],[\yy]) \leq \psi\left(H_\std([\pp])\cdot H_\std([\qq])\right)
\end{aligned}\right\}.
\]
Then $\A_{M_{\form_0}}(\psi) = \Phi(\A_\psi)$. So to prove the convergence case of Theorem \ref{theoremkhinchinquadratic}, we should show that $\lambda\times\lambda(\A_\psi) = 0$, assuming that the series
\begin{equation}
\label{loglog3}
\sum_{\numberq\in 2^\N} \numberq^2 \log\log \numberq \psi^2(\numberq)
\end{equation}
converges.

For each $n\geq 0$, let
\[
\ZZ_n = \left\{[\pp]\in\P_\Q^1: 2^n \leq H_\std([\pp]) < 2^{n + 1} \right\}.
\]
By \eqref{PQ1asymp}, we have
\[
\#(\ZZ_n) \asymp_\times (2^n)^2.
\]
Now fix $0\leq n\leq N$, and let
\[
A_{n,N} = B\big(\ZZ_n,C\psi(2^N)\big)\times B\big(\ZZ_{N - n},C\psi(2^N)\big),
\]
where $C > 0$ is a large constant. Since $\psi$ is regular (as assumed in Theorem \ref{theoremkhinchinquadratic}), 
if $C$ is large enough then
\[
\A_\psi \subset \limsup_{N\to\infty}\bigcup_{0\leq n\leq N}A_{n,N},
\]
and so by the Borel--Cantelli lemma, if the series
\begin{equation}
\label{AnNseries}
\sum_{N = 0}^\infty \sum_{n = 0}^N (\lambda\times\lambda)(A_{n,N})
\end{equation}
converges, then $(\lambda\times\lambda)(\A_\psi) = 0$. So to complete the proof, it suffices to show that $\eqref{AnNseries} \lesssim_\times \eqref{loglog3}$.

Fix $0\leq n\leq N$. We have
\[
(\lambda\times\lambda)(A_{n,N}) = \lambda\left(B\big(\ZZ_n,\psi(2^N)\big)\right)  \cdot \lambda\left(B\big(\ZZ_{N - n},\psi(2^N)\big)\right).
\]
Since $\lambda\big(B([\xx],\rho)\big) \asymp_\times r$ for all $[\xx]\in\P_\R^1$ and $0 < \rho \leq 1$, subadditivity gives
\begin{equation}
\label{Ahlforsbound}
\lambda\left(B\big(\ZZ_n,\psi(2^N)\big)\right) \lesssim_\times \#(\ZZ_n) \psi(2^N).
\end{equation}
However, in some cases it may be better to simply estimate from above by $\lambda(\P_\R^1) = 1$:
\begin{equation}
\label{stupidbound}
\lambda\left(B\big(\ZZ_n,\psi(2^N)\big)\right)  \leq 1.
\end{equation}
Similar bounds hold for $\lambda\left(B\big(\ZZ_{N - n},\psi(2^N)\big)\right)$. Thus
\begin{equation}
\label{AnNcalculation}
\begin{split}
(\lambda\times\lambda)(A_{n,N}) &\lesssim_\times \min\left(1,\#(\ZZ_n)\psi(2^N)\right)\min\left(1,\#(\ZZ_{N - n})\psi(2^N)\right)\\
&\asymp_\times \min\left(1,(2^n)^2\psi(2^N)\right)\min\left(1,(2^{N - n})^2\psi(2^N)\right)\\
&=_\pt \begin{cases}
(2^n)^2\psi(2^N) & n \leq N + \log_2\sqrt{\psi(2^N)}\\
(2^{N - n})^2\psi(2^N) & n \geq -\log_2\sqrt{\psi(2^N)}\\
(2^N)^2\psi^2(2^N) & \text{otherwise}
\end{cases}.
\end{split}
\end{equation}
The case $N + \log_2\sqrt{\psi(2^N)} \ge n \ge -\log_2\sqrt{\psi(2^N)}$ cannot occur (for all but finitely many $N$) since $\psi(2^N)$ is  less than $1/2^N$ for all sufficiently large $N$ (otherwise the series \eqref{loglog3} would diverge). 

Note that geometrically, the first two cases correspond to the bounds on $(\lambda\times\lambda)(A_{n,N})$ which result from covering $A_{n,N}$ by vertical and horizontal rectangles, respectively, while the third case corresponds to covering $A_{n,N}$ by squares.

Now fix $N$ and vary $0\leq n\leq N$. We have
\begin{align*}
\sum_{n = 0}^N (\lambda\times\lambda)(A_{n,N})
&\asymp_\times \sum_{n = 0}^{\lfloor N/2\rfloor}(\lambda\times\lambda)(A_{n,N}) \by{symmetry}\\
&\lesssim_\times \sum_{n = 0}^{\lfloor N + \log_2\sqrt{\psi(2^N)}\rfloor} (2^n)^2 \psi(2^N)
\hspace{.11 in} + \sum_{n = \lfloor N + \log_2\sqrt{\psi(2^N)}\rfloor + 1}^{\lfloor N/2\rfloor}(2^N)^2\psi^2(2^N) \noreason\\
&\asymp_\times \left(2^{N + \log_2\sqrt{\psi(2^N)}}\right)^2\psi(2^N)
\hspace{.18 in} + (2^N)^2\psi^2(2^N)\left(\frac{N}{2} - \big(N + \log_2\sqrt{\psi(2^N)}\big)\right)\noreason\qquad\qquad\\
&=_\pt (2^N)^2\psi^2(2^N) + (2^N)^2\psi^2(2^N)\frac{1}{2}\log_2\left(\frac{1}{2^N\psi(2^N)}\right)\noreason\\
&\asymp_\times (2^N)^2\psi^2(2^N)\log\left(\frac{1}{2^N\psi(2^N)}\right).
\end{align*}
Thus, for any function $\psi$ satisfying
\begin{equation}
\label{logequalsloglog}
\log\left(\frac{1}{q\psi(q)}\right) \lesssim_\times \log\log q\,,
\end{equation}
we have $\eqref{AnNseries}\lesssim_\times\eqref{loglog3}$, and thus the conclusion of Theorem \ref{theoremkhinchinquadratic} holds in the convergence case for such $\psi$.

To complete the proof, fix $\epsilon > 0$ and let
\[
\psi_\ast(q) = \frac{1}{q\log^{1/2 + \epsilon}q}\cdot
\]
Then $\psi_\ast$ satisfies \eqref{logequalsloglog}; moreover, \eqref{loglog3} converges at $\psi = \psi_\ast$. Given any function $\psi$, let
\[
\psi' = \max(\psi_\ast,\psi).
\]
Then if \eqref{loglog3} converges at $\psi$, it also converges at $\psi'$. Moreover, $\psi'$ satisfies \eqref{logequalsloglog}, so if \eqref{loglog3} converges at $\psi$, then $\A_{\psi'}$ is a nullset. But since $\psi'\geq \psi$, we have $\A_\psi \subset\A_{\psi'}$, so this completes the proof.
\end{proof}

\begin{remark}
\label{remarknontrivialrelation}
There are two important points to be made about the above proof. The first point is that the calculation \eqref{AnNcalculation} indicates what the nontrivial relation is which causes the series \eqref{loglog3} to differ from \eqref{convergence3}. Indeed, \eqref{AnNcalculation} shows that if $n\leq N + \log_2\sqrt{\psi(2^N)}$ or $n\geq -\log_2\sqrt{\psi(2^N)}$, then we are better off computing $(\lambda\times\lambda)(A_{n,N})$ not by simply adding the measures of the squares $$B\big(\cdot,C\psi(2^N)\big)\times B\big(\cdot,C\psi(2^N)\big)$$ which define $A_{n,N}$, but by estimating the measure of $A_{n,N}$ in terms of the rectangles $$B\big(\cdot,\psi(2^N)\big)\times\P_\R^1\quad\text{or}\quad\P_\R^1\times B\big(\cdot,\psi(2^N)\big)$$ respectively. Inside each rectangle there are many overlapping squares, and this overlap is what causes the difference in the series.

The second point is that we should not expect there to be a difference in series for the Jarn\'ik--Besicovitch theorem if $s < {d-1}$. Indeed, the same argument would work up until the point where the inequality \eqref{logequalsloglog} is required. But when $s < {d-1}$, then the $\psi$ which we ``expect to see'' (i.e.\ those which are near the boundary of convergence/divergence) will satisfy
\[
\log\left(\frac{1}{q\psi(q)}\right) \asymp_\times \log q 
\]
rather than \eqref{logequalsloglog}. Thus the ``refined argument'' for the convergence case produces in this case the same series \eqref{convergence3}.
\end{remark}

\ignore{\appendix
\section{Comments on Theorem \ref{theoremheathbrown}}
\label{appendixliberties}

In this appendix we clarify the relation between the paraphrased version of \cite[Theorems 5, 6, 7, and 8]{Heath-Brown_quadratic} which appears in the introduction (namely Theorem \ref{theoremheathbrown}) with the original theorems as they appear in \cite{Heath-Brown_quadratic}. We make the following comments:

\begin{itemize}
\item[1.] The relation between counting primitive vectors and counting all lattice vectors is clarified on \cite[p.12]{Heath-Brown_quadratic}.
\item[2.] In \cite{Heath-Brown_quadratic}, the function $w$ is required to be $\CC^\infty$, but to deduce Theorem \ref{theoremheathbrown}, one must let $w = \one_{B(\0,1)}$. Since $w_0 = \one_{B(\0,1)}$ can be approximated from above and below by $\CC^\infty$ functions $w_n$ in a way such that $\sigma_\infty(F,w_n)\to \sigma_\infty(F,w_0) \in (0,\infty)$, \cite[Theorems 5, 6, 7, and 8]{Heath-Brown_quadratic} will hold only without an estimate on the error term; namely, we have
\[
\lim_{P\to\infty}\frac{N(F,w_0)}{\text{leading term}} = 1
\]
for each result in \cite{Heath-Brown_quadratic}. In Theorem \ref{theoremheathbrown} we have stated only the weaker conclusion that the left hand side is bounded from above and below (in limsup and liminf respectively).
\item[3.] In \cite{Heath-Brown_quadratic}, it is shown that the modified singular series $\sigma^*$ is positive and finite if and only if the equation $\form = 0$ has nontrivial solutions in every $p$-adic field. Since the forms we deal with satisfy $\P_\Q^d\cap M_\form\neq\emptyset$, the equation $\form = 0$ has nontrivial solutions over $\Q$, and so certainly over every $p$-adic field.
\item[4.] The forms considered in \cite[Theorem 7]{Heath-Brown_quadratic} are precisely the forms $\form$ for which $\form\sim\form_0$, as we show now:
\end{itemize}

\begin{proposition}
\label{propositiondetsquare}
Let $\form:\R^4\to\R$ be a quadratic form with rational coefficients, and suppose that $\P_\Q^3\cap M_\form\neq\emptyset$. Then $\form\sim\form_0$ (i.e.\ $\Qrank = 2$) if and only if the determinant of $\form$ is the square of a rational number.
\end{proposition}
\begin{proof}
Recall that the \emph{determinant} of a quadratic form $\form:\R^4\to\R$ is the determinant of the linear map $\phi_\form:\R^4\to (\R^4)^*\equiv\R^4$ defined by $\xx\mapsto B_\form(\xx,\cdot)$. It is not invariant under the action of $\GL_4(\R)$ but rather possesses the property that $\det(\form\circ\phi) = \det(\form)\det(\phi)^2$. In particular, if $\form_1$ and $\form_2$ are equivalent over $\Q$, then $\det(\form_1)$ is a square if and only if $\det(\form_2)$ is. Thus the forward direction follows immediately upon calculating that $\det(\form_0) = 1/16$.

For the reverse direction, suppose that $\det(\form)$ is a square. By Proposition \ref{propositionrenormalization}, we may without loss of generality assume that $\form$ is $1$-normalized. In this case, we have $\det(\form) = -(1/4)\det(\w\form)$ where $\w\form$ is the remainder of $\form$. By the well-known canonical form of quadratic forms, we may without loss of generality assume that $\w\form(\xx) = a_1 x_1^2 + a_2 x_2^2$ for some $a_1,a_2\in\Q$. Then $-\det(\w\form) = -a_1 a_2$ is a square. Thus $\bb := (0,a_2,\sqrt{-a_1 a_2},0)\in\Q^4$, and $\R\ee_0 + \R\bb$ is a totally isotropic subspace of dimension $2$, proving that $\Qrank = 2$. Proposition \ref{propositionrenormalization} gives $\form\sim\form_0$.
\end{proof}}

\bibliographystyle{amsplain}

\bibliography{bibliography}

\end{document}